\documentclass{amsart}
\usepackage{amscd,amssymb,amsmath,amsthm}
\usepackage[dvips]{graphicx}
\usepackage[all]{xy}
\theoremstyle{plain}
\newtheorem{proposition}{Proposition}
\newtheorem{theorem}[proposition]{Theorem}

\newtheorem*{conjecture*}{Conjecture}
\newtheorem{definition}[proposition]{Definition}
\newtheorem{corollary}[proposition]{Corollary}
\newtheorem{lemma}[proposition]{Lemma}
\newtheorem{remark}[proposition]{Remark}

\newtheorem{example}[proposition]{Example}

\newtheorem{proposition-definition}[proposition]{Proposition/Definition}
\newtheorem*{proposition*}{Proposition}
\newtheorem*{theorem*}{Theorem}
\newtheorem*{maintheorem*}{Main Theorem}
\newtheorem*{maincorollary*}{Main Corollary}
\newtheorem*{corollary*}{Corollary}
\newtheorem*{lemma*}{Lemma}
\newtheorem*{remark*}{Remark}
\newtheorem*{definition*}{Definition}
\newtheorem*{example*}{Example}
\newtheorem*{examples*}{Examples}

\def\co{\colon\thinspace}

\newcommand{\Z}{\mathbb{Z}}
\newcommand{\Q}{\mathbb{Q}}
\newcommand{\R}{\mathbb{R}}
\newcommand{\C}{\mathbb{C}}
\newcommand{\K}{\mathbb{K}}
\newcommand{\CP}{\C\mathbb{P}}

\newcommand{\coker}{\mathrm{coker}\,}

\renewcommand{\P}{\mathbb{P}}

\begin{document}

\title[Floer theory for negative line bundles]{Floer theory for negative line bundles via Gromov-Witten invariants}

\author{Alexander F. Ritter}

\address{Mathematical Institute, University of Oxford, Oxford OX2 6GG, England.\newline \indent \emph{Former address where research was carried out:}\newline \indent Trinity College, University of Cambridge, CB2 1TQ, England.
}

\email{ritter@maths.ox.ac.uk}

\date{version: \today}


\begin{abstract}
 We prove that the GW theory of negative line bundles $M=\mathrm{Tot}(L\to B)$ determines the symplectic cohomology: indeed $SH^*(M)$ is the quotient of $QH^*(M)$ by the kernel of a power of quantum cup product by $c_1(L)$. We prove this also for negative vector bundles and the top Chern class. 

We calculate $SH^*$ and $QH^*$ for $\mathcal{O}(-n) \to \CP^m$. For example: for $\mathcal{O}(-1)$, $M$ is the blow-up of $\C^{m+1}$ at the origin and $SH^*(M)$ has rank $m$.

We prove Kodaira vanishing: for very negative $L$, $SH^*=0$; and Serre vanishing: if we twist a complex vector bundle by a large power of $L$, $SH^*=0$.

Observe $SH^*(M)=0$ iff $c_1(L)$ is nilpotent in $QH^*(M)$. This 
implies Oancea's result: $\omega_B(\pi_2(B))=0\Rightarrow SH^*(M)=0$.

We prove the Weinstein conjecture for any contact hypersurface surrounding the zero section of a negative line bundle.

For symplectic manifolds $X$ conical at infinity, we build a homomorphism from $\pi_1(\mathrm{Ham}_{\ell}(X,\omega))$ to invertibles in $SH^*(X,\omega)$. This is similar to Seidel's representation for closed $X$, except now they are not invertibles in $QH^*(X,\omega)$.
\end{abstract}
\maketitle
%
\section{Introduction}
%
\subsection{Gromov-Witten invariants versus Floer cohomology}
\strut\\ \indent
The focus of this paper will be symplectic invariants of the total space 
$$\boxed{M=\mathrm{Tot}(\pi_M:L \to B)}$$
 of negative (complex) line bundles $L\to B$ over closed symplectic manifolds $(B,\omega_B)$, although we will show that our techniques work more generally for open symplectic manifolds $M$ conical at infinity which admit Hamiltonian circle actions, and also for negative vector bundles $E \to B$ (these are not conical at infinity).

By negative line bundle $L \to B$ we mean that $\boxed{c_1(L)=-n[\omega_B]}$ for real $n>0$.
%
%
%
%

\noindent \emph{\textbf{Examples:} $\mathcal{O}(-n) \to \P^m$ (classifies negative holomorphic line bundles over $\P^m$ for $n\in \Z$). Duals of ample holomorphic line bundles over compact complex manifolds.}

 Negativity ensures there is a natural symplectic form $\omega$ on $M$ making $M$ conical at infinity (a convexity condition) with $n[\omega_B] \mapsto [\omega]$ via $\pi_M^*\co H^2(B)\cong H^2(M)$. With this symplectic form, the base and the fibres are symplectic submanifolds.

The invariants we will be concerned with are the genus zero Gromov-Witten invariants involved in the construction of the \emph{quantum cohomology} of $M$, and the Floer invariants involved in the construction of \emph{symplectic cohomology} (the natural generalization of Floer homology to open symplectic manifolds $M$ which are conical at infinity, constructed in the exact setup by Viterbo \cite{Viterbo1} (although there were similar previous incarnations), and in the non-exact setup by the author \cite{Ritter2}).

\begin{remark*}[Mirror symmetry]
 Symplectic cohomology plays an important role in mirror symmetry for non-compact manifolds. At the cohomology level, the morphism spaces of the wrapped Fukaya category are modules over  $SH^*(M)$. In many examples, $SH^*(M)$ is the Hochschild homology of the wrapped Fukaya category, and therefore it recovers the Hochschild homology of the derived category of coherent sheaves of the mirror. For a discussion of this, we refer the reader to Ritter-Smith \cite{Ritter-Smith} and the references contained therein. One of the applications in \cite{Ritter-Smith} is the computation of the wrapped Fukaya category of the negative line bundles $\mathcal{O}(-n)\to \P^m$, and the key ingredient was the computation of $SH^*(\mathcal{O}_{\P^m}(-n))$ done in this paper.
\end{remark*}

Gromov-Witten invariants are in principle understood for most closed symplectic manifolds, and often they are explicitly calculable thanks to algebraic geometry. We suggest Ruan-Tian \cite{Ruan-Tian} and McDuff-Salamon \cite{McDuff-Salamon2} as references. We will be concerned with genus zero GW invariants of the (non-closed) $M$ and of certain Hamiltonian fibrations over $\P^1$ with fibre $M$. The first arise in algebraic geometry as the \emph{twisted Gromov-Witten invariants} of $(B,L)$ and were studied by Coates-Givental \cite{Coates-Givental} and Lee \cite{Lee}: essentially the GW theory of $M$ is determined by the GW theory of $B$ and the invariant $c_1(L)$. The second are known for closed symplectic manifolds by the work of Seidel \cite{Seidel3}, and we succeeded in generalizing these to the open setup despite the difficulties caused by the non-compactness.

Floer invariants, on the other hand,  are notoriously difficult to calculate explicitly because the chain differential comes from counting solutions of certain elliptic partial differential equations which require a \emph{generic} choice of $\omega$-compatible almost complex structure $J$ on $M$. In practice, this means that one always has to perturb a given $J$, so one cannot compute anything unless things vanish for grading reasons.

For symplectic cohomology, the difficulty of computing the invariants is even more dramatic, because they arise as a direct limit of Floer cohomologies:
$$
SH^*(X,\omega_X) = \varinjlim HF^*(H,\omega_X)
$$
involving Hamiltonians $H:X \to \R$ which are ``linear'' at infinity, and the connecting maps $HF^*(H_1,\omega_X) \to HF^*(H_2,\omega_X)$ are Floer continuation maps which increase the slope at infinity. These continuation maps, again solutions of an elliptic PDE, can almost never be computed explicitly for the same reasoning.

This phenomenon is apparent in the literature, where known computations involve showing vanishing results by indirect grading/action tricks (for example, for $X=\C^m$ and generally $X=$ subcritical Stein manifold \cite{Cieliebak}). For this reason, a precious guide to proving non-vanishing of $SH^*(X)$ \emph{a posteriori} has been by detecting submanifolds which obstruct vanishing (for example, when $\omega_X=d\theta$ is exact, and $X$ contains an exact Lagrangian submanifold \cite{Viterbo1}). Other attempts involve proving that $SH^*(X)$ reduces to a topological invariant by continuation arguments, again not explicitly computable (for instance, various versions of Viterbo's result \cite{Viterbo1} that $SH^*(T^*B,d\theta;\Z/2)\cong H_{\dim_{\C}(B)-*}(\mathcal{L} B;\Z/2)$, where $\mathcal{L} B$ is the free loop space).

It comes therefore as a surprise that for $M = \mathrm{Tot}(L\to B)$ we will calculate the Floer cohomologies and the continuation maps explicitly and directly, by transforming the Floer theoretic problem into an essentially algebraic-geometric problem in Gromov-Witten theory. It is also surprising that we will explicitly recover the ring structure on symplectic cohomology. Finally we emphasize that the setup we are in is very novel for symplectic cohomology literature: we are in a highly non-exact setup (the zero section is a symplectic submanifold and holomorphic spheres play a crucial role, unlike the much studied setup of exact cotangent bundles) and for $\dim_{\C} B>1$ we are dealing with manifolds which do not admit a Stein structure.

\begin{theorem}\label{Theorem Intro r_g on QH and SH}
 Let $M$ be the total space of a negative line bundle $L \to (B,\omega_B)$ $($satisfying a weak$^+$ monotonicity condition$)$. Then for $k\geq \dim H^*(B)$,
$$
\boxed{SH^*(M) \cong QH^*(M)/\ker r^{k}}
$$
is an isomorphism of $\Lambda$-algebras, where $r: QH^*(M) \to QH^{*+2}(M)$ is the $($non-invertible$)$ $\Lambda$-module endomorphism given by quantum cup product by the first Chern class $r(1)=\pi_M^*c_1(L)\in QH^2(M)$.
Thus $SH^*(M)$ is the quantum cohomology quotiented by the generalized $0$-eigenspace of the action of $\pi_M^*c_1(L)$.
\end{theorem}

The isomorphism is induced by $c^*:QH^*(M) \to SH^*(M)$: 
a canonical algebra homomorphism identifiable with $r^k$.
The induced action of $r$ on $SH^*(M)$ is 
$$\boxed{\mathcal{R}: SH^*(M) \to SH^{*+2}(M)}$$
 a degree $2$ $\Lambda$-module automorphism which on the chain level is a natural rotation action $\mathcal{S}:HF^*(H,J)\to HF^{*+2}(g^*H,g^*J)$ determined by the loop $g=(e^{2\pi i t})_{t\in S^1}$ of Hamiltonian diffeomorphisms which rotate the fibres of $L$.

\begin{corollary*}
 $c^*$ is never an isomorphism. So any contact hypersurface in $M$ surrounding the zero section contains a closed Reeb orbit (Weinstein conjecture).
\end{corollary*}
\begin{proof} $1\notin\mathrm{im}\, r$ as it would involve a GW invariant with an evaluation condition with the point class, which can be moved to infinity. So $\ker r \neq 0$, so $\ker c^*\neq 0$. The rest is a standard consequence of the construction and invariance of $SH^*$.
\end{proof}

\begin{remark*}
 I thank the anonymous referee for pointing out that an alternative approach to prove the Weinstein Conjecture for negative line bundles $L$ would be to first compactify them by $\mathbb{P}(L\oplus \C)$, then to look at the Gromov-Witten invariants of the fibre, and finally to apply a neck-stretching argument.
\end{remark*}

\begin{corollary}\label{Corollary SH is zero if rg is nilpotent}
 $SH^*(M)=0 \Leftrightarrow \pi_M^*c_1(L)$ is nilpotent in $QH^*(M)$.
 In particular, if the quantum cup product reduces to the ordinary cup product, then $SH^*(M)=0$.
\end{corollary}

In Section \ref{Section Negative vector bundles} we generalize the above Theorem to \emph{negative vector bundles} $E\to B$. Definition \ref{Definition negative vector bundle} will explain the precise meaning of \emph{negative} in this context, which is a negative curvature condition on some Hermitian connection for $E$.

\begin{theorem*} For negative vector bundles $E\to B$, the analogue of the Theorem holds:
$r$ is a degree $2\,\mathrm{rank}_{\C} E$ endomorphism given by quantum cup product by the top Chern class $r(1)=\pi_M^*c_{\mathrm{rank}_{\C}}(E)$, and $\mathcal{R}$ is a degree $2\,\mathrm{rank}_{\C} E$ $\Lambda$-module automorphism on $SH^*(\mathrm{Tot}(E))$. In particular if $\mathrm{rank}_{\C} E>\dim_{\C}B$ then $SH^*(\mathrm{Tot}(E))=0$. 
\end{theorem*}

\subsection{How $\mathbf{r}$ arises algebro-geometrically and Floer theoretically}
\strut\\ \indent
Algebro-geometrically the map $r$ arises from $2$-pointed genus $0$ Gromov-Witten invariants counting sections of the Hamiltonian fibration $E_g\to \P^1$ with fibre $M$, constructed from the loop of rotations $g_t=e^{2\pi i t}$ by the clutching construction. Heuristically $r$ is the pull-push map
$$
H^*(M) \to H^{*+2}(M), \;\; a \mapsto \sum_{\beta\in H_2(M)} (\mathrm{ev}_{z_{\infty}})_{!}\left(\mathrm{ev}_{z_0}^*(a) \wedge e(\mathrm{Obs}_{\beta})\right)
$$
where $\mathrm{Obs}_{\beta}$ is the obstruction bundle over the moduli space 
$$\mathcal{M}_{\beta}=\overline{\mathcal{M}}_{0,2}(E_g, [\P^1]+(j_{z_0})_*\beta)$$
 of stable maps $u$ from $2$-pointed genus $0$ nodal curves to $E_g$ representing the class $[\P^1]+(j_{z_0})_*\beta$ where $[\P^1]\in H_2(E_g)$ is the base of $E_g$ and $j_{z_0}$ is inclusion of the fibre at the South Pole $z_0\in \P^1$. Composing with the projection $\pi_g:E_g\to \P^1$, the main component of $u$ yields an isomorphism to $\P^1$. So $u$ can be viewed as a holomorphic section of $E_g$ possibly with holomorphic bubbles in the fibres (killing the $PSL(2,\C)$ reparametrization freedom by making it a section). The two maps $\mathrm{ev}: \mathcal{M}_{\beta} \to M$ are evaluation of sections of $E_g$ at the two Poles $z_0,z_{\infty}\in \P^1$.

More precisely, $r$ is a Novikov-weighted count of the zero dimensional moduli spaces of pseudo-holomorphic sections of $E_g\to \P^1$ which intersect a given locally finite quantum cycle in the fibre over $z_0$ and a given quantum cycle over $z_{\infty}$.

Floer theoretically, the map $r$ is the composite
$$
\xymatrix{
HF^*(H_0,J,\omega) 
\ar@{->}^-{\mathcal{S}}[r]
\ar@{->}@/_1pc/_-{\mathcal{R}}[rr]
&
HF^{*+2}(g^*H_0,g^*J,\omega)   
 \ar@{->}^-{\varphi_0}[r] 
&
HF^{*+2}(H_0,J,\omega)  \ar@{->}^-{\psi^+}[d] 
\\
QH^*(M,\omega) \ar@{->}^-{\psi^-}[u] \ar@{-->}^-{r}[rr] & & QH^{*+2}(M,\omega)
}
$$
Here $H_0:M\to \R$ is a Hamiltonian of ``slope zero'' at infinity (more precisely: whose positive slope decays to $0$ at infinity). Such a Hamiltonian $H_0$ gives rise to identifications $\psi^{\pm}$ between the Floer complexes and the (quantum) Morse chain complexes. The above map $\mathcal{S}$ is the natural isomorphism at the chain level induced by identifying the relevant Floer moduli spaces by pulling back the data $H_0,J$ via $g$. Finally $\varphi_0$ is a Floer continuation map obtained by homotopying the data.

\begin{theorem*}
 The algebro-geometrical and the Floer theoretical construction of $r$ agree, that is the above diagram commutes.
\end{theorem*}

This result, and Theorem \ref{Theorem Intro r_g on QH and SH}, can be heuristically viewed as a symplectic analogue of the quantum Lefschetz hyperplane theorem \cite{Lee}: the invariants of the hyperplane section $B\subset M$ are recovered from invariants of the ambient $M$ and a quantum multiplication operation by an Euler class.

The difficulty in relating the two constructions of $r$ (compared with a similar setup in the closed case due to Seidel \cite{Seidel3}) involves the fact that we are using \emph{non-compact} Hamiltonian fibrations and \emph{non-monotone} homotopies (arising in $\psi^+$).
%
\subsection{The role of $\mathbf{r}$ in determining $\mathbf{SH^*(M)}$}
\label{Subsection The role of rg in determining SH}
\strut\\ \indent
The symplectic cohomology of $M$ arises as a direct limit of Floer cohomologies
$$
\xymatrix{
QH^*(M) \ar@<2pt>[r]^-{\psi^-} \ar@{<-}@<-2pt>
[r]_-{\psi^+} & HF^*(H_0) \ar[r]^-{\varphi_1} & HF^*(H_1) \ar[r]^-{\varphi_2} &  HF^*(H_2) \ar[r]^-{\varphi_3} &  \cdots
}
$$
where $H_i$ are carefully chosen Hamiltonians with slope proportional to $i$, the $\varphi_i$ are continuation maps. The direct limit of the composition of those maps defines $c^*: QH^*(M) \to SH^*(M)$. We prove in \ref{Subsection relation to quantum homology} that the $\psi^{\pm}$ are identifications of algebras and that $c^*$ is a $\Lambda$-algebra homomorphism (using a Novikov ring $\Lambda$).

After suitable identifications, we prove in \ref{Subsection The Floer construction} that the above sequence becomes:
$$
V = V \stackrel{\varphi}{\to} V \stackrel{\varphi}{\to} V \stackrel{\varphi}{\to} \cdots
$$
where $V=QH^*(M)$ and $\varphi$ is quantum cup product by $r(1)$. This involves a special choice of $H_i$: recall $H=m\pi|z|^2$ on $\C$ for $m\notin \Z$ only has Hamiltonian $1$-orbit $0$, and in our case the cohomology of the zero section plays the role of this $1$-orbit $0$.

By linear algebra, $\varphi^k(V)$ stabilizes for $k\geq \mathrm{rank}_{\Lambda}\, QH^*(M)$ and
$
\varphi^k(V) \cong V/\ker \varphi^{k}.
$
Say it stabilizes at stage $k$. Then $\varphi$ is an automorphism on $\varphi^k(V)$. In the direct limit, we identify $v\sim \varphi(v)$, so $SH^*(M)$ can be identified as a $\Lambda$-vector space to $\varphi^k(V)\subset HF^*(H_k)$, and $\varphi^k:V \to \varphi^k(V)$ can be identified with $c^*$.

Thus $c^*$ is surjective and $\ker c^* = \ker r^k$. 
Since $c^*$ is an algebra homomorphism, it induces the quotient isomorphism of $\Lambda$-algebras
$$SH^*(M)\cong QH^*(M)/\ker c^* = QH^*(M)/\ker r^k,$$
proving Theorem \ref{Theorem Intro r_g on QH and SH}.
The product structure is discussed in more detail in \ref{Subsection Product structure on SH using rg}.

The reason the sequence simplifies so dramatically, is that conjugation by the rotation $\mathcal{S}:HF^*(H_i)\to HF^{*+2}(H_{i-1})$ recovers all $\varphi_i$ from $\varphi_0$: $\varphi_{i}=\mathcal{S}^{-i}\varphi_0 \mathcal{S}^{i}$. So $SH^*(M)$ is determined via linear algebra by a map
$$
QH^*(M,M\setminus B) \to QH^*(M)
$$
corresponding to $HF^*(-H_0)\to HF^*(H_0)$ (identifiable with $\varphi_0$). Up to first approximation, this map is the natural map for the pair $(M,M\setminus B)$, which in the Gysin sequence for the sphere bundle of $L$ corresponds to ordinary cup product $H^*(B)\to H^{*+2}(B)$ by $c_1(L)$. The surprising result is that there are quantum correction terms in the Floer continuation map, and this first approximation equals the continuation map of Morse cohomologies (the Floer complexes for small $\pm H_0$ reduce to Morse complexes). This is unlike the exact setup \cite{Ritter3} or the setup  $\omega_B(\pi_2(B))=0$, in which by arguments \`{a} la Salamon-Zehnder \cite{Salamon-Zehnder} for a homotopy of $C^2$-small time-independent Morse Hamiltonians the Floer continuation map  reduces to the Morse continuation map (solutions become time-independent).
%
\subsection{Non-vanishing of symplectic cohomology of the blowup of $\mathbf{\C^{m+1}}$}
%
\begin{corollary}\label{Corollary Intro O-1 case}
For $\mathcal{O}(-1) \to \P^m$, $SH^*(M)$ has rank $m$. Indeed as $\Lambda$-algebras: 
$$
\begin{array}{lll}
QH^*(M) = \Lambda[\omega_Q]/(\omega^{m+1}_Q+t\cdot \omega_Q)\\
SH^*(M) \cong \Lambda[\omega_Q]/(\omega^{m}_Q+t\cdot 1)
\end{array}$$
where $\omega_Q = \pi_M^*\omega_{\P^m} \otimes 1 \in QH^2(M)$, $\omega_Q^{m}$ are quantum powers, and $\Lambda=\mathbb{K}(\!(t)\!)$ is the Novikov field (Laurent series in a formal variable $t$, one can even replace $\K$ by $\Z$).
\end{corollary}

Recall that the $M$ of the Corollary arise as the blow-up of $\C^{m+1}$ at the origin. So the symplectic cohomology has changed under blow-up as $SH^*(\C^{m+1})=0$. Interestingly the growth-rate \cite[Sec.(4a)]{Seidel} of $SH^*(M)$ is $0$ despite $SH^*(M)\ncong QH^*(M)$. When this non-isomorphism occurs, there is a non-constant Hamiltonian orbit, and one typically expects its iterates to force $\dim_{\Lambda}SH^*(M)=\infty$.

\begin{remark}[Smith]\label{Remark Ivan} Ivan Smith discovered an essential torus in $\mathrm{Tot}(\mathcal{O}(-1)\to \P^1)$ in \cite[Corollary 4.22]{Smith}: the sphere bundle lying over the equator of $\P^1$ of constant radius making the torus monotone. Essential tori are defined in \cite[Sec. (5b)]{Seidel}. In \cite[Prop.5.2]{Seidel}, Seidel-Smith state that if a $4$-dimensional $($exact$)$ Liouville domain $(M,d\theta)$ contains an essential Lagrangian torus then $SH^*(M,d\theta)\neq 0$. The proof is not written down in the literature in detail, but it is briefly sketched in Seidel \cite[Sec.(5b)]{Seidel}.
If one assumes that this result holds also for monotone essential tori in non-exact $4$-dimensional symplectic $M$ conical at infinity, then the presence of the essential torus would imply a posteriori that $SH^*(\mathrm{Tot}(\mathcal{O}(-1)\to \P^1))\neq 0$. 
\end{remark}

\subsection{Non-vanishing of symplectic cohomology of $\mathbf{M=\mathbf{Tot}(\mathcal{O}(-n)\to \P^m)}$}
The difficulty in calculating the Gromov-Witten invariants is that for the standard integrable complex structure the moduli spaces of holomorphic curves are essentially never regular. This is because the maximum principle prevents the curves from escaping the zero section, so the dimension count does not agree with the expected dimension because the moduli spaces do not ``notice'' the fibre direction. To address this issue one requires algebraic-geometric techniques. By using virtual localization techniques similar to Kontsevich \cite{Kontsevich} and Graber-Pandharipande \cite{Pandharipande2}, but adapted to the setup of holomorphic sections of $E_g$,  we determine $r$ explicitly for $\mathcal{O}(-n)\to \P^m$, when $n<1+\frac{m}{2}$, and determine enough about $r$ for $n<1+m$:

\begin{theorem}\label{Theorem Intro O-n over Pm} Let $M=\mathrm{Tot}(\mathcal{O}(-n)\to \P^m)$, $N=1+m-n$. In characteristic $0$:\\[1mm]
\renewcommand{\arraystretch}{1.25}\indent
\begin{tabular}{|l|l|}
\hline
 $1\leq n < 1+\frac{m}{2}$ & 
$\begin{array}{lll}
 QH^*(M) &=&  \Lambda[\omega_Q]/(\omega_Q^{1+m}- (-n)^{n} t\omega_Q^n) \\
 SH^*(M) &=&  \Lambda[\omega_Q]/(\omega_Q^N - (-n)^{n} t)
\end{array}
$
\\
\hline
 $1+\frac{m}{2} \leq n < 1+m$ & $SH^*(M)\neq 0$ has rank a multiple of $N$\\
\hline
 $n=1+m$ & $SH^*(M)=0$ $($borderline case: $c_1(TM)=0)$ \\
\hline
 $2+m \leq n \leq 2m$ & $M$ does not satisfy weak$^+$ monotonicity\\
\hline
$n>2m$ & 
$\begin{array}{lll}
 QH^*(M) &=&  \Lambda[\omega_Q]/(\omega_Q^{1+m}) \textrm{ is ordinary}\\
 SH^*(M) &=&  0 
\end{array}
$
\\ \hline
\end{tabular}
\renewcommand{\arraystretch}{1}
\\[1mm]
Where $\Lambda=\mathbb{K}(\!(t)\!)$ is the same as in Corollary \ref{Corollary Intro O-1 case}. Over characteristic $2$, the above holds except $SH^*(M)=0$ for even $n$.
\end{theorem}

\subsection{The aspherical case: $\mathbf{\omega_B(\pi_2(B))=0}$}
\label{Subsection Aspherical case}
\strut\\ \indent
Negative line bundles satisfying $\omega_B(\pi_2(B))=0$ have been studied by Oancea in his Ph.D. thesis (see \cite{Oancea}). This involves a difficult construction of a Leray-Serre spectral sequence for symplectic cohomology, which immediately collapses for negative line bundles since fibres have $SH^*(\C)=0$, and so $SH^*(M)=0$. Observe that Corollary \ref{Corollary SH is zero if rg is nilpotent} gives a new proof of this result: when $\omega_B(\pi_2(B))=0$ then $\omega(\pi_2(M))=0$ so quantum cup product on $M$ is ordinary cup product.

Because of this vanishing result, Oancea conjectured that vanishing should hold for any negative line bundle even without the condition $\omega_B(\pi_2(B))=0$.

Corollary \ref{Corollary Intro O-1 case} shows this conjecture is not true. It also shows there cannot be a spectral sequence $E_2^{p,q}=QH^p(B,\mathcal{SH}^q(\C))$ converging to $SH^*(M)$. So $\omega(\pi_2(M))=0$ is more than a technical assumption, which is surprising in Floer theory. 

We also have to point out that the assumption $\omega_B(\pi_2(B))=0$ is extremely restrictive: it excludes all simply connected $B$, and it excludes any complex variety $B$ which contains a holomorphic $\P^1$. However it holds for surfaces of genus $\geq 1$.

\subsection{The Calabi-Yau type case: $\mathbf{c_1(TM)(\pi_2(M))=0}$}
\label{Subsection CY type case}
\begin{theorem}\label{Theorem c1=0 implies SH=0}
 If $c_1(TM)(\pi_2(M))=0$, then $SH^*(M)=0$.
\end{theorem}
\begin{proof}[Proof 1] $\Lambda$ has grading $0$, so $(\pi_M^*c_1(L))_Q^{\dim_{\C}B+1} \in H^{2\dim_{\C}B+2}(M)\otimes \Lambda=0$.
\end{proof}
\begin{proof}[Proof 2]
 $SH^*(M)$ is $\Z$-graded. Rotation in the fibres induces $\mathcal{S}:HF^*(H_i) \stackrel{\sim}{\to} HF^{*+2}(H_{i-1})$. In the direct limit,
$\mathcal{S}: SH^*(M) \to SH^{*+2}(M)$ is an automorphism.
So $SH^*(M)$ is $2$-periodic, so it is either $0$ or $\infty$-dimensional. But $\mathrm{rank}_{\Lambda} HF^*(H_i) = \mathrm{rank}_{\Lambda} QH^*(M)$, so $\mathrm{rank}_{\Lambda} SH^*(M)\leq \mathrm{rank}_{\Lambda} QH^*(M)$, so $SH^*(M)=0$.
\end{proof}

For example, this applies to $\mathcal{O}(-(1+m))\to \P^m$. More generally, let $B$ be a \emph{Fano variety}: a closed complex manifold with ample anticanonical bundle $\mathcal{K}^{\vee}$, where $\mathcal{K}=\Lambda^{top}_{\C} T^*B$. Since $c_1(TB)=-c_1(\mathcal{K})$, and in general 
$c_1(TM)\equiv c_1(TB)+c_1(L)$ (via the identification $H^2(M)\cong H^2(B)$), we deduce:

\begin{example*}
 Let $L=$ canonical bundle $\mathcal{K} \to \textrm{Fano variety }B$. Then $SH^*(M)=0$.
\end{example*}

\begin{example*}
 Hyperk\"{a}hler ALE spaces $($minimal resolutions of simple singularities $\C^2/\Gamma)$ are not total spaces of negative line bundles $($except for $T^*S^2$ which is $\mathcal{O}(-2)\to \P^1)$, but they admit a circle action $g$ similar to rotation in the fibres for $\omega=\omega_I$ $($see \cite{Ritter2}$)$. Since $c_1(\textrm{ALE space})=0$, we deduce $SH^*(\textrm{ALE space},\omega_I)=0$.
\end{example*}

\subsection{The role of weak$^+$ monotonicity}
\label{Subsection Intro role of weak+ monotonicity}
\strut\\ \indent
%
Because of \ref{Subsection Aspherical case} and \ref{Subsection CY type case}, one is really interested in the case $\omega_B(\pi_2(B))\neq 0$ and $c_1(TM)(\pi_2(M))\neq 0$. This causes two difficulties in Floer homology: (1) the action functional which defines the chain differential becomes multivalued and bubbling phenomena can occur; (2) Floer homology is only $\Z/2N$-graded where
$$N\Z=c_1(TM)(\pi_2(M))=(c_1(TB)+c_1(L))(\pi_2(B)).$$
A standard machinery due to Hofer-Salamon \cite{Hofer-Salamon} ensures Floer homology can be defined if we assume that $M$ is \textbf{weak}, meaning \textbf{at least one of}:

\begin{enumerate}
\item $\omega(\pi_2(M))=0$ or $c_1(TM)(\pi_2(M))=0$,

 \item $M$ is monotone: $\exists\lambda > 0$ such that $\omega(A) = \lambda c_1(TM)(A)$ for all $A\in \pi_2(M)$,

\item the minimal Chern number $|N|\geq \dim_{\C} B$.
\end{enumerate}

\subsection{The rank of $\mathbf{SH^*(M)}$}
\begin{corollary}\label{Corollary Intro condition on ranks} For weak $M$,
\begin{enumerate}
 \item $\mathrm{rank}_{\Lambda}\, SH^*(M) < \mathrm{rank}\, H^*(B)$;
 \item $\mathrm{rank}_{\Lambda}\, SH^*(M)$ is a multiple of $|N|$.
 \item if $|N| \geq \mathrm{rank}\, H^*(B)$ then $SH^*(M)=0$. 
\end{enumerate}
\end{corollary}
\begin{proof}
 (1): follows by Theorem \ref{Theorem Intro r_g on QH and SH}. (2): $SH^*(M)$ is $\Z/2N$-graded, $\Lambda$ is generated by elements in degrees $\in 2N\Z$ so they preserve the grading of $SH^*$. So the automorphism $\mathcal{R}$ induces 
$
SH^0(M) \cong SH^2(M) \cong \cdots SH^{2|N|-2}(M).
$
Similarly for odd pieces. So $\mathrm{rank}_{\Lambda} SH^*(M) = |N|\cdot (d_0+d_1)$, $d_i\!=\!\mathrm{rank}_{\Lambda} SH^i(M)$. (3) follows.
\end{proof}

\begin{example*}
 $SH^*(M)=0$ for $\mathcal{O}(-n) \to \P^m$ if $n\geq 2m+2$ since $|N|\geq 1+m$.
\end{example*}

\subsection{Kodaira Vanishing for $\mathbf{SH^*(M)}$}
\label{Subsection Kodaira Vanishing Intro}
\begin{theorem}\label{Theorem Kodaira Vanishing theorem statement}
 If the line bundle $L \to B$ is sufficiently negative then quantum cup product on $M$ is ordinary cup product, so $SH^*(M)=0$ by Corollary \ref{Corollary SH is zero if rg is nilpotent}.
\end{theorem}
\begin{proof}[Proof 1]
 $\omega*\omega^j = \sum \mathrm{GW}^M_{0,3,\beta}(\mathrm{PD}(\omega),\mathrm{PD}(\omega^j),2\ell\textrm{-cycle})\; [2\ell\textrm{-form}]\otimes \beta$, summing over $\ell=1+j-c_1(TM)(\beta)$ and over appropriate forms/cycles. Now
$$
c_1(TM)(\beta)=c_1(TB)(\pi_M\circ u) + c_1(L)(\pi_M\circ u).
$$
But now observe that $\pi_M\circ u$ is holomorphic. \emph{Technical remark: to achieve regularity for $J$ one can perturb it whilst keeping it lower triangular $J=\left(\begin{smallmatrix} J_B & 0 \\ * & i \end{smallmatrix}\right)$ in the splitting of $TM=TB\oplus L$ determined by a Hermitian connection (see Section \ref{Subsection Definition and properties neg line bdle}), which ensures $\pi_M: M \to B$ is $(J,J_B)$-holomorphic. Alternatively one can work with the original (possibly non-regular) $J=J_B\oplus i$, which ensures $\pi_M$ is holomorphic, and then argue by a limiting compactness argument that only the constants contribute.}

Since $\pi_M\circ u$ is holomorphic, $c_1(L)(\pi_M\circ u) = -n\omega_B(\pi_M\circ u)<0$ unless $u$ is constant (if $\pi_M\circ u$ is constant then by the maximum principle $u$ is constant). So for non-constant $u$: $n\gg 0$ implies $c_1(TM)(\beta)\ll 0$ so $\ell\gg 0$ so there are no $2\ell$-forms. So only constants contribute, which yield ordinary cup product.
\end{proof}
\begin{proof}[Proof 2] Let $N_{\mathrm{eff}}\Z \!=\! c_1(TM)\{\textrm{pseudo-holomorphic }v:\P^1\!\to\! B\!\subset\! M \}$, and $\Lambda_{\mathrm{eff}}\subset \Lambda$ the 
subring generated by ``effective'' $\pi_2(M)$-classes i.e. arising as such $[v]$. Now $r(1)$ and quantum product involve (forms $\otimes$ $[v]$), so restrict $r\in \mathrm{End}_{\Lambda_{\mathrm{eff}}}(H^*(M)\otimes {\Lambda_{\mathrm{eff}}})$.
$N_{\mathrm{eff}}$ grows proportionally to $n$ (since $\omega_B(v)\!>\!0$ unless $v$ is constant). So for $n\gg 0$, $N_{\mathrm{eff}}\geq \max\,(\mathrm{rank} \,H^*(M),\dim_{\C}B)$, so the characteristic polynomial of $r$ yields a linear dependence among $r(1),r(1)^2,\ldots,r(1)^{|N_{\mathrm{eff}}|}$, but these lie in different degrees ($\Lambda_{\mathrm{eff}}$ is in degrees $2N_{\mathrm{eff}}\Z$), so some $r(1)^k=0$, so $SH^*(M)=0$ by Corollary \ref{Corollary SH is zero if rg is nilpotent}.
\end{proof}

\begin{examples*}
 For a K3 surface $B$, $\omega_B\in H^2(B;\Z)$, and $n \geq 24$ then $SH^*(M)=0$.\\
%
%
 For $L=\mathcal{K}^{k+1} \to \mathrm{Fano}\, B$ with $k\geq \mathrm{max}(\mathrm{rank}\, H^*(B), \dim_{\C}B)$,  $SH^*(M)=0$.
\end{examples*}

\begin{remark*}
 The vanishing of symplectic cohomology for very negative curvature is not a consequence of a Hamiltonian displaceability property \emph{(}compare \cite{Ritter3}\emph{)}. Indeed the zero section is never displaceable because $c_1(L) = -n[\omega_B]\neq 0 \in H^2(B)$.
\end{remark*}

\begin{corollary}\label{Corollary Kodaira vanishing full}
 If $E \to B$ is any line bundle, and $L \to B$ is any negative line bundle, then $M_k=\mathrm{Tot}(E \otimes L^{\otimes k} \to B)$ is weak for $k \gg 0$ and
$
SH^*(M_k) = 0$ for $k \gg 0.
$
\end{corollary}
\begin{proof}
Hermitian connections on $L,E$ induce one on $E\otimes L^{\otimes k}$ with curvature $\mathcal{F}^{E} + k \mathcal{F}^{L}$. Then $\frac{1}{2\pi i} k \mathcal{F}^{L}(v,J_B v) + \frac{1}{2\pi i} \mathcal{F}^{E}(v,J_B v)<0$ for $k\gg 0$ by making the first term dominate (see Lemma \ref{Lemma negative curvature is negativity}: we pick a suitable connection on $L$). Hence $E\otimes L^{\otimes k}$ is a negative line bundle with $n\gg 0$ if $k\gg 0$ (see the comment after Lemma \ref{Lemma negative curvature is negativity}). 
\end{proof}

\begin{remark*}
 Strictly speaking, weakness may not be satisfied by $M_k$ in case (3) of \ref{Subsection Intro role of weak+ monotonicity} if $|N|<\dim_{\C}B$. But since $|N_{\mathrm{eff}}|\geq \dim_{\C}B$, all Floer theoretic issues such as bubbling can be avoided: only effective $\pi_2(M)$-classes are involved in these issues.
\end{remark*}

\subsection{Negative vector bundles and Serre Vanishing}
Complex vector bundles $E \to B$ are negative if a suitable negative curvature condition holds (Definition \ref{Definition negative vector bundle}). The automorphism $\mathcal{R}:SH^*(M) \to SH^{*+2\mathrm{rank_{\C}E}}(M)$ implies $2\mathrm{rank_{\C}E}$-periodicity in ranks so
Corollary \ref{Corollary Intro condition on ranks} becomes: for weak $M=\mathrm{Tot}(E \to B)$,
\begin{enumerate}
 \item $\mathrm{rank}_{\Lambda}\, SH^*(M) < \mathrm{rank}\, H^*(B)$;

 \item $\mathrm{rank}_{\Lambda}\, SH^*(M)$ is a multiple of $|N|/\mathrm{rank}_{\C} E$, the multiple is
$d_0 + d_1 + \cdots + d_{2\mathrm{rank}_{\C} E-1}$
where $d_i = \mathrm{rank}_{\Lambda} SH^i(M)$.

 \item if $|N| \geq \mathrm{rank}_{\C} E\cdot \mathrm{rank}\, H^*(B)$ then $SH^*(M)=0$. 

\end{enumerate}
Theorem \ref{Theorem c1=0 implies SH=0} and Corollary \ref{Corollary SH is zero if rg is nilpotent} (for $\pi_M^*c_{\mathrm{rank}_{\C}E}(E)$) hold for the same reasons.

\begin{theorem}\label{Theorem Serre vanishing full}
Let $E\to B$ be any complex vector bundle, and $L\to B$ a negative line bundle. Then for $k\gg 0$: $M_k=\mathrm{Tot}(E\otimes L^{\otimes k}\to B)$ is a negative vector bundle and is weak. Also $SH^*(M_k)=0$ for $k\gg 0$.
\end{theorem}
\begin{proof}
The proof of Corollary \ref{Corollary Kodaira vanishing full} applies via Definition \ref{Definition negative vector bundle}. In \emph{Proof 1} of Theorem \ref{Theorem Kodaira Vanishing theorem statement} we use $c_1(TM_k) = \pi_M^*(c_1(TB) +c_1(E) + kc_1(L))$). In \emph{Proof 2}, we now need $|N_{\mathrm{eff}}|\geq \mathrm{rank}_{\C}E\cdot \mathrm{rank}\,H^*(B)$ and (for weakness) $|N_{\mathrm{eff}}| \geq \mathrm{dim}_{\C}B+\mathrm{rank}_{\C}E-1$.
\end{proof}

\subsection{The general theory: a representation of $\mathbf{\pi_1(\mathbf{Ham}_{\ell}(X,\omega))}$ on $\mathbf{SH^*(X)}$}
\label{Subsection Intro general theory}
\strut \indent Let $(X,\omega)$ be any symplectic manifold conical at infinity satisfying weak$^+$ monotonicity $($\ref{Subsection Conical symplectic manifolds}$)$. So $X$ has the form $\Sigma \times [1,\infty)$ at infinity, with coordinate $R\in [1,\infty)$.

Denote $\mathrm{Ham}_{\ell}(X,\omega)$ the Hamiltonian diffeomorphisms generated by Hamiltonians $K$ which are linear at infinity: $$K=\kappa R + \textrm{constant}$$ for large $R$, for some constant slope $\kappa$. 
Write $\mathrm{Ham}_{\ell \geq 0}(X,\omega),\mathrm{Ham}_{\ell > 0}(X,\omega)$ for the subsets involving only slopes $\kappa\geq 0$, $\kappa>0$ respectively.

For $g:S^1 \to \mathrm{Ham}_{\ell}(X,\omega)$, there is a group $\Gamma$ of choices of ``lifts'' $\widetilde{g}$ related to the Novikov ring $\Lambda$ $($\ref{Subsection G tildeG groups}$)$. These define an extension  
$\widetilde{\pi_1}(\mathrm{Ham}_{\ell}(X,\omega))$ of $\pi_1(\mathrm{Ham}_{\ell}(X,\omega))$.

\begin{theorem*}
 Any $g:S^1 \to \mathrm{Ham}_{\ell}(X,\omega)$ yields a $\Lambda$-module automorphism on symplectic cohomology, $\mathcal{S}_{\widetilde{g}}\in \mathrm{Aut}(SH^*(X))$, given by pair-of-pants product by an invertible element $\mathcal{S}_{\widetilde{g}}(1)\in SH^{2I(\widetilde{g})}(X)^{\times}$.
Moreover, there is a homomorphism:
$$\boxed{
\widetilde{\pi_1}(\mathrm{Ham}_{\ell}(X,\omega)) \to SH^*(X)^{\times},\;\; \widetilde{g} \mapsto \mathcal{S}_{\widetilde{g}}(1)}
$$
\end{theorem*}

\begin{theorem*}
Any $g:S^1 \to \mathrm{Ham}_{\ell \geq 0}(X,\omega)$ gives rise to $\Lambda$-module automorphisms $\mathcal{R}_{{\widetilde{g}}}=\mathcal{S}_{{\widetilde{g}}}:SH^*(X)\to SH^{*+2I(\widetilde{g})}(X)$ making the following diagram commute:
$$\boxed{
\xymatrix{ SH^*(X)  \ar@{<-}_{c^*}[d]
\ar@{->}_-{\sim}^-{\mathcal{R}_{\widetilde{g}}}[r]  & SH^{*+2I(\widetilde{g})}(X)
 \ar@{<-}^{c^*}[d] \\
QH^*(X)
\ar@{->}[r]^-{r_{\widetilde{g}}} & QH^{*+2I(\widetilde{g})}(X)}
}
$$
$r_{\widetilde{g}}$ is a count of holomorphic sections of a Hamiltonian fibration $E_g\to S^1$, it is quantum cup product by $r_{\widetilde{g}}(1)\in QH^{2I(\widetilde{g})}(X)$, and via $\psi^{\pm}$ it can be identified with $\mathcal{R}_{\widetilde{g}}=\varphi\circ \mathcal{S}_{\widetilde{g}}: HF^*(H_0)\to HF^{*+2I(\widetilde{g})}(H_0)$ where $\varphi$ is a continuation map.

Moreover, there is a homomorphism:
$$\boxed{
\widetilde{\pi_1}(\mathrm{Ham}_{\ell\geq 0}(X,\omega)) \to QH^*(X),\;\; \widetilde{g} \mapsto r_{\widetilde{g}}(1)}
$$
\end{theorem*}
\begin{proof}
 The maps $\mathcal{R}_{\widetilde{g}}=\varphi_H\circ \mathcal{S}_{\widetilde{g}}: HF^*(H) \to HF^*(g^*H) \to HF^*(H)$, where $\varphi_H$ is the continuation, are compatible with continuations since $\mathcal{S}_{\widetilde{g}}$ is (Theorem \ref{Theorem commutative diagram for action}) and continuations are. The identification of $r_{\widetilde{g}}$ is a gluing argument (Theorem \ref{Theorem rg element}).
\end{proof}

\begin{remark*}
 We briefly explain why negative slopes $\kappa$ of $K$ are not allowed for $r_{\widetilde{g}}(1)\in QH^*(M)$ whereas $\mathcal{S}_{\widetilde{g}}(1)\in SH^*(M)$ is always defined (Theorem \ref{Theorem commutative diagram for action}).
To define an endomorphism of $SH^*(M)$ it suffices that (1) for each slope $k\in \R$ there is a $k'\in \R$ and a map $HF^*(H_k) \to HF^*(H_{k'})$, where $H_k$ denotes a Hamiltonian of slope $k$ at infinity; and (2) these maps are compatible with continuation maps. Since at infinity $g^*H_k$ has slope $k-\kappa$, one can define a continuation map $HF^{*}(g^*H_0,g^*J) \to HF^*(H_{k'},J)$ provided $k'\geq k-\kappa$. If the slope $\kappa<0$ is negative, then one cannot hope to choose $k'=0$, 
which is required to build $r_{\widetilde{g}}(1) \in QH^*(M)\cong HF^*(H_0)$.
\end{remark*}

The existence of a loop $g:S^1 \to \mathrm{Ham}_{\ell>0}(X,\omega)$ is equivalent to assuming that there is a Hamiltonian $S^1$-action on $X$ which agrees with the Reeb flow at infinity.

\begin{corollary*}
For any $g: S^1 \to \mathrm{Ham}_{\ell > 0}(X,\omega)$,
$$
\boxed{SH^*(X) \cong QH^*(X)/\ker r_{\widetilde{g}}^k}
$$
is induced by $c^*:QH^*(X) \to SH^*(X)$ for $k\geq \mathrm{rank}\, H^*(X)$.
\end{corollary*}
\begin{proof}
 For $H_0$ small, $HF^*(H_0)\cong QH^*(X)$ (it reduces to the Morse complex). There is a natural pull-back $H_k=(g^k)^*H_0$ (see \ref{Subsection Overview of the construction for closed mfds}). The Hamiltonian generating $g$ has positive linear growth, so $SH^*(X)=\varinjlim HF^*(H_k)$. Finally, $S_{\widetilde{g}}^k:HF^*(H_k)\cong HF^{*+2kI(\widetilde{g})}(H_0)$, so as in \ref{Subsection The role of rg in determining SH}: $SH^*(X)\cong HF^*(H_0)/\ker \mathcal{R}_{\widetilde{g}}^k$.
\end{proof}

\subsection{Comparison with the Seidel representation}
\label{Subsection Comparison with the Seidel representation}
 The element $r_{\widetilde{g}}(1)$ plays a similar role to the quantum invertible element $q(g,\widetilde{g})$ of the Seidel representation \cite{Seidel3} for closed symplectic manifolds $(C,\omega)$:
$$
q: \widetilde{\pi_1}(\mathrm{Ham}(C,\omega)) \to QH_*(C,\omega)^{\times}, \; g \mapsto q(g,\widetilde{g}).
$$
These invertibles arise naturally in Floer homology and one can pass to quantum homology via $QH_*(C)\cong HF_*(C)$. In our case there is only a homomorphism $c^*:QH^*(X) \to SH^*(X)$, the $r_{\widetilde{g}}(1)$ can be non-invertible in $QH^*(X)$, but they become invertibles $\mathcal{R}_{\widetilde{g}}(1)$ on the quotient $SH^*(X)$. Indeed, $r_{\widetilde{g}}$ represents the continuation maps defining $SH^*(X)$ and $r_{\widetilde{g}}$ is nilpotent precisely when $SH^*(X)=0$.

The natural generalization of the Seidel representation to non-compact $(X,\omega)$ would have been to consider compactly supported Hamiltonian diffeomorphisms $\mathrm{Ham}_{\ell=0}(X,\omega)$, so that their action does not affect the dynamics at infinity. In that case, $r_{\widetilde{g}}(1)=\mathrm{PD}[q(g,\widetilde{g})]$ is an invertible in $QH^*(X,\omega)$ in degree $0$ (Example \ref{Example Maslov zero loops}), and it induces a degree preserving automorphism $\mathcal{R}_{\widetilde{g}}$ on $SH^*(X,\omega)$.

However, it would not have helped to compute $SH^*(X,\omega)$. To help compute $SH^*(X)$ we need the diffeomorphism to dramatically affect the dynamics at infinity, so that it relates the different Floer cohomologies arising in the direct limit.
%
\subsection{Outline of the paper, Conventions, Acknowledgements.}
\strut\\ \indent
\textbf{Outline of the paper.} \\[1mm]
\begin{tabular}{|l|l|}
\hline
 \emph{Section \ref{Section Symplectic cohomology}}: review of $HF^*,SH^*$ 
&  \emph{Section \ref{Section Negative line bundles}}: negative line bundles \\
\hline
\emph{Section \ref{Section Hamiltonian symplectomorphisms act}}: $\mathcal{S}_{\widetilde{g}}$ and $\pi_1\mathrm{Ham}(M,\omega)$ & \emph{Section \ref{Section O(-n) over CP1}}: $SH^*(\mathcal{O}(-n)\to \P^1)$ by a
\\
action on Floer complexes 
& Grothendieck-Riemann-Roch argument
\\
\hline
\emph{Section \ref{Section Construction of the automorphism}}: $\mathcal{R}_{\widetilde{g}}$ and the Floer 
& \emph{Section \ref{Section O(-n) over CPm}}: $QH^*(\mathcal{O}(-1)\to \P^m)$ directly and
\\
 theoretic $r_{\widetilde{g}}$ 
& $SH^*(\mathcal{O}(-n)\to \P^m)$ by virtual localization\\
\hline
\emph{Section \ref{Section Pseudoholomorphic sections}}: $E_g\to \P^1$ and the 
& \emph{Section \ref{Section General theory for negative line bundles}}:  $r(1)=\pi_M^*c_1(L)$ for\\
 algebro-geometric $r_{\widetilde{g}}$ 
& negative $L\to B$\\
\hline
 \emph{Section \ref{Section Gromov-Witten invariants}}: review GW invariants
 & \emph{Section \ref{Section Negative vector bundles}}: negative vector bundles.\\
\hline
\end{tabular}
%
\strut \\[1mm] \indent
\textbf{Conventions.}
 We only consider the summand of $SH^*(M)$ coming from the contractible orbits $($which is everything if $\pi_1(B)\!=\!1)$. Observe that if $\pi_1(B)\!\neq\! 1$, then a vanishing result for this summand implies vanishing of the full $SH^*(M)$ since the unit lies in this summand (Corollary \ref{Corollary unit for SH}). We use char$(\Lambda)=2$ to avoid discussing orientations, but we kept track of orientation signs: Remark \ref{Remark orientation signs}.
%
\strut \\[1mm] \indent
\textbf{Acknowledgements.} I thank Paul Seidel and Davesh Maulik for helpful discussions in the early stages of this project. I thank Ivan Smith for his great patience in listening to the progress on this project, particularly when it first seemed that my work was (erroneously) contradicting his observation (Remark \ref{Remark Ivan}) because of the unexpected result at the end of \ref{Subsection The role of rg in determining SH}. I thank Gabriel Paternain for suggesting to rephrase Theorem \ref{Theorem Kodaira Vanishing theorem statement} as Corollary \ref{Corollary Kodaira vanishing full}.
%
%
\section{Conical symplectic manifolds and symplectic cohomology}
\label{Section Symplectic cohomology}
%
\subsection{Conical symplectic manifolds}
\label{Subsection Conical symplectic manifolds}
We will consider non-compact symplectic manifolds
$(M,\omega)$, whose symplectic form $\omega$ is typically non-exact. We call $M$ \textbf{conical at infinity} if outside a bounded domain $M_0\subset M$ there is a symplectomorphism
$$
\psi: (M\setminus M_0,\omega|_{M\setminus M_0}) \cong (\Sigma \times [1,\infty), d(R \alpha)).
$$
where $(\Sigma,\alpha)$ is a contact manifold, and $R$ is the coordinate on $[1,\infty)$.

The conical condition implies that outside of $M_0$ the symplectic form becomes exact: $\omega=d\theta$ where $\theta=\psi^*(R\alpha)$. It also implies that the Liouville vector field $Z=\psi^*(R\partial_R)$ (defined by $\omega(Z,\cdot) = \theta$) will point stricly outwards along $\partial M_0$. Finally, it implies that $\psi$ is induced by the flow of $Z$ for time $\log R$, so we can simply write $\Sigma=\partial M_0$, $\alpha=\theta|_{\Sigma}$ (pull-back).

By \emph{conical structure} $J=J_t$ we mean a (typically time-dependent) $\omega$-compatible almost complex structure on $M$ (so $\omega(\cdot,J\cdot)$ is a $J$-invariant metric) satisfying the \emph{contact type condition} $
J^*\theta=dR
$ for large $R$.
On $\Sigma$ this implies $J Z = Y$ where $Y$ is
the Reeb vector field for $(\Sigma,\alpha)$ defined by $\alpha(Y)=1$, $d\alpha(Y,\cdot)=0$. 

By choosing $\alpha$ or $\Sigma$ generically, one ensures that $\alpha$ is sufficiently generic so that the periods of the Reeb vector field $Y$ form a countable closed subset of $[0,\infty)$.

In this Section we succinctly construct $SH^*(M)$. In the exact setup ($\omega=d\theta$ on all of $M$) symplectic cohomology was introduced by Viterbo \cite{Viterbo1}, and we refer to \cite{Ritter1} for details and to Seidel \cite{Seidel} for a survey. In the non-exact setup it was first constructed by the author in \cite{Ritter2}, to which we refer for details. In this paper we use a larger Novikov ring than in \cite{Ritter2} (see \ref{Subsection geometrical novikov ring}), so that our conventions mirror \cite{Hofer-Salamon,Seidel}.
%
\subsection{Weak$^+$ monotonicity}
\label{Subsection weak monotonicity}
We assume $M$ satisfies \textbf{at least one of}:

\begin{enumerate}
 \item there is a $\lambda\geq 0$ such that $\omega(A) = \lambda c_1(TM)(A)$ for all $A\in \pi_2(M)$;

\item $c_1(TM)(A)= 0$ for all $A \in \pi_2(M)$;

\item the minimal Chern number $|N| \geq \dim_{\C} M -1$.

\end{enumerate}

Recall $|N|$ is defined by $c_1(TM)(\pi_2(M))=N\Z$. 
The requirement that one of these conditions holds is equivalent to the statement:
$$
A \in \pi_2(M),\; 2-\dim_{\C}M\leq c_1(TM)(A) < 0 \Longrightarrow \omega(A)\leq 0.
$$
%
%
%
\subsection{Hamiltonian dynamics}
\label{Subsection Hamiltonian dynamics}
Our Hamiltonians $H=H_t \in C^{\infty}(M\times S^1,\R)$ (typically time-dependent) will always be linear at infinity:
$$
H=mR+\textrm{constant} \quad \textrm{ for }R\gg 0
$$
with slope $m \in \R$ not equal to a Reeb period.
The Hamiltonian vector field $X_H$ is defined by
$
\omega(\cdot,X_H) = dH
$, and we call \emph{1-orbits} the $1$-periodic orbits $x$ of $X_H$, $\dot{x}(t)=X_{H_t}(x(t))$. In the region where $H$ is linear the $1$-orbits $x(t)$ lie inside hypersurfaces $R=\textrm{constant}$ and correspond to the Reeb orbits $y(t)=x(t/T)$ in $\Sigma$ of period $T=h'(R)<m$.
The $1$-orbits are the zeros of the \emph{action 1-form},
$\textstyle
dA_H(x) \cdot \xi = - \int_0^1 \omega(\xi, \dot{x} - X_H) \, dt
$
where $x\in \mathcal{L}M = C^{\infty}(S^1,M)$, and $\xi \in T_x\mathcal{L}M$.
%
\subsection{A cover of the loop space}
\label{Subsection Cover of the loop space}
%
\strut\\
\textbf{Convention.} \emph{From now on, consider only the component $\mathcal{L}_0 M \subset \mathcal{L}M$ of contractible free loops in $M$. We abbreviate 
$c_1 = c_1(TM,\omega).$
}

Consider the cover of $\mathcal{L}_0 M$ introduced by Hofer-Salamon \cite{Hofer-Salamon},
$$\widetilde{\mathcal{L}_0 M}= \{(v,x): x\in \mathcal{L}_0 M, v: D^2 \to M \textrm{ a smooth disc with boundary } \partial v = x\}/\sim$$
identifying $(v_1,x_1) \sim (v_2,x_2)$ whenever $x_1=x_2$ and $\omega,c_1$ both vanish on the sphere $v_1 \# \overline{v_2}$ obtained by gluing the two discs together along the common boundary.

The covering group of $\widetilde{\mathcal{L}_0 M} \to \mathcal{L}_0 M$ is 
$\boxed{\Gamma = \pi_2(M)/\pi_2(M)_0}$
 where $\pi_2(M)_0$ is generated by the spheres on which $\omega,c_1$ both vanish. $\Gamma$ acts by ``gluing in spheres''. This cover is useful because the action is now well-defined:
$$A_H: \widetilde{\mathcal{L}_0 M} \to \R,\; A_H(v,x) = -\int_{D^2} v^*\omega + \int_0^1 H_t(x(t))\, dt.$$
%
%
\subsection{$\Z$-grading on the cover}
\label{Subsection grading of symplectic homology}
%
The Conley-Zehnder grading of $(v,x) \in \widetilde{\mathcal{L}_0M}$ is well-defined and denoted 
$\boxed{\mu_H(v,x)}.$
%
%
We refer to Salamon \cite{Salamon} for details. 

\textbf{Convention.} \emph{For a $C^2$-small time-independent Hamiltonian, the $\mu_H$ of a critical point of $H$ is equal to its Morse index. Our conventions differ from \cite{Salamon} by reversing the sign of $H$, but our index $\mu_H$ in fact agrees with the $\mu_H$ of \cite{Salamon}.}
%
\subsection{Coefficient ring $\Lambda$}
\label{Subsection geometrical novikov ring}
%
The geometrical Novikov ring $\Lambda = \oplus_k \Lambda_k$ is defined using
$$\boxed{
\begin{array}{l}
\Gamma_k = \{\gamma\in \Gamma: 2c_1(\gamma)=k\} \qquad (\textrm{cohomological grading})\\
\Lambda_k = \{\;{\displaystyle  \sum_{j=0}^{\infty} n_{j} \gamma_j :
n_{j} \in \Z/2, \gamma_j\in \Gamma_k,  \lim_{j \to \infty} \omega(\gamma_j) = \infty
}\;\}
\end{array}
}
$$

\textbf{Homological grading.} in homology, the grading is reversed ($-2c_1(\gamma)=k$) so quantum cohomology/homology is compatible with Poincar\'{e} duality (\cite[11.1.16]{McDuff-Salamon2}).

\textbf{Characteristic 2.} \emph{We use $\Z/2$ to avoid the labour of discussing orientation signs, but one can do everything over characteristic $0$. Also see Remark \ref{Remark orientation signs}.}
%
%
\subsection{Floer cohomology}
\label{Subsection Floer cohomology}
%
Denote
$$\mathcal{P}_k(H)=\{c=(v,x) \in \widetilde{\mathcal{L}_0 M}: x \textrm{ is a (contractible) 1-orbit of } H \textrm{ and } \mu_H(c)=k\}
$$ 
Then $CH^*(H,J)$ is generated over $\Lambda$ by $\mathcal{P}_*(H)$:
$$
CF^k(H) =
 \{\; \sum_{j=0}^{\infty} n_{j} c_j :
n_{j} \in \Z/2, c_j \in \mathcal{P}_k(H), \lim_{j \to \infty} A_H(c_j) = -\infty
\;\}.
$$
The choice of sign is because $A_H(\gamma\# v,x) = A_H(v,x) - \omega(\gamma)$ for $\gamma \in \Gamma$, and we want $CF^*(H) = \oplus_k CF^k(H)$ to be a $\Lambda$-module by extending the action of $\Gamma$:
$$
\Gamma\ni\gamma: CF^k(H) \to CF^{k+|\gamma|}(H), \gamma\cdot (v,x) = (\gamma \# v,x)
$$
where we use that $\mu_H(\gamma \# v) = \mu_H(v)+2c_1(\gamma)$ (see \cite{Hofer-Salamon}).

\textbf{Convention.} \emph{We always assume that we made a time-dependent perturbation of $(H,J)=(H_t,J_t)$ to ensure that $(H,J)$ is regular: all $1$-orbits are non-degenerate zeros of $dA_H$ (which ensures that $CF^*(H,J)$ is finitely generated over $\Lambda$) and the following moduli spaces of Floer trajectories are smooth:}
$$\begin{array}{lll} \mathcal{M}(x,y)&=&\{u: \R\times S^1 \to M: \partial_s u + J_t(\partial_t u - X_{H_t}(u)) = 0 \\
   &&\quad\quad u\to x,y \textrm{ as } s\to -\infty,+\infty\}
 \; / \; (u \sim u(\cdot + \textrm{constant},\cdot))
 \end{array}$$ 
Separating the moduli spaces according to lifts yields dimension \& energy estimates:
$$
\begin{array}{l}
\mathcal{M}((v,x),(w,y))  = \{ u \in \mathcal{M}(x,y): \exists \textrm{ lift } \widetilde{u}: \R \to \widetilde{\mathcal{L}_0 M} \textrm{ with ends } (v,x),(w,y) \}\\[1mm]
\textrm{dim }\mathcal{M}(c,c') = \mu_H(c) - \mu_H(c')-1\\[1mm]
E(u) = \int_{\R} |\partial_s u|_J^2 \, ds = A_H(c) - A_H(c'),\; \forall u\in \mathcal{M}(c,c') \quad (|\cdot|_J^2 = \int_{S^1} \omega(\cdot,J_t\cdot)\,dt)
\end{array}
$$
This energy estimate, combined with a maximum principle and a bubbling analysis, ensures that these moduli subspaces are compact up to broken trajectories. The maximum principle forces the trajectories to stay in a bounded region determined by $x,y,J$ (using $J$ is conical). The bubbling of $J$-holomorphic spheres is ruled out by the methods of Hofer-Salamon \cite{Hofer-Salamon} (using weak monotonicity).

The differential $d: CF^k(H,J) \to CF^{k+1}(H,J)$ on a generator $c'$ is
$d c' = \sum c$
summing over $u\in \mathcal{M}(c,c')$ with $\mu_H(c)-\mu_H(c')-1=0$. Extending $d$ linearly and proving $d^2=0$ one obtains Floer cohomology
$
HF^k(H,J) = H^k(CF^*(H,J),d).
$
%
%
%
%
\subsection{Continuation maps}
\label{Subsection continuation maps}
For Hamiltonians $H^{\pm}$ with slopes $m^+\leq m^-$,
$$\varphi^*: CF^*(H^+,J^+) \to CF^*(H^-,J^-),$$
is $\varphi^*(c^+)\! =\! \sum c^-$ summing over $u \in \mathcal{N}(c^-,c^+)$ with $\mu_{H^-}(c^-) \!-\! \mu_{H^+}(c^+)\!=\!0$, where:
$$
\begin{array}{l}
\mathcal{C}(x^-,x^+) = \{ u: \R\times S^1 \to M: \partial_s u + J_z(\partial_t u - X_{H_z}) = 0, {\displaystyle \lim_{s\to \pm \infty}} u(s,\cdot) = x^{\pm}\}\\
\mathcal{N}(c^-,c^+)  = \{ u \in \mathcal{C}(\partial c^-,\partial c^+): \exists \textrm{ lift } \widetilde{u}: \R \to \widetilde{\mathcal{L}_0 M} \textrm{ with ends } c^-,c^+ \}\\[1mm]
\textrm{dim }\mathcal{N}(c^-,c^+) = \mu_H(c^-) - \mu_H(c^+)\\[1mm]
E(u) = \int_{\R} |\partial_s u|_{J_z}^2 \, ds = A_H(c^-) - A_H(c^+) + \int_{\R \times S^1} \partial_s H_z(u)\, ds\wedge dt,\; \forall u\in \mathcal{N}(c^-,c^+)
\end{array}
$$
where we fix some \emph{monotone} homotopy $H_z=H_{s+it}$ from $H^-_t$ to $H^+_t$. Recall \emph{monotone} means $H_z=h_s(R)$ for large $R$ with $\partial_s h_s'(R)\leq 0$. 

To be precise, $(H_z,J_z)$ depend on the cylinder's coordinates $z=s+it$, with $(H_z,J_z) = (H^{\pm}_t,J^{\pm}_t)$ for large $|s|$ (so $\partial_s H_z(u) = 0$ outside a compact subset of $\R\times S^1$ determined by $H_z$), and a generic choice $(H_z,J_z)$ ensures $\mathcal{N}(c^-,c^+)$ is smooth. The maximum principle still applies (this uses that $J_z$ is conical at infinity, and that $H_z$ is monotone) so $u$ must land in a compact $C$ determined by $x^-,x^+,J_z$ and so in the above energy estimate $|\partial_s H_z(u)|\leq \max_C |H_z|$. This ensures $\mathcal{N}(c^-,c^+)$ are compact up to broken trajectories.

Extending $\varphi$ linearly, and proving $\varphi^*$ is a chain map, yields continuation maps
$$
\varphi^*: HF^*(H^+,J^+) \to HF^*(H^-,J^-) \qquad (m^+ \leq m^-).
$$
%
%
%
%
\subsection{Properties of continuation maps}
\label{Subsection properties of continuation maps}
\begin{enumerate}
\item $\varphi^*: HF^*(H^+) \to HF^*(H^-)$ only depends on the slopes at infinity, since the choice $(H_s,J_s)$ only affects the map up to a chain homotopy.

\item Concatenating monotone homotopies yields composites of continuation maps. 

\item If the slopes are the same, then the continuation map is an isomorphism.
\end{enumerate}

\begin{lemma}\label{Lemma continuation isos if no in between Reebs}
 If there are no Reeb periods between $m^-$ and $m^+$, then $\varphi^*: HF^*(H^+) \to HF^*(H^-)$ is an isomorphism.
\end{lemma}
\begin{proof}
 After a continuation isomorphism which does not change the slopes at infinity, we may assume $H^-$, $H^+$ are equal except on $R\geq R_0$ where $h^-=m^- \cdot R$ and $h^+=m(R)\cdot R$ with $m(R)$ decreasing from $m^-$ to $m^+$ on a compact subinterval of $R\geq R_0$ and then remaining constantly $m^+$.

All generators for $H^{\pm}$ coincide and lie in the region $R<R_0$ where $H^-=H^+$. Pick a homotopy $H_s$ from $H^-$ to $H^+$ which is $s$-independent on $R<R_0$ and monotone on $R\geq R_0$. By the maximum principle, all continuation solutions $u$ lie in $R<R_0$. But in that region $H_s$ is $s$-independent so non-constant $u$ would yield a $1$-dimensional family: $u(\cdot+\textrm{constant},\cdot)$. So the $0$-dimensional moduli spaces consist of constant $u$'s. So $\varphi^*=\textrm{identity}: HF^*(H^+) \to HF^*(H^-)$.
\end{proof}
%
%
%
%
\subsection{Symplectic cohomology}
\label{Subsection symplectic cohomology}
\label{Subsection summary symplectic cohomology}
%
%
$SH^*(M) = \varinjlim HF^*(H)$ is the direct limit over the continuation maps. It can be computed as the direct limit over a sequence of Hamiltonians with slopes $\to \infty$.
%
\subsection{Negative slopes and Poincar\'{e} duality}
\label{Subsection duality and negative slopes}
%
Replacing incoming trajectories by outgoing trajectories defines Floer homology $HF_*(H)$:
$$
\begin{array}{l}
CF_k(H) =
 \{\; \prod_{j=0}^{\infty} n_{j} c_j :
n_{j} \in \Z/2, c_j \in \mathcal{P}_k(H), {\displaystyle \lim_{j \to \infty}} A_H(c_j) = -\infty
\;\}\\
\displaystyle \delta c =
\prod c' \quad \textrm{taking product over all }u\in \mathcal{M}(c,c'),\mu_H(c)-\mu_H(c')-1=0.
\end{array}
$$
\begin{lemma}[Poincar\'e duality]\label{Lemma poincare duality}\strut\!\!\!
 $CF^*(H_t) \!\cong\! CF_{2n-*}(-H_{-t})$ are canonically isomorphic chain complexes $($send orbits $x(t)$ to $x(-t)$, Floer solutions $u(s,t)$ to $u(-s,-t))$.
\end{lemma}

\begin{remark*} We always deal with finitely generated modules over $\Lambda$, so $CF_*,CF^*$ are identifiable modules, but the differentials are dual to each other. We compared $SH^*(M),SH_*(M)$ in \cite{Ritter3}. In this paper we will only use $SH^*(M)$.
\end{remark*}
%
%
\subsection{Quantum cohomology and locally finite homology}
%
The \emph{quantum cohomology} as a $\Lambda$-module is $QH^*(M,\omega)=H^*(M;\Lambda)$ with underlying chain complex $$QC^k(M)=\bigoplus_{i+j=k} C^i(M)\otimes \Lambda_j.$$ 

The \emph{locally finite quantum homology} is the $\Lambda$-module
$$
\begin{array}{l}
QH^{lf}_*(M) = H_*^{lf}(M;\Lambda). 
\end{array}
$$
Recall the latter is \emph{locally finite homology}: at the chain level one allows infinite $\Lambda$-linear combinations of chains provided that any point of $M$ has a neighbourhood intersecting only finitely many of the chains arising in the sum. We could identify this with a relative homology, $H_*^{lf}(M;\Lambda) \cong H_*(M;M\setminus M_0;\Lambda)$, but we will not.

By \emph{quantum intersection product} we mean the map:
$$
\boxed{*:QH_*^{lf}(M) \otimes QH_*^{lf}(M) \to QH_*^{lf}(M)}
$$
constructed as follows. Given $\alpha\in H_i^{lf}(M)$, $\beta\in H_j^{lf}(M)$, $\gamma\in H_k(M)$, and $A\in H_2(M;\Z)$, let $\mathrm{GW}_{0,3,A}^M(\alpha,\beta,\gamma)$ be the genus $0$ Gromov-Witten invariant (modulo $2$) of $J$-holomorphic spheres in class $A$ meeting generic representatives of the lf cycles $\alpha,\beta$ and of the cycle $\gamma$. This invariant is zero for generic $J$ unless $i+j+k = 4\dim_{\C}M-2c_1(A).$ 
For details, see Section \ref{Section Gromov-Witten invariants}.
Then define the quantum product $*$: 
$$(\alpha * \beta)\bullet \gamma = \sum_A \mathrm{GW}_{0,3,A}^M(\alpha,\beta,\gamma)\otimes A \in \Lambda$$
 where $\bullet$ is the (ordinary) intersection product: 
$$
\boxed{\bullet: H_*^{lf}(M) \otimes H_{2\dim_{\C}M-*}(M) \to \Z/2}
$$
This determines $\alpha*\beta$, then extend $\Lambda$-linearly to $QH_*^{lf}(M)^{\otimes 2}$.

\emph{Poincar\'{e} duality} is induced by ordinary Poincar\'e duality (see \ref{Subsection geometrical novikov ring} for grading):
$$
\boxed{\mathrm{PD}:QH^*(M) \cong QH_{2\dim_{\C}M-*}^{lf}(M)}
$$
Via Poincar\'e duality, quantum intersection product becomes quantum cup product 
$$\boxed{*: QH^p(M) \otimes QH^q(M) \to QH^{p+q}(M)}$$
%
%
\subsection{Canonical map $\mathbf{c^*:QH^*(M)\to SH^*(M)}$}
\label{Subsection relation to quantum homology}
The map $c^*$
is the direct limit of the continuation maps $HF^*(H_0) \to HF^*(H)$, where we fix a time-independent Hamiltonian $H_0$ which is a $C^2$-small Morse perturbation of $0$, having (possibly variable) positive slopes at infinity smaller than the minimal Reeb period. The choice of $H_0$ does not affect $HF^*(H_0)$ or $c^*$ by Lemma \ref{Lemma continuation isos if no in between Reebs}. For small enough $H_0$, the $1$-orbits of $H_0$ are all critical points of $H_0$ and the Floer trajectories are all time-independent $-\nabla H_0$ trajectories. So $CF_*(H_0) = CM_*(H_0;\Lambda)$ is the Morse complex for $H_0$. Finally, Morse homology is isomorphic to ordinary homology.

\begin{remark*}
Section \ref{Subsection Construction of the psi- psi+ maps}  constructs $c^*$ as $\psi^-: QH^*(M,\omega) \to SH^*(H)$, via a count of pseudo-holomorphic sections of a Hamiltonian fibration over a disc intersecting a given lf cycle at the disc's centre. In \cite{Ritter3}, we constructed $c^*$ as a count of spiked discs (a $-\nabla H$ flowline from a critical point of $H$ to the centre of a disc satisfying a Floer continuation equation). Both constructions involve the same count of discs. The spike is used to identify locally finite homology and Morse cohomology.
\end{remark*}

Now $c^*$ intertwines the (quantum) cup product on $QH^*(M)$ and the pair-of-pants product on $SH^*(M)$: we proved this in \cite{Ritter3} (our discussion there explains how the proof works in the non-exact setup using \cite{PSS}). One only needs to prove this for $QH^*(M)\to HF^*(H_0)$ since the POP product is compatible with continuations \cite{Ritter3}. Our proof in \cite{Ritter3} becomes simpler now thanks to the mutually inverse
$$\boxed{\psi^-:QH^*(M) \stackrel{\sim}{\to} HF^*(H_0),\quad \psi^+:HF^*(H_0) \stackrel{\sim}{\to} QH^*(M)}$$
which we construct in \ref{Subsection Construction of the psi- psi+ maps} (such $\psi^{\pm}$ were the main difficulty in the proof \cite{Ritter3}).


\begin{lemma}\label{Lemma HF H0 and QH are iso as rings}
 $HF^*(H_0) \cong QH^*(M)$ is an isomorphism of rings using the pair-of-pants product and the quantum cup product, and $c^*: QH^*(M) \to SH^*(M)$ respects the product structures (in fact, also the TQFT structures).
\end{lemma}

\begin{corollary}\label{Corollary unit for SH}
 $\psi^-(1)$ is the unit for $SH^*(M)$ (see \cite{Ritter3} for a TQFT proof).
\end{corollary}

\begin{remark*}
 In the Lemma, we actually compose the pair-of-pants product with a continuation map:
$
HF^*(H_0) \otimes HF^*(H_0) \to HF^*(2H_0) \stackrel{\equiv}{\to} HF^*(H_0).
$
By Lemma \ref{Lemma continuation isos if no in between Reebs}, for small $H_0$ the continuation is an identification.
\end{remark*}

\begin{lemma}\label{Lemma lf homology is sh of -h}\strut\!\!\!\!
 $HF_*(-H_0) \!\cong\! QH_*^{lf}(M)$ as $\Lambda$-algebras via Poincar\'{e} duality Lemma \ref{Lemma poincare duality}.
\end{lemma}
%
%
%
%
\section{Hamiltonian symplectomorphisms action on Floer cohomology}
\label{Section Hamiltonian symplectomorphisms act}
%
%
\subsection{$\mathbf{G,\widetilde{G}}$ groups, and the index $\mathbf{I(\widetilde{g})}$}
\label{Subsection G tildeG groups}\label{Subsection Maslov index I of g}
%
Let $(M,\omega)$ be as described in Section \ref{Section Symplectic cohomology}.

Let $\textrm{Ham}(M,\omega)$ denote the group of (smooth) Hamiltonian automorphisms. Let $G$ denote the group of (smooth) loops based at the identity:
$$G=\{g:S^1 \to \textrm{Ham}(M,\omega), g(0)=\textrm{id}\}.$$
Let $K^g$ be the Hamiltonian generating $g_t$ (recall any smooth path $(g_t)_{0\leq t \leq 1}$ of Hamiltonian diffeos yields a smooth 
$K^g: [0,1] \times M \to \R$ with $\partial_t(g_t \cdot) = X_{K^g}(t,g_t\cdot)$
and two choices of $K^g$ differ by a constant. Since $g$ is $1$-periodic, so is $K^g$).

Observe that $g\in G$ acts on $\mathcal{L}_0 M \subset C^{\infty}(S^1,M)$ by
$$
(g\cdot x)(t) = g_t(x(t)),
$$
which lifts to an action on the cover $\widetilde{\mathcal{L}_0 M}$ from \ref{Subsection Cover of the loop space}. Denote by $\widetilde{g}$ a choice of lift. The lifts define a group $\widetilde{G}$ \cite{Seidel3} which is an extension of $G$: $1\to \Gamma \to \widetilde G \to G \to 1$. 
We recall the Maslov index $I(\widetilde{g})$ from \cite{Seidel3}. Any $c=(v,x)\in \widetilde{\mathcal{L}_0M}$ determines (up to homotopy) a symplectic trivialization of $x^*(TM,\omega)$, namely
$$
\tau_c: x^*TM \to S^1 \times (\R^{2n}, \omega_0)
$$
obtained by restricting a trivialization of $v^*(TM,\omega)$. A lift $\widetilde{g}$ induces (up to homotopy) a loop of symplectomorphisms $\ell(t) \in \textrm{Sp}(2n,\R)$ by writing its linearization in terms of this trivialization:
$$
\ell(t) = \tau_{\widetilde{g}(c)}(t) \circ dg_t(x(t)) \circ \tau_c(t)^{-1}.
$$
Then define the Maslov index 
$\boxed{
I(\widetilde{g}) = \textrm{deg}(\ell)
}$
where $\textrm{deg}: H_1(\textrm{Sp}(2n,\R)) \to \Z$ is the isomorphism induced by the determinant $U(n) \to S^1$ on $U(n)\subset \textrm{Sp}(2n,\R)$.

$I(\widetilde{g})$ is independent of the choice of $(v,x)$ and it only depends on the homotopy class $\widetilde{g}\in \pi_0(\widetilde{G})$. The induced map $\pi_0(\widetilde{G}) \to \Z$ is a homomorphism. For $g=\textrm{id}$ and picking $\widetilde{g}$ to be multiplication by $\gamma \in \Gamma$, $2I(\gamma) = 2c_1(\gamma)$ (homological grading).
\begin{example}\label{Example Maslov zero loops}
If $K^g$ is compactly supported, and we pick $\widetilde{g}$ to preserve the constants $(v\equiv x_0,x_0)$ outside the support of $g$, then computing $I$ for $(x_0,x_0)$: $I(\widetilde{g})=0$.
\end{example}
%
%
%
\subsection{$\mathbf{\widetilde{G}}$-action of Floer cohomology}
\label{Subsection Overview of the construction for closed mfds}
%
Define the pull-back $(g^*H,g^*J)$ of $(H,J)$:
$$\boxed{
g^*H_t(y) = H_t(g_t y) - K^g_t(g_t y)}\qquad
\boxed{J_t^g = dg_t^{-1}\circ J_t \circ dg_t}
$$
The following results are just a rephrasing of the analogous results in the closed case (Seidel \cite[Sec.4]{Seidel3}), so we omit the proofs.
\begin{lemma*}
 The pull-back of the action $1$-form is $g^*(dA_H) = dA_{g^*H}$. Therefore the lift $\widetilde{g}$ induces the pull-back $\widetilde{g}^* A_H = A_{g^*H} + \textrm{constant}$.
\end{lemma*}

\begin{corollary}\label{Corollary bijection of generators}
 The $1$-orbits $($being zeros of the action $1$-form$)$ biject via
$$
\textrm{Zeros}\,(dA_{g^*H}) \to \textrm{Zeros}\,(dA_H),\; x \mapsto g\cdot x.
$$
The Floer solutions $($being negative gradient trajectories of
 the action $1$-form with respect to the metric induced by the almost complex structure$)$, biject via
$$
\begin{array}{l}
\mathcal{M}(c,c';g^*H,g^*J) \to \mathcal{M}(\widetilde{g}c,\widetilde{g}c';H,J),\; u\mapsto g\cdot u \\
\mathcal{N}(c^-,c^+;g^*H_z,g^*J_z) \to \mathcal{N}(\widetilde{g}c^-,\widetilde{g}c^+;H_z,J_z),\; u\mapsto g\cdot u \\
\end{array}
$$
where $(g\cdot u)(s,t) = g_t(u(s,t))$.
\end{corollary}

\begin{theorem}\label{Theorem commutative diagram for action}
 For $g\in G$ with lift $\widetilde{g}$, we obtain an isomorphism
$$
\mathcal{S}_{\widetilde{g}}: CF^*(H) \to CF^{*+2I(\widetilde{g})}(g^*H), c \mapsto \widetilde{g}^{-1}\cdot c
$$
with $\mathcal{S}_{\widetilde{g}}^{-1}=\mathcal{S}_{\widetilde{g}^{-1}}$ using the reversed loop. These commute with continuations:
$$
\xymatrix@R=12pt{ CF^*(H^-,J^-)
\ar@{<-}_{g^*\varphi}[d] \ar@{->}[r]^-{\mathcal{S}_{\widetilde{g}}} & CF^{*+2I(\widetilde{g})}(g^*H^-,g^*J^-)
\ar@{<-}^{\varphi}[d] \\
CF^*(H^+,J^+)
\ar@{->}[r]^-{\mathcal{S}_{\widetilde{g}}} & CF^{*+2I(\widetilde{g})}(g^*H^+,g^*J^+) 
}
$$
where $\varphi$ is a monotone continuation map, and $g^*\varphi$ is the continuation map using 
$$
g^*H_z(y) = H_z(g_t y) - K^g_t(g_t y)
\qquad \qquad 
g^*J_z=dg_t^{-1}\circ J_z \circ dg_t.$$ 
The commutativity follows because the generators and the moduli spaces defining the continuation maps biject by Corollary \ref{Corollary bijection of generators}. In particular, one can check \cite[Lemma 4.1]{Seidel3} that $(g^*H,g^*J)$ is regular if $(H,J)$ is, and similarly for the continuation data.
\end{theorem}

\section{Construction of the automorphism on symplectic cohomology}
\label{Section Construction of the automorphism}
%
\subsection{Hamiltonian symplectomorphisms of linear growth}
\label{Subsection Hamiltonian symplectomorphisms linear at infinity}
Recall $\textrm{Ham}_{\ell}(M,\omega)$ from Section \ref{Subsection Intro general theory}. Let $G_{\ell}\subset G$ be the subgroup of all $g: S^1 \to \mathrm{Ham}_{\ell}(M,\omega)\subset \mathrm{Ham}(M,\omega)$, with $g_0=\mathrm{id}$. So $g$ is generated by a (typically time-dependent) Hamiltonian $K_t^g:M \to \R$ with $K_t^g(y) = \kappa_t R$, some $\kappa_t\in \R$, where $R(y)\in [1,\infty)$ is the radial coordinate on the conical end. 

Similarly, define $G_{\ell\geq 0},G_{\ell>0}$ by requiring $\kappa_t\geq 0,\kappa_t>0$.

We will make the simplifying assumption that the slopes $\kappa_t$ are time-independent for large $R$. The next technical remark explains that, by a time-reparametrization trick, one can always homotope $g_t$ within $G_{\ell}$ to ensure this.

\begin{remark}[Technical Remark]\label{Remark Technical Rmk about time indep of K at infty}
 The Hamiltonian vector field of $K_t$ for large $R$ is $X_{K_t}=\kappa_t Y$, where $Y$ is the Reeb vector field. Consider a flowline $\gamma(t)$, so $\gamma'(t)=\kappa_t Y$. The time-reparametrized curve $\mu(t)=\gamma(\alpha(t))$ satisfies $$\mu'(t)=\alpha'(t)\kappa_{\alpha(t)} Y = \lambda'(\alpha(t)) Y,$$ where $\lambda(t)=\int_0^{\alpha(t)} \kappa_T \, dT$. If we ensure that $\lambda'=c$ is constant and $\alpha(1)=1$, then the homotopy from $g_t$ to $g_{\alpha(t)}$ obtained by interpolating the time-coordinates allows us to assume that $K_t$ is time-independent at infinity (after the homotopy, $K_t = cR$ at infinity). To ensure those conditions, we choose $c=\int_0^1 \kappa_T \, dT $, and we want $\int_0^{\alpha(t)} \kappa_T \,dT = ct$  and $\alpha(1)=1$. If $\kappa_t>0$ everywhere then $\int_0^{t} \kappa_T \,dT$ is strictly increasing with $t$, so there is a unique solution $\alpha(t)$. Similarly if $\kappa_t<0$. If $\kappa_t$ is zero or changes sign, then the flow $g_t$ of $\kappa_t Y$ stops or reverses respectively, but a further time-reparametrization (and thus a homotopy of $g_t$) will undo this flow.

We remark that if two loops $g_t,g'_t$ are homotopic in $G_{\ell}$ and their Hamiltonians $K_t,K_t'$ are time-independent at infinity, then in fact they have the same slopes at infinity and there is a homotopy $(K_{t,r})_{0\leq r \leq 1}$ which is time-independent at infinity. For such homotopies, the maps $\mathcal{S}_{\widetilde{g}}$, $\mathcal{S}_{\widetilde{g'}}$ are the same (up to a continuation map, as explained in Corollary \ref{Corollary properties of Sg}), which is proved by the same homotopy-invariance result as in the closed case \cite[Sec.5]{Seidel3} (in the case $c<0$ one first runs the argument for $g^{-1}$, and then one applies $\mathcal{S}_{\widetilde{g}^2}$). Therefore $\mathcal{S}_{\widetilde{g}}$ really only depends on $(g,\widetilde{g})$ up to homotopy, so the above simplifying assumption is legitimate.
\end{remark}

%
\begin{lemma}\label{Lemma g*J g*H are linear}
$g^*J$ is conical and $g^*H$ is linear at infinity of slope $\mathrm{slope}(H)-\kappa$.
\end{lemma}
\begin{proof}
$g_t$ preserves $R$ at infinity since $K=K^g$ is radial there, therefore $dg_t$ preserves the Liouville vector field $Z=R \partial_R$. Moreover $g_t$ preserves the Reeb vector field $Y=X_R$, because the Lie derivative $\mathcal{L}_{X_K} Y$ vanishes:
$$
\mathcal{L}_{X_K} Y = [X_K,Y]=[X_K,X_R] = -X_{\omega(X_K,X_R)}=0,
$$
since $X_K=\partial_R K \cdot  X_R$. Since $dg_t$ preserves $Z$ and $Y$, $g^*J=dg_t^{-1}\circ J \circ dg_t$ is of contact type at infinity since $J$ is. 
Moreover, for $H$ of slope $m$ at infinity,
$
g^*H = H \circ g_t - K^g \circ g_t = (m-\kappa)R+\textrm{constant}
$
for large $R$.
\end{proof}
%
%
%

The Lemma ensures that the data $(g^*H,g^*J)$ can be used to compute the symplectic cohomology, as we let the slope of $H$ grow to infinity.

Thus, taking the direct limit of the maps $\mathcal{S}_{\widetilde{g}}: CF^*(H,J) \to CF^{*+2I(\widetilde{g})}(g^*H,g^*J)$ from Theorem \ref{Theorem commutative diagram for action}, we obtain the $\Lambda$-module automorphism:
$$
\boxed{[\mathcal{S}_{\widetilde{g}}]: SH^*(M,\omega) \stackrel{\sim}{\longrightarrow}  SH^{*+2I(\widetilde{g})}(M,\omega)}$$
\begin{remark*}
 We defined $\mathcal{S}_{\widetilde{g}}$ using $\widetilde{g}^{-1}$ instead of $\widetilde{g}$ because it should act by $\widetilde{g}$ on $CF_*$, and so on $\mathrm{Hom}_{\Lambda}(CF_*,\Lambda)$ it acts by $(\widetilde{g}\phi)(\cdot)=\phi(\widetilde{g}^{-1}\cdot)$, but we tacitly identified $CF^*\equiv CF_*$ as $\Lambda$-modules, so we should act by $\widetilde{g}^{-1}$ on $CF^*$.
\end{remark*}
The following results are just a rephrasing of the analogous results in the closed case (Seidel \cite[Sec.4, Sec.5]{Seidel3}), so we omit the proofs.
\begin{corollary}\label{Corollary properties of Sg}
 $\mathcal{S}_{\widetilde{g}}$ is the identity for $(g,\widetilde{g})=(\mathrm{id},\mathrm{id})$, and is multiplication by $\gamma$ for $(g,\widetilde{g})=(\mathrm{id},\gamma)$. It is a right-action: 
$
\mathcal{S}_{\widetilde{g_1}} \circ \mathcal{S}_{\widetilde{g_2}} = \mathcal{S}_{\widetilde{g_2}\widetilde{g_1}}
$
\emph{(}although in fact $\mathcal{S}_{\widetilde{g_2}\widetilde{g_1}}=\mathcal{S}_{\widetilde{g_1}\widetilde{g_2}}$ since we will show that $\mathcal{S}_{\widetilde{g}}$ is pair-of-pants product by the \emph{even} degree element $\mathcal{S}_{\widetilde{g}}(1)$\emph{)}. 
It is homotopy invariant: if $(g_{r,t})_{0\leq r\leq 1}$ is a smooth family of Hamiltonian automorphisms in $G_{\ell}$ based at $g_{r,0}=\mathrm{id}$, and $(\widetilde{g}_{r,t})_{0\leq r\leq 1}$ is a smooth lift to $\widetilde{G}$, then on cohomology
$
[\varphi]\circ [\mathcal{S}_{\widetilde{g_0}}] = [\mathcal{S}_{\widetilde{g_1}}]
$
where $\varphi: HF^*(g_0^*H,g_0^*J) \to HF^*(g_1^*H,g_1^*J)$ is the continuation isomorphism. In particular, the choice of $K^g_t$ generating $g_t$ does not affect the map $\mathcal{S}_{\widetilde{g}}$ on cohomology.
\end{corollary}
%
%
%
%
%
\subsection{Floer theoretic construction of $\mathbf{r_{\widetilde{g}}}$ and $\mathbf{\mathcal{R}_{\widetilde{g}}}$}
\label{Subsection The Floer construction}
In order to construct the endomorphism $\mathcal{R}_{\widetilde{g}}=\varphi_H\circ S_{\widetilde{g}}$ of $HF^*(H)$ for a \emph{monotone} continuation map $\varphi_H: HF^*(g^*H,g^*J) \to HF^*(H,J)$, one needs $g^*H\leq H$, so we require $g\in G_{\ell\geq 0}$. 
Just as before (see Remark \ref{Remark Technical Rmk about time indep of K at infty}) we will assume that, for large $R$, the Hamiltonian $K^g_t = \kappa R$ is time-independent. Since we work with $G_{\ell \geq 0}$, the slope $\kappa \geq 0$. We will treat the case $\kappa>0$, since the case $\kappa=0$ is rather easy (see Section \ref{Subsection Comparison with the Seidel representation}). 

By these assumptions, the Reeb flow is an $S^1$-action. So, after rescaling $\omega$, we may assume the time $1$ Reeb flow is a Hamiltonian $S^1$-action which is not an iterate.

\textbf{Examples:}  $S^1$-action of $g_t=e^{2\pi i t}$ on $\C^{m+1}$; rotation in the fibres $g_t=e^{2\pi i t}$ of line bundles; $S^1$-actions on Hyperk\"{a}hler ALE spaces $(X,\omega_I)$ (see \cite{Ritter2}).
%
%
%
%
%
%
%
%
\strut\\
\noindent
\begin{tabular}{lllll}
Denote: & $H_0 = \textrm{ generic Hamiltonian with slope }0<\delta<(\textrm{min Reeb period})$\\ &
$H_1 = H_0 + K^g \quad (\textrm{generic Hamiltonian with slope } \delta+\kappa)$\\ &
$H_k = H_0 + k K^g  \quad (\textrm{generic Hamiltonian with slope } \delta+k\kappa, k\in \Z)$
\end{tabular}\\
Then, by Lemma \ref{Lemma g*J g*H are linear},
$
\boxed{g^*H_k = H_{k-1}}
$
%
%
%
%

In \ref{Subsection Construction of the psi- psi+ maps} we construct the chain maps $\psi^{\pm}$ which are homotopy inverse to each other: 
$$
\psi^-: QC^*(M) \to CF^*(H_0)\qquad \psi^+: CF^*(H_0) \to QC^*(M).
$$
To ensure $\psi^+$ exists we actually need $H_0$ to be bounded at infinity. So take $H_0$ a generic $C^2$-small Hamiltonian, with $H_0=h_0(R)$ convex for $R\gg 0$ and $h_0'(R)\to 0^+$ as $R \to \infty$ ($H_0$ is not linear at infinity, but that is not an issue). $H_0,H_k$ should be thought of as perturbations of  slope $0,k\kappa$ Hamiltonians.

\begin{definition*} 
 Define
$$
\begin{array}{lll}
\mathcal{R}_{\widetilde{g}} = \varphi_0 \circ \mathcal{S}_{\widetilde{g}}: CF^*(H_{0}) \to CF^{*+2I(\widetilde{g})}(H_{0}) 
\\
r_{\widetilde{g}} = \psi^+ \circ \mathcal{R}_{\widetilde{g}}\circ \psi^-: QC^*(M) \to QC^{*+2I(\widetilde{g})}(M)
\end{array}$$
where $\varphi_0$ is the monotone continuation $($for a homotopy from $H_{0}$ to $g^*H_{0}=H_{-1})$.
$$\boxed{
\xymatrix@C=10pt{ QC^*(M) \ar@{->}^-{\psi^-}[r] \ar@{->}@/_5ex/_-{r_{\widetilde{g}}}[rrrr]
&
CF^*(H_{0})  \ar@{->}^-{\mathcal{S}_{\widetilde{g}}}[r]  \ar@{->}@/_2ex/_-{\mathcal{R}_{\widetilde{g}}}[rr]
&
CF^{*+2I(\widetilde{g})}(H_{-1})   \ar@{->}^-{\varphi_0}[r]
&
CF^{*+2I(\widetilde{g})}(H_{0})   \ar@{->}^-{\psi^+}[r]
&
QC^{*+2I(\widetilde{g})}(M)
}
}
$$
\end{definition*}

\begin{theorem}\label{Theorem SH is lim of S-bRb}  $\mathcal{S}_{\widetilde{g}}^{-k} \mathcal{R}_{\widetilde{g}}^{k}:HF^*(H_{0}) \to HF^*(H_{k})$ is a continuation map. For $k\geq \mathrm{dim}\, H^*(M)$, we may identify $SH^*(M)\equiv \mathrm{image}(\mathcal{S}_{\widetilde{g}}^{-k} \mathcal{R}_{\widetilde{g}}^{k})$ and $ 
c^*\equiv (\mathcal{S}_{\widetilde{g}}^{-k} \mathcal{R}_{\widetilde{g}}^{k})\circ \psi^- =(\mathcal{S}_{\widetilde{g}}^{-k} \psi^-)\circ r_{\widetilde{g}}^k: QH^*(M) \to SH^*(M)$. Thus
$
\boxed{SH^*(M) \cong QH^*(M)/\ker r_{\widetilde{g}}^k}
$
\end{theorem}
\begin{proof}
Abbreviate $\mathcal{S}=\mathcal{S}_{\widetilde{g}}$, $\mathcal{R}=\mathcal{R}_{\widetilde{g}}$. Consider the monotone continuation maps
$\varphi_k: HF^*(H_{k-1}) \to HF^*(H_{k}).$
 Theorem \ref{Theorem commutative diagram for action} yields the commutative diagram
$$
\xymatrix@R=16pt{ HF^{*}(H_{k+1})
\ar@{<-}_{g^*\varphi_k}[d] \ar@{->}[r]^-{\mathcal{S}}_-{\sim} & HF^{*+2I(\widetilde{g})}(H_{k})
\ar@{<-}^{\varphi_k}[d] \\
HF^{*}(H_{k})
\ar@{->}[r]^-{\mathcal{S}}_-{\sim} & HF^{*+2I(\widetilde{g})}(H_{k-1}) 
}
$$
By Property (1) in \ref{Subsection properties of continuation maps}, 
$\varphi_{k+1} = g^*\varphi_k= \mathcal{S}^{-1} \circ \varphi_{k} \circ \mathcal{S}$
and so by induction, for $k\in \Z$,
$$\varphi_{k} = \mathcal{S}^{-k} \circ \varphi_{0} \circ \mathcal{S}^{k}.$$
By property (2) in \ref{Subsection properties of continuation maps}, the continuation $HF^*(H_{0}) \to HF^*(H_{k})$ equals
$$
\varphi_{k} \varphi_{k-1}  \cdots \varphi_{1} =
 (\mathcal{S}^{-k}  \varphi_{0}  \mathcal{S}^{k})
 (\mathcal{S}^{-(k-1)}  \varphi_{0}  \mathcal{S}^{k-1})
\cdots 
 (\mathcal{S}^{-1}  \varphi_{0}  \mathcal{S}^{1}) 
= 
 \mathcal{S}^{-k}   (\varphi_{0}  \mathcal{S})^{k}
=
 \mathcal{S}^{-k}   \mathcal{R}^{k}.
$$
%
%
Let $V=HF^*(H_0)$. Conjugation by $\mathcal{S}^k$ identifies $V=HF^*(H_k)$ so $\varphi_k$ becomes $\mathcal{R}:V\to V$. Using $\psi^{\pm}$ we can identify $V=QH^*(M,\omega)$, which turns $\mathcal{R}$ into $r_{\widetilde{g}}$.
The claims then follow by the argument in \ref{Subsection The role of rg in determining SH}.
\end{proof}
\begin{definition*}
Let $\mathcal{R}: HF^*(H_{\ell})\stackrel{S}{\to} HF^{*+2I(\widetilde{g})}(H_{\ell-1}) \stackrel{\psi}{\to} HF^{*+2I(\widetilde{g})}(H_k)$
where $\psi$ is a continuation map, so $\psi=\varphi_k\circ \cdots \varphi_{\ell} = \mathcal{S}^{-k}\mathcal{R}^{k-\ell+1} S^{\ell-1}$ (proved like the case $\ell=1$ above). One easily checks that these $\mathcal{R}=\mathcal{S}^{-k}\mathcal{R}^{k-\ell+1}\mathcal{S}^{\ell}$ form a family of maps compatible with continuation maps, so they define a map on direct limits:
$$
\boxed{[\mathcal{R}_{\widetilde{g}}]: SH^*(M) \to SH^{*+2I(\widetilde{g})}(M)}
$$
\end{definition*}

\begin{corollary*}
$
SH^*(M)\!\cong\! HF^*(H_0)/\ker \mathcal{R}^k
$
for $k\!\geq\! \dim H^*(M)$, and 
$[\mathcal{R}_{\widetilde{g}}]\!=\![\mathcal{S}_{\widetilde{g}}]\in \mathrm{Aut}(SH^*(M))$.
 For $g_1,g_2\in G_{\ell\geq 0}$, $[\mathcal{R}_{\widetilde{g}_1}][\mathcal{R}_{\widetilde{g}_2}]=[\mathcal{R}_{\widetilde{g}_2\widetilde{g}_1}]$, so $[\mathcal{R}_{\widetilde{g}^k}]=[\mathcal{R}_{\widetilde{g}}]^k$.
\end{corollary*}
\begin{proof}
 The $1^{st}$ claim is Theorem \ref{Theorem SH is lim of S-bRb}. So $\mathcal{R}:HF^*(H_0) \to HF^{*+2I(\widetilde{g})}(H_0)$ induces $[\mathcal{R}]$. It is an automorphism since
$\ker \mathcal{R}^{k+1}=\ker \mathcal{R}^{k}$.  $\mathcal{S}:HF^*(H_0)\to HF^{*+2I(\widetilde{g})}(H_{-1})$ determines $[\mathcal{S}]$, so $[\mathcal{S}]=[\mathcal{R}]$ since $\mathcal{S}^{-1}\mathcal{R} = \varphi_1$ represents the identity on $SH^*(M)$. The $3^{rd}$ claim is Corollary \ref{Corollary properties of Sg} (or check directly on $HF^*(H_0)$).
\end{proof}
%
%
\subsection{Product structure on $\mathbf{SH^*(M)}$}
\label{Subsection Product structure on SH using rg}
%
\begin{theorem}\label{Theorem Rg and products}
$\mathcal{R}_{\widetilde{g}},\mathcal{S}_{\widetilde{g}}$ are compatible with products, meaning $\boxed{\mathcal{R}_{\widetilde{g}} (a \cdot b) = (\mathcal{R}_{\widetilde{g}} a) \cdot b}$
$\mathcal{R}_{\widetilde{g}},\mathcal{S}_{\widetilde{g}}$ are pair-of-pants product by $\mathcal{R}_{\widetilde{g}}(1),\mathcal{S}_{\widetilde{g}}(1)$ and $r_{\widetilde{g}}$ is quantum product by $r_{\widetilde{g}}(1)$.
\end{theorem}
\begin{proof} 
Recall \cite{Ritter3} the product $HF^*(H_k) \otimes HF^*(H_{\ell}) \to HF^*(H_{k+\ell})$ counts isolated solutions $u: S \to M$ to the equation $(du-X_{H_1}\otimes \beta)^{0,1}=0$, where $S$ is a pair-of-pants surface (so diffeomorphic to $\R\times S^1 \setminus (0,0)$) and $\beta$ is a $1$-form on $S$ which equals $(k+\ell)\, dt,\ell\, dt,k\, dt$ near the three punctures $-\infty,(0,0),+\infty$ and satisfies $d\beta=0$, where a cylindrical parametrization $(s,t)$ has been chosen near $(0,0)$ (say $e(s,t)=(\frac{1}{4}e^{2\pi s}\cos 2 \pi t,\frac{1}{4}e^{2\pi s} \sin 2\pi t)$, where $s\in (-\infty,0]$).

This is similar to the closed setup \cite[Sec.6]{Seidel3}, except we do not homotope $\beta$ to zero near $s=\pm 2$ (which would contradict $d\beta= 0$, and would cause compactness problems). The only difference with the definition of product in \cite{Ritter3} is that the Novikov ring in our current setup is larger, so we need to specify what it means for $u$ to converge to $c^-,c_0,c^+\in \widetilde{\mathcal{L}_0M}$. As in \cite[Def.6.1]{Seidel3}, this means: if $c_0=(v_0,x_0)$, then $u\# v_0$ (gluing $v_0$ onto $x_0=\lim_{s \to -\infty} u \circ e$), viewed as a path of loops, must lift to a path in $\widetilde{\mathcal{L}_0M}$ with limits $c^-,c^+$.

On homology $\mathcal{S}_{\widetilde{g}}$ only depends on $(g_t,\widetilde{g}_t)$ up to homotopy by Corollary  \ref{Corollary properties of Sg}. So we can arrange (as in \cite[Prop.6.3]{Seidel3}) that $g_t$ is the identity for $t\in [-\frac{1}{4},\frac{1}{4}]$, so we can ensure $K^g_t=0$ there.

\emph{Technical Remark. A reparametrization $g_{\alpha(t)}$ has Hamiltonian $L_t = \alpha'(t) K_{\alpha(t)}^g$. Since $g_0=\mathrm{id}$, taking $\alpha'\geq 0$, $\alpha=0$ for $t\in [0,\frac{1}{4}]$ and $\alpha=1$ for $t\in [\frac{3}{4},1]$, ensures $L_t=0$ for $t\in [-\frac{1}{4},\frac{1}{4}]$. Although the slope of $L_t$ varies at infinity, the homotopy argument of Corollary \ref{Corollary properties of Sg} still holds by \cite[Sec.5]{Seidel3}, since throughout the homotopy we can ensure that the slopes are non-negative.
}

Observe that near the puncture $(0,0)$, where we use a different $t$ coordinate than the global $t\in S^1$ of $S\cong \R \times S^1 \setminus (0,0)$, the data $g^*H_1,g^*J$ is the same as $H_1,J$ since $K^g(t,\cdot)=0$ there. Therefore, as in \cite[Lemma 6.4]{Seidel3}, the following moduli spaces of pair-of-pants solutions biject:
$$
\mathcal{M}_{(S,\beta)}(c^-,c_0,c^+;g^*H_1,g^*J) \cong \mathcal{M}_{(S,\beta)}(\widetilde{g}c^-,c_0,\widetilde{g}c^+;H_1,J), u(s,t) \mapsto g_t(u(s,t)).
$$
We deduce $\mathcal{S}_{\widetilde{g}}(a \cdot b) = (\mathcal{S}_{\widetilde{g}}a )\cdot b$ for any $g\in G$. Since continuation maps preserve the product structure \cite{Ritter3}, the same holds for $\mathcal{R}_{\widetilde{g}}$ (using the Remark after Lemma \ref{Lemma HF H0 and QH are iso as rings}).
So, using the unit $1=\psi^-(1)$ of Corollary \ref{Corollary unit for SH}: $\mathcal{R}_{\widetilde{g}}(a)=\mathcal{R}_{\widetilde{g}}(1)\cdot a$, $\mathcal{S}_{\widetilde{g}}(a)=\mathcal{S}_{\widetilde{g}}(1)\cdot a$. 

Recall $r_{\widetilde{g}}=\psi^+\mathcal{R}_{\widetilde{g}} \psi^-$. So $
r_{\widetilde{g}}(y) = \psi^+[ \mathcal{R}_{\widetilde{g}}(1)\cdot \psi^- (y)]$.
 By Corollary \ref{Corollary unit for SH}: $\mathcal{R}_{\widetilde{g}}(1)=\psi^-r_{\widetilde{g}}(1)$. By Lemma \ref{Lemma HF H0 and QH are iso as rings}: $\psi^+[\psi^-(r_{\widetilde{g}}(1)) \cdot \psi^-(y)]= r_{\widetilde{g}}(1) \ast y$, using $\psi^+\psi^-=\mathrm{id}$.
\end{proof}

 Abbreviate $c=r_{\widetilde{g}}(1)$. Identifying $QH^*(M)\equiv HF^*(H_k)$ the continuation $\varphi^k:HF^*(H_0) \to HF^*(H_k)$ is quantum cup product by $c_Q^k$ (quantum powers).
 By Lemma \ref{Lemma HF H0 and QH are iso as rings}, $\alpha_0 * \beta_0=\alpha_0\cdot \beta_0$ via $QH^*(M)\equiv HF^*(H_0)$, but not on $HF^*(H_k)$ unless one correctly interprets $k$.  For general reasons \cite{Ritter3} $\cdot$ will not preserve $H_k$:
$$HF^*(H_k)\otimes HF^*(H_{\ell}) \to HF^*(H_{k+{\ell}}),\; \alpha_k \otimes \beta_{\ell} \mapsto \alpha_k \cdot \beta_{\ell}.$$
 This can be elucidated in our case. Suppose $\alpha_k, \beta_{\ell}$ have $HF(H_0)$ representatives: $\alpha_k=\varphi^k(\alpha_0), \beta_{\ell} =\varphi^{\ell}(\beta_0)$. Since continuations are compatible with products \cite{Ritter3}, 
$$\varphi^{k+{\ell}}(\alpha_0*\beta_0)=\varphi^{k+{\ell}}(\alpha_0\cdot\beta_0) = \varphi^k(\alpha_0)\cdot \varphi^{\ell}(\beta_0) = \alpha_k\cdot \beta_{\ell},$$
which proves that
$
\alpha_k\cdot \beta_{\ell} =  c^{k+{\ell}}_Q* \alpha_0 * \beta_0 = (c^k_Q*\alpha_0) * (c^{\ell}_Q*\beta_0) = \alpha_k * \beta_{\ell}. 
$
%
%
\section{Pseudoholomorphic sections}
\label{Section Pseudoholomorphic sections}
%
%
\subsection{Space of sections $\mathcal{S}(j,\hat{J})$}
\label{Subsection Space of sections}
We briefly recall some definitions [Sec. 7,\cite{Seidel3}].

Let $(\pi: E \to S^2,\Omega)$ be a \emph{symplectic fibre bundle} with fibre $(M,\omega)$, meaning: $\Omega_z$ is a symplectic form for the fibre $E_z$ over $z\in S^2$, smoothly varying in $z$. It is understood that we fix an isomorphism $i:(M,\omega) \to (E_{z_0},\Omega_{z_0})$ where $z_0\in S^2$ is the South pole (view $S^2=D^+\cup_{S^1} D^-$ as a union of two discs, $z_0=$\,centre$(D^-)$).

Let $\mathcal{J}(E,\Omega)$ be the space of $J=(J_z)_{z\in S^2}$ (smooth in $z$), where $J_z$ is a conical structure on the fibre $(E_z,\Omega_z)$ (see \ref{Subsection Conical symplectic manifolds}). Fix a positively oriented complex structure $j$ on $S^2$. Call an almost complex structure $\hat{J}$ on $E$ \emph{compatible} with $(j,J)$ if 
$d\pi \circ \hat{J} = j \circ d\pi$
 and $\hat{J}$ restricts to $J$ fibrewise. Denote $\hat{\mathcal{J}}(j,J)$ the space of compatible $\hat{J}$.

\begin{definition*}
 For $j,J,\hat{J}\in\hat{\mathcal{J}}(j,J)$ as above, denote $\mathcal{S}(j,\hat{J})$ the space of $(j,\hat{J})$-holomorphic sections, meaning all $s:S^2 \to E$ with
$$
ds \circ j = \hat{J}\circ ds.
$$ 
\end{definition*}

\begin{definition*}
 Call $(E,\Omega)$ \emph{Hamiltonian} (symplectic fibre bundle) if there is a closed two-form $\widetilde{\Omega}$ on $E$ restricting to $\Omega_z$ fibrewise.
\end{definition*}

\begin{definition*}
For $(E,\Omega,\widetilde{\Omega})$ Hamiltonian, two sections $s,s'$ are $\Gamma$-equivalent if $\widetilde{\Omega}(s)=\widetilde{\Omega}(s')$ and $c_1(TE^v,\Omega)(s)=c_1(TE^v,\Omega)(s')$, where $TE^v =\ker d\pi$. Denote $\mathcal{S}(j,\hat{J},S)\subset \mathcal{S}(j,\hat{J})$ the subspace of sections in the $\Gamma$-equivalence class $S$. 
\end{definition*}

The equivalence classes do not depend on the choice of $\widetilde{\Omega}$, but only on $\Omega$: a difference of two sections $S-S' \in \pi_2(E)$ maps to $[S^2]-[S^2]=0\in \pi_2(S^2)$ via the fibration $E\to S^2$, so $S-S' \in \textrm{im}(\pi_2(M)\to \pi_2(E))$, so the $\Omega$-value on this fibre class determines whether $\widetilde{\Omega}(S-S')$ is zero or not. 

By \cite[Lemma 2.10]{Seidel3} for $S,S'$, there is a unique $\gamma \in \Gamma$ such that the $\widetilde{\Omega}$ values on $S,S'$ differ by $\omega(\gamma)$, and the $c_1(TE^v,\Omega)$ values differ by $c_1(TM,\omega)(\gamma)$. Conversely, given $S$, $\gamma \in \Gamma$ there is a unique class denoted $S'=S+\gamma$ for which this holds.
%
\subsection{Hamiltonian fibration}
\label{Subsection Hamiltonian fibration}
%
From now on, we view $S^2$ as $D^+\cup_{S^1} D^-$ and we will use the coordinates $z=\exp(s+it)$ on $D^+=\{z\in \C: |z|\leq 1 \}$ and the complex structure $j\partial_s = \partial_t$. Near $\partial D^{+}$ these coordinates lie in $(s,t)\in (-\epsilon,0]\times S^1$ and we can extend these coordinates to $D^-$ near $\partial D^{-}$ via $(s,t)\in [0,\epsilon) \times S^1$.
\begin{definition*}
Given $g\in \textrm{Ham}_{\ell}(M,\omega)$ generated by $(K^g_t)_{t\in S^1}$, define the symplectic fibre bundle $(\pi_g:E_g \to S^2,\Omega_g)$ by the clutching construction
$$
\begin{array}{l}
 E_g = (D^+\times M) \cup_{\phi^g} (D^-\times M)\\
 \phi^g: \partial D^+ \times M \to \partial D^- \times M, \; \phi^g(t,y)  = (t,g_t(y)) 
\end{array}
$$
with form $\Omega_g=\omega^{\pm}$ on the fibres (the pull-backs of $\omega$ from $M$ to $D^{\pm} \times M$).
\end{definition*}
\noindent Let $H^{\pm}: D^{\pm} \times M \to \R$ be Hamiltonians which:
\begin{enumerate}
 \item  vanish near the centres of $D^{\pm}$; 
 \item only depend on $t\in S^1$ near $\partial D^{\pm}$: $H^{\pm}(s+it,y)=H^{\pm}_t(y)$;

 \item and which satisfy the gluing condition on $\partial D^{\pm} \times M$:
$$
H^+_t(y) = g^*H^-_t(y) = H^-_t(g_t y) - K^g_t(g_t y).
$$
\end{enumerate}
\begin{definition}\label{Definition Monotone H}
 We call $H$ \emph{monotone} if $\partial_s (h^+)'\leq 0$ and $\partial_s (h^-)'\leq 0$, where $H^{\pm}=h^{\pm}(R)$ at infinity and where $s$ is the coordinate determined by the parametrizations
$$
\begin{array}{ll} 
(-\infty,0]\times S^1 \to D^+\setminus z_{\infty}\equiv \{z\in \C\setminus 0: |z|\leq 1 \},& (s,t)\mapsto e^{2\pi (s+it)}\\[0mm]
[0,+\infty)\times S^1 \to D^-\setminus z_0\equiv \{z\in \C: |z|\geq 1 \},& (s,t)\mapsto e^{2\pi (s+it)}.
\end{array}
$$
\end{definition}
\noindent Define a one-form $\tau^{\pm}$ on $D^{\pm}\times M$ and a closed $2$-form $\widetilde{\Omega}$ on $E_g$ by
$$\boxed{
\begin{array}{l}
\tau^{\pm} = H^{\pm}\, dt\\
 \widetilde{\Omega}|_{D^{\pm} \times M} = \omega^{\pm} - d\tau^{\pm} = \omega^{\pm} - dH^{\pm} \wedge dt - \partial_s H^{\pm} \, ds \wedge dt.
\end{array}}
$$
Note $\widetilde{\Omega}$ glues correctly since $g_t^* \omega = \omega$ and since
$$((\phi^g)^*\omega^-)(\vec{v},\partial_t) = \omega^-(d\phi^g \cdot \vec{v},X_{K_g}\circ g_t)  = d(K^g_t\circ g_t\,dt)\cdot(\vec{v},\partial_t).
$$
\begin{definition}\label{Definition c1 of Sg in terms of I}
 Recall \cite[Lemma 2.12]{Seidel3}, that $E_g$ has a continuous section $s_{\widetilde{g}}$ built as follows. Pick any $c\in \widetilde{\mathcal{L}M}$, and pick representatives $(v,x),(v',x')$ of $c,\widetilde{g}(c)$. Then glue:
$$
\begin{array}{l}
 s^+_{\widetilde{g}}: D^+ \to D^+ \times M, s^+_{\widetilde{g}}(z) = (z,v(z))\\
 s^-_{\widetilde{g}}: D^- \to D^- \times M, s^-_{\widetilde{g}}(z) = (z,\overline{v'}(z))\\
\end{array}
$$
where $\overline{v'}: D^- \cong D^2 \to M$ involves an orientation-reversing identification $\cong$. 

The $\Gamma$-equivalence class $S_{\widetilde{g}}$ of $s_{\widetilde{g}}$ does not depend on the choices $c,v,v'$, and  $$\boxed{I(\widetilde{g}) = -c_1(TE^v_g, \Omega_g)(s_{\widetilde{g}})}.$$
\end{definition}
%
\subsection{Admissible almost complex structures $\hat{J}$}
\label{Subsection Admissible almost cx str}
%
\begin{remark*}
To build a symplectic form on the total space of a symplectic fibration (Thurston's method) one modifies the symplectic form by a pull-back of a large multiple of a symplectic form on the base to achieve non-degeneracy in the horizontal distribution. This fails in our case because the fibres are non-compact and the given symplectic form grows like $R$ at infinity, so such pull-backs cannot dominate. Thus \cite[Lemma 7.4]{Seidel3} fails in our setup. The remedy is to require that $\hat{J}$ has a special form at infinity, depending on the Hamiltonian.
\end{remark*}
\begin{definition}\label{Definition admissible J hat}
Call $(j,J,\hat{J})$ admissible if $J_z$ is $(\Omega_g)_z$-compatible, $\hat{J}\in \hat{\mathcal{J}}(j,J)$ and such that for large $R$ they have the form
$$
\hat{J}_{(z,y)} = \left(\begin{array}{ll} j & 0 \\ \nu_{(z,y)}\circ j & J_z  \end{array} \right)
= \left(\begin{array}{ll} j & 0 \\ ds \otimes X_H - dt \otimes J_zX_H & J_z  \end{array} \right)
$$
where $\nu_{(z,y)}: T_z S^2 \to T_y M$ is the $(j,J_z)$-antilinear homomorphism given by
$$
\nu_{(z,y)} = ds \otimes J_z X_H(z,y) + dt \otimes X_H(z,y).
$$
\end{definition}

\begin{remark}
 \label{Remark Floer continuation gives admissible Jhat} $($due to Gromov$)$
The $\hat{J}_z=\left(\begin{smallmatrix} j & 0 \\ \nu_{(z,y)}\circ j & J_z  \end{smallmatrix} \right)$ arise from turning the Floer continuation $\partial_s u + J_z(u)(\partial_t u - X_H(z,u))=0$, for $z=(s,t)\in \R\times S^1$, into $$du + J \circ du \circ j = \nu,$$ and finally into $ds\circ j = \hat{J}\circ ds$ for 
$s: \R\times S^1 \to \R\times S^1 \times M,\; s(z) = (z,u(z)).$
\end{remark}

\begin{remark*}
 Our $H$, $\hat{J}$ correspond in the notation of \cite[Sec.8.1\,(p.243)]{McDuff-Salamon2} to $G$ and $\widetilde{J}_{G\,dt}= \widetilde{J}_{\tau}$ $($their Hamiltonian vector fields are opposite to ours$)$. The curvature \cite[Sec.8.1]{McDuff-Salamon2} is $F_{\tau} \, \mathrm{vol}_{S^2} = \partial_s H\, \mathrm{vol}_{S^2}$. The $\widetilde{\Omega}$-horizontal distribution over $D^+$ is $\mathrm{Hor} = \{ \xi \in T(D^+ \times M): \widetilde{\Omega}(\ker d\pi,\xi)=0\} = \mathrm{span}\, \{ \partial_s, \partial_t + X_{H}\}$. 
$\hat{J}$ preserves $\mathrm{Hor}$.
\end{remark*}

\begin{lemma}\label{Lemma Symplectic form for admissible J}
For admissible $(j,J,\hat{J})$, if $\partial_s H\leq 0$ then for large $c\in \R$, $\hat{J}$ is $\widetilde{\Omega}+\pi^*(c\cdot \mathrm{vol}_{S^2})$-compatible and so $\widetilde{\Omega}+\pi^*(c\cdot \mathrm{vol}_{S^2})$ is symplectic. Without the condition $\partial_s H \leq 0$, this still holds provided we assume $\partial_s H$ is bounded above.
\end{lemma}
\begin{proof}
At infinity, a computation shows that:
$$
\begin{array}{l}
\widetilde{\Omega}(a\partial_s+b\partial_t+\vec{m},a'\partial_s+b'\partial_t+\vec{m}') = \omega(\vec{m}-bX_H,\vec{m}'-b'X_H) - (ab'-a'b)\partial_s H\\
\widetilde{\Omega}(a\partial_s+b\partial_t+\vec{m}, \hat{J}(a\partial_s+b\partial_t+\vec{m})) =
\omega(\vec{m}-bX_H,J(\vec{m}-bX_H)) -(a^2+b^2)\partial_s H
\end{array}
$$
adding $\pi^*(c\cdot ds \wedge dt)$ to $\widetilde{\Omega}$ contributes an extra $c\cdot (a^2+b^2)$. Since $J$ is $\omega$-compatible, this proves the claim at infinity for $c\geq \partial_s H$.

In the compact region where $\hat{J} = \left( \begin{smallmatrix} j & 0 \\ \nu\circ j & J \end{smallmatrix} \right)$ does not have $\nu$ in the special form of Definition \ref{Definition admissible J hat}, we need positivity of:
$$
\omega(\vec{m},J\vec{m}) + \omega(\vec{m}-bX_H,\nu j (a\partial_s + b\partial_t)) 
- \omega(bX_H,J\vec{m})
-\omega(\vec{m},aX_H) +(a^2+b^2)(c-\partial_s H).
$$
Abbreviate $|\vec{m}|^2=\omega(\vec{m},J\vec{m})$. By rescaling, assume $a^2+b^2+|\vec{m}|^2=1$.
If $a^2+b^2 \ll |\vec{m}|^2$, then the first term dominates (on the compact region all terms are bounded). Otherwise, we make the last term dominate by making $c\gg 0$.
\end{proof}

\begin{example}\label{Example non-monotone homotopy H}
 Let $H_0 = \delta(R) R + \textrm{constant}$ for $R\gg 0$, such that $\delta(R)R>0$ is $C^2$-bounded and concave. By Lemma \ref{Lemma continuation isos if no in between Reebs}, $HF^*(H_0)$ is the same as if $\delta(R)<($min Reeb period$)$ was constant. The advantage now is that one can find (non-monotone) homotopies $H_s$ from $0$ to $H_0$ with $\partial_s H_s$ bounded, which is crucial for Lemmas \ref{Lemma Symplectic form for admissible J} and \ref{Lemma Monotonicity Lemma} and Theorem \ref{Theorem psi+ and psi-}.
\end{example}

\begin{lemma}\label{Lemma compatibility trick}
 Suppose $\partial_s H \leq 0$. For any $(j,\hat{J})$-holomorphic section $u: S^2 \to E_g$,  $u^*\widetilde{\Omega}\geq 0$ at all points $z$ for which $u(z)$ lies in the region where $\hat{J}$ has the special form as in Definition \ref{Definition admissible J hat}.
\end{lemma}
\begin{proof}
 Locally $u(z)=(z,u^{\pm})\in D^{\pm}\times M$. Using the proof of Lemma \ref{Lemma Symplectic form for admissible J}:
$$
\begin{array}{lll}
 \dfrac{u^*\widetilde{\Omega}}{ds\wedge dt} &=& \widetilde{\Omega}(du \circ \partial_s,du \circ j\partial_s)
= \widetilde{\Omega}(\partial_s u, \hat{J}\partial_s u)\\
&=& \widetilde{\Omega}(1\cdot \partial_s + \partial_s u^{\pm}, \hat{J}(1\cdot \partial_s + \partial_s u^{\pm})) 
= \omega(\partial_s u^{\pm},J\partial_s u^{\pm}) -\partial_s H
\geq 0. \qedhere
\end{array}
$$
\end{proof}
%
%
\subsection{Compactness result for $\mathbf{\mathcal{S}(j,\hat{J})}$}
\label{Subsection Compactness for sections}
%
By Lemma \ref{Lemma Symplectic form for admissible J}, $\widetilde{\Omega}+\pi^*(\sigma)$ is symplectic on $E=E_g$ for some form $\sigma$ on $S^2$, and admissible $\hat{J}$ are $\widetilde{\Omega}+\pi^*(\sigma)$-compatible. 

\begin{lemma}\label{Lemma Compactness for space of sections}
Under the assumptions of Lemma \ref{Lemma Symplectic form for admissible J}, and $J$ generic, then for every $C\in \R$, and any given compact $D\subset E_{z_0}$, only finitely many $\Gamma$-equivalence classes $S$ have $\widetilde{\Omega}(S)\leq C$ with $\mathcal{S}(j,\hat{J},S)$ containing a section intersecting $D$ over $z_0$.
\end{lemma}
\begin{proof}
Consider a sequence $s_n \in \mathcal{S}(j,\hat{J},S)$ with $\widetilde{\Omega}(s_n)\leq C$, $c_1(TE^v,\Omega)(s_n)\leq c$ and with $s_n(z_0) \to y \in E_{z_0}$. 

 Three out of four possible failures of sequential compactness are analogous to the case of closed manifolds \cite[Lemmas 7.5,7.6]{Seidel3}. These three failures would imply the existence of a holomorphic section $s\in \mathcal{S}(j,\hat{J})$, which respectively: (1) passes through $y$ but $c_1(TE^v,\Omega)(s)<c$; or (2) passes through $y$ and $c_1(TE^v,\Omega)(s)=c$ but a holomorphic bubble appears in some fibre $E_z$ and intersects $s(z)$; or (3) a cusp-curve of total Chern number $\leq c-c_1(TE^v,\Omega)(s)$ appears in $E_{z_0}$ whose initial marked point lands at $s(z_0)$ and whose last marked point lands at $y$.

Failure (4): the $s_n$ are unbounded in the fibre direction. This cannot happen by Lemma \ref{Lemma key to compactness}.
\end{proof}
In all situations, except for the construction of $\psi^+$ in Section \ref{Subsection Construction of the psi- psi+ maps}, we will only work with monotone $H$ (Definition \ref{Definition Monotone H}). Equivalently, this corresponds via Remark \ref{Remark Floer continuation gives admissible Jhat} to the usual assumption $\partial_s h_z'\leq 0$ which ensures that the maximum principle for Floer continuation solutions holds. We recall the argument below.
\begin{lemma}[Maximum principle]\label{Lemma Maximum principle}
Assume that for $R\geq R_0$ the following hold:  
$H=h_z(R)$, $\partial_s h'_z \leq 0$ and $\hat{J}$ has the form as in Definition \ref{Definition admissible J hat}. Then all $(j,\hat{J})$ pseudo-holomorphic sections $s:S^2 \to E_g$ lie in the region $R\leq R_0$.
\end{lemma}
\begin{proof} By Remark \ref{Remark Floer continuation gives admissible Jhat}, $s$ has the form $u: \R \times S^1 \to \R \times S^1 \times M$ (defined on a subset of $\R \times S^1$), with $du\circ j = J\circ du + \nu \circ j$. Let $\rho=R\circ u$. 
Since $dR \circ J = - \theta = -R\alpha$,
$$
d\rho \circ j = dR (J\circ du + \nu\circ j)
=  -u^*\theta + dt \otimes \theta(X_H)
$$
using $dR(X_H)=dR(h'(R)\mathcal{R})=0$. Arguing as for the Maximum Principle in \cite{Ritter3},
$$
\Delta \rho\, ds \wedge dt = -d(d\rho\circ j)= \frac{1}{2}\|du - X_H \otimes dt\|^2 + \frac{h' d\rho\wedge dt - d(\rho h' dt)}{ds \wedge dt}
$$
so $(\Delta \rho + \textrm{first order terms in }\rho) \geq -\rho (\partial_s h')$. So the maximum principle for $\rho$ applies provided $\partial_s h' \leq 0$.
\end{proof}
In the construction of $\psi^+$ in Section \ref{Subsection Construction of the psi- psi+ maps}, we need to allow a non-monotone Hamiltonian as described in Example \ref{Example non-monotone homotopy H}. To prevent Floer solutions from escaping to infinity we appeal to the following monotonicity lemma.
\begin{lemma}[Monotonicity Lemma]\label{Lemma Monotonicity Lemma}
 Suppose $H$ is $C^2$-bounded, in particular $\partial_s H$ is bounded above, and $(j,J,\hat{J})$ is admissible. Then there is a constant $C>0$, such that for any $(j,\hat{J})$ pseudo-holomorphic disc $s: D \subset S^2 \to E_g$ and boundary $s(\partial D)$ lying in the boundary of a ball of radius $\epsilon$ with centre intersecting $s(D)$, the energy $E(s) = \int_D \| du \|_{\hat{J}}^2\, ds \wedge dt$ calculated with respect to the metric $(\widetilde{\Omega}+\pi^*\sigma)(\cdot, \hat{J}\cdot)$ (Lemma \ref{Lemma Symplectic form for admissible J}) is at least $C \epsilon^2$.
\end{lemma}
\begin{proof}
 This is a standard consequence of the isoperimetric inequality, see \cite[Sec.V.4]{Audin-Lafontaine}. This uses the fact that $M$, and hence $E_g$, is geometrically bounded. In particular, it uses that $X_H=h'_z(R)\mathcal{R}$ is $C^1$-bounded at infinity since $H$ is $C^2$-bounded, and $\hat{J}$ is prescribed in terms of $j,J,X_H$ at infinity by Definition \ref{Definition admissible J hat}. In particular, the condition that $\partial_s H$ is bounded above is required for Lemma \ref{Lemma Symplectic form for admissible J} to hold.
\end{proof}
\begin{lemma}\label{Lemma key to compactness}
Let $(j,J,\hat{J})$ be admissible. Suppose $H:E_g \to \R$ is monotone (Definition \ref{Definition Monotone H}) except possibly over a subset of the base $S^2=D^+\cup_{S^1} D^-$ where the assumption of Lemma \ref{Lemma Monotonicity Lemma} holds. Then all sections $s \in \mathcal{S}(j,\hat{J})$ which intersect a compact domain $D\subset E_{z_0}$ with $\widetilde{\Omega}(s)\leq C'$ are contained in a compact region of $E_g$ determined by $C',D$.
\end{lemma}
\begin{proof}
If $H$ is monotone everywhere, then this is a consequence of the maximum principle (Lemma \ref{Lemma Maximum principle}). Consider first the section restricted to the subset of $S^2$ where $H$ is non-monotone. Then, by Lemma \ref{Lemma Monotonicity Lemma} and the energy estimate $(\widetilde{\Omega}+\pi^*\sigma)(s) \leq C + \sigma[S^2]$, the section is forced to lie in a compact region of $E_g$ determined a priori by $C'$, $D$ (and the constant $C$ from Lemma \ref{Lemma Monotonicity Lemma}). On the remaining region of $S^2$ where $H$ is monotone, one can apply the proof of the maximum principle to deduce that the maximal value of $R$ occurs at the boundary of that region of $S^2$. But there, the section is constrained to lie in the compact region of $E_g$ determined previously, so the $R$-coordinate is bounded a priori in terms of $C',D$.
\end{proof}
%
%
%
%
\subsection{Transversality for $\mathcal{S}(j,\hat{J})$}
\label{Subsection Transversality for sections}
%
%
\begin{lemma}\label{Lemma dim of space of sections}
 After a small generic perturbation of $(J,H)$, for admissible $(j,J,\hat{J})$ the moduli space $\mathcal{S}(j,\hat{J})$ is a smooth manifold of dimension
$$
d(s) = (\dim_{\R} \textrm{of }\mathcal{S}(j,\hat{J})\textrm{ near }s)  = 2\dim_{\C}M+2c_1(TE_g^v,\Omega)(s),
$$
where $TE^v = \ker d\pi$ (abbreviating $E=E_g$), and the evaluation maps
$$
\begin{array}{l}
\mathrm{ev}: S^2 \times \mathcal{S}(j,\hat{J}) \to E,\; \mathrm{ev}(z,s) =s(z)\\
\mathrm{ev}_{z_0}: \mathcal{S}(j,\hat{J}) \to M,\; \mathrm{ev}_{z_0}(z,s) =i^{-1}(s(z_0))
\end{array}
$$
are transverse respectively to 
$$
\begin{array}{l}
\eta: \mathcal{M}_0^s(J) \times_{PSL(2,\C)} \CP^1 \to E, \;\eta(z,w,x) =w(x)\\
\eta_1: \mathcal{C}_{r,k}(J) \to M, \;\eta_1(w_1,\ldots,w_r,t_1,\ldots,t_r,t_1',\ldots,t_r') = w_1(t_1'),
\end{array}
$$
so the evaluation 
$
\mathrm{ev}_{z_0}: \mathcal{S}(j,\hat{J},S) \to M
$
is a pseudo-cycle of dimension $d(S)$.

\textbf{Notation:} $\mathbf{\mathcal{M}_k^s(J)}=\{(z,w)\in S^2 \times C^{\infty}(\CP^1,E): w$ simple $J_z$-holomorphic curve in $E_z$ with $c_1(TE_z,\Omega_z)(w)=k\}$. This is empty for $k<0$ and is a $\dim E$-manifold for $k=0$ whose image under $\eta$ has codim $= 4$ (uses weak$^+$-monotonicity). 

$\mathbf{\mathcal{C}_{r,k}(J)}$ is the $(2n+2k-2r)$-manifold of simple $J$-holomorphic cusp-curves with $r\geq 1$ components of total Chern number $k$ quotiented by the $PSL(2,\C)^r$ action, where $w_i(t_i)=w_{i+1}(t_{i+1}')$ are the nodes for $i=1,\ldots,r-1$.
\end{lemma}
This Lemma is the analogue of \cite[Prop.7.3]{Seidel3}, except at infinity we perturb $\hat{J}$ in a controlled way by perturbing $H$ (thus preserving admissibility). The proof of transversality using perturbations of $H$ is in Sec. 8.3 \& 8.4 of \cite{McDuff-Salamon2}. The proof that $ev_{z_0}$ is a pseudo-cycle then follows by Lemma \ref{Lemma Compactness for space of sections}, just like in \cite[Prop.7.7]{Seidel3}. Indeed, the proof of Lemma \ref{Lemma Compactness for space of sections} describes how $\mathrm{ev}_{z_0}(\mathcal{S}(j,\hat{J},S))$ can be compactified by countably many images of manifolds (since we only care about the image, we may assume the holomorphic bubbles and cusp-curves that we described are simple). By a dimension count, using the above transverseness claims about $\eta,\eta_1$, one shows that these additional manifolds have dimension $\leq d(S) -2$.
\begin{remark}
 $\mathcal{S}(j,\hat{J},S)$ depends on $H$ in so far as $\hat{J}$ depends on $H$ (admissibility), but equivalence classes $S$ are independent of this choice by \ref{Subsection Space of sections} (they depend on $\Omega_g$).
\end{remark}
%
\subsection{Construction of the $\psi^+$ and $\psi^-$ maps}
\label{Subsection Construction of the psi- psi+ maps}
%
\begin{theorem}\label{Theorem psi+ and psi-}
 Let $H_0$ be a Hamiltonian on $M$ which at infinity equals $h(R)$ (non-linear) with slopes $0<h'<
(\textrm{min Reeb period})$, $h''<0$ and $h'\to 0$ fast enough so that $H_0$ is $C^2$-bounded. Then there are chain maps 
$$
\psi^+: CF^*(H_0) \to QC_{2n-*}^{lf}(M) \qquad \psi^-: QC_{2n-*}^{lf}(M) \to CF^*(H_0)
$$
homotopy inverse to each other, where $2n=\dim_{\R} M$. Via Poincar\'e duality:
$$
\psi^+: CF^*(H_0) \to QC^*(M) \qquad \psi^-: QC^*(M) \to CF^*(H_0)
$$
\end{theorem}
\begin{proof}
 $\psi^+$ will count $(j,\hat{J}^+=\hat{J}|_{D^+})$-holomorphic sections $s^+:D^+ \to D^+ \times M$, for $(j,J,\hat{J})$ admissible, where $H^+=H_0$  on $\partial D^+$ and $H^+=0$ at the centre of the disc (compare \ref{Subsection Hamiltonian fibration}). We can ensure that $\partial_s H_z$ is bounded above since $H_0$ is bounded.

Let $c=(v,x) \in \widetilde{\mathcal{L}_0 M}$, where $x$ is a $1$-orbit of $X_{H_0}$. Denote $\mathcal{M}^+=\mathcal{M}^+(c;H^+,\hat{J}^+)$ the moduli space of such sections $s^+$ with
$(D^2\cong D^+ \stackrel{s^+}{\to} M)=c \in \widetilde{\mathcal{L}_0 M},$
 where $\cong$ is the orientation-preserving identification. These moduli spaces are defined in the closed setup in \cite[Sec.8]{Seidel3}. For generic $(\hat{J},H)$, $\mathcal{M}^+$ is smooth and
$$
\dim \mathcal{M}^+(c;H^+,\hat{J}^+) = 2n - \mu_{H_0}(c),
$$
(see \ref{Subsection grading of symplectic homology} for gradings). The evaluation at the centre of the disc $\mathrm{ev}_{z_{\infty}}: \mathcal{M}^+ \to M, u \mapsto u(z_{\infty})$ is a locally finite pseudo-cycle of that dimension. To ensure the locally finite condition, we use Lemma \ref{Lemma Monotonicity Lemma} and the a priori energy estimate for $u$:
$$
\begin{array}{lll}
E(u) &=& \int_{D^+} \|u\|_{(\widetilde{\Omega}+\pi^*\sigma)(\cdot,\hat{J}^+\cdot)}^2\, ds \wedge dt \\[1mm]
&=& (\widetilde{\Omega}+\pi^*\sigma)[c]\\[1mm]&=& \int_{D^+} u^*\omega + \int_{D^+} u^*d(-H^+dt) + \int_{D^+} u^*(\pi^*\sigma)\\[1mm]
&=& \omega[c]-\int_{S^1} H^+(x)\, dt +\sigma[D^+].
\end{array}
$$
%
%
Indeed, this estimate and Lemma \ref{Lemma Monotonicity Lemma} imply that all $u\in \mathcal{M}^+ (c;H^+,\hat{J}^+)$ which intersect a given compact $C'$ of $M$ must lie in a compact subset $C''$ of $M$ determined by $C'$. But then a standard Gromov compactness argument implies the compactness up to breaking of the subset of all $u\in \mathcal{M}^+ (c;H^+,\hat{J}^+)$ intersecting $C'$. 

Define $\psi^+$ by extending linearly the map defined on generators by
$$\boxed{
\begin{array}{rcl}
\psi^+: CF^*(H_0) & \to & QC_{2n-*}^{lf}(M),\\
\psi^+(c) & = & \displaystyle \sum_{\gamma \in \Gamma}\;\; (\mathrm{ev}_{z_{\infty}})_*
[\mathcal{M}^+(\gamma \cdot c;H^+,\hat{J}^+)] \otimes <\gamma>
\end{array}
}
$$
As $\dim \mathcal{M}^+ (\gamma \cdot c;H^+,\hat{J}^+) = 2n- |\gamma \cdot c| = 2n-|c|-|\gamma|$, the right hand side above has degree $2n-|c|$. For $c=(v,x)$, the energy of $u \in \mathcal{M}^+ (\gamma \cdot c;H^+,\hat{J}^+)$ is 
$$\textstyle E(u)=\omega(\gamma)+\omega[c]-\int_{S^1} H^+(x)\, dt +\sigma[D^+]$$
So for fixed $c$ but varying $\gamma$, the $\omega(\gamma)$ must grow to $\infty$ if such energies were to grow to $\infty$. So $\psi^+$ is well-defined.

Similarly define $\mathcal{M}^-(c;H^-,\hat{J}^-)$ requiring
$(D^2 \, \cong \, D^- \stackrel{s^-}{\to} M)=c \in \widetilde{\mathcal{L}_0 M},
$
where $\cong$ is orientation-reversing. Then $\dim\, \mathcal{M}^-(c;H^-,\hat{J}^-)=\mu_{H_0}(c)$. Since $H^-$ is a homotopy from $H_0$ to $0$ we can choose it to be monotone: $\partial_s H_z^- \leq 0$. So by Lemma \ref{Lemma Maximum principle} we obtain a (finite) pseudo-cycle $\mathrm{ev}_{z_0}:\mathcal{M}^-(c;H^-,\hat{J}^-) \to M$.
$$
\boxed{
\begin{array}{rcl}
\psi^-: QC_{2n-*}^{lf}(M) & \to & CF^*(H_0),\\
\psi^-(\alpha) & = & \displaystyle \hspace{-4ex}\sum_{\dim \mathcal{M}^-(c;H^-,\hat{J}^-) + \dim \alpha = 2n}\hspace{-4ex} ((\mathrm{ev}_{z_0})_*[\mathcal{M}^-(c;H^-,\hat{J}^-)] \bullet \alpha) \; <c>
\end{array}
}
$$
where $\bullet$ is the intersection product between the pseudo-cycle $\mathrm{ev}_{z_0}$ and the lf cycle $\alpha$. In particular, for the unit $[M] \otimes 1 \in QH_{2n}^{lf}(M)$, 
$$\psi^-([M]) = \sum_{\mu_{H_0}(c)=2n} \# \mathcal{M}^-(c;H^-,\hat{J}^-) \cdot <c>.$$

By standard arguments (combining \cite{PSS} and \cite{Ritter3}), one checks that $\psi^-$, $\psi^+$ are chain maps inverse to each other up to chain homotopy. We omit the details. 
\end{proof}
%
\subsection{Algebro-geometric construction of 
$\mathbf{r_{\widetilde{g}}}$.}
\label{Subsection Description of rg}
\begin{theorem}\label{Theorem rg element}
$r_{\widetilde{g}}(1)\in QH^{2I(\widetilde{g})}(M)$ is represented Poincar\'e dually by the lf cycle
$$
r_{\widetilde{g}}[M] = \sum_{\gamma \in \Gamma}\; (\mathrm{ev}_{z_{\infty}})_*[\mathcal{S}(j,\hat{J},\gamma+S_{\widetilde{g}})]\; \otimes \; \gamma \; \in QC_{2n-2I(\widetilde{g})}^{lf}(M).
$$
%
%
%
After Poincar\'{e} dualizing $r_{\widetilde{g}}$, and for a generic lf chain $\alpha: \Delta^{|\alpha|}\to M$,
$$
\begin{array}{lll}
r_{\widetilde{g}}: QH_{2n-*}^{lf}(M)  \to  QH^{lf}_{2n-*-2I(\widetilde{g})}(M)\\[0mm]
r_{\widetilde{g}}(\alpha \otimes 1)  = \displaystyle \sum_{\gamma \in \Gamma} (\mathrm{ev}_{z_{\infty}})_*\left[\mathcal{S}(j,\hat{J},\gamma+S_{\widetilde{g}}) \times_{\mathrm{ev}_{z_{0}},\alpha} \Delta^{|\alpha|}\right]  \otimes \gamma \\
\qquad\quad\quad\; = \displaystyle \sum_{\gamma \in \Gamma} \sum_i (\mathrm{ev}_{z_{\infty}}\times \mathrm{ev}_{z_0})_*\left[\mathcal{S}(j,\hat{J},\gamma+S_{\widetilde{g}})\right] \bullet \left[\mathrm{D}[\beta_i] \times \alpha\right]\;\; \beta_i \otimes \gamma
\end{array}
$$
counts holomorphic sections intersecting the lf chain $\alpha$ over $z_0$ and the $($finite$)$ chain $\mathrm{D}[\beta_i]$ over $z_{\infty}$, where $\mathrm{D}[\beta_i]$ is the dual basis with respect to the intersection product $\bullet: H_*^{lf}(M) \otimes H_*(M) \to \Z/2$ of a basis $\beta_i$ of lf cycles for $H_*^{lf}(M)$.
\end{theorem}
\begin{proof}
Recall from Lemma \ref{Lemma lf homology is sh of -h} that 
$$
QH^*(M) \cong HF^*(H_0) \cong HF_{2n-*}(-H_0) \cong QH_{2n-*}^{lf}(M).
$$
So $\psi^-: QH_{2n-*}^{lf}(M) \to HF^*(H_0)$ factors through $HF_{2n-*}(-H_0)$ (canonically identified with $HF^*(H_0)$ by identifying generators), and the intermediate map $QH_{2n-*}^{lf}(M) \to HF_{2n-*}(-H_0)$ equals the $\Psi^+$ map of \cite{Seidel3} for the data $-H_{-(s+it)}$ on $D^+$ (which is dual to the data $H_{s+it}$ on $D^-$ by Lemma \ref{Lemma poincare duality}).

Similarly, our composite $r_{\widetilde{g}}=\psi^+ \circ \varphi_0 \circ \mathcal{S}_{\widetilde{g}} \circ \psi^-$ is analogous to the composite $\Psi^-\circ \varphi_0 \circ HF_*(\widetilde{g}^{-1}) \circ \Psi^+$ which arises in \cite[Sec.8]{Seidel3} but using the dual data $-H_{-(s+it)}$ instead of $H_{s+it}$. 
%
%
%
%
%
%
%
%
The gluing argument of \cite[Sec.8]{Seidel3} proves that the image of the unit $[M] \otimes 1 \in QH_{2n}^{lf}(M)$ under $r_{\widetilde{g}}: QH_{2n-*}^{lf}(M) \to QH_{2n-*-2I(\widetilde{g})}^{lf}(M)$ is
$$
r_{\widetilde{g}}([M]) 
=\sum_{\gamma \in \Gamma} (\mathrm{ev}_{z_{\infty}})_*[\mathcal{S}(j,\hat{J},\gamma+S_{\widetilde{g}})] \otimes \gamma \in QH_{2n-2I(\widetilde{g})}(M)
$$
(we evaluate at $z_{\infty}$ instead of $z_0$ because of the dualization which changes domain coordinates). In particular, since gluing sections $s^+,s^-$ representing $c',\widetilde{g}c'$ defines the equivalence class $S_{\widetilde{g}}$ (for any $c'$),  the gluing of $s^+ \in \mathcal{M}^+(\gamma \cdot \widetilde{g}^{-1}c;H^+,\hat{J}^+)$ and $s^- \in \mathcal{M}^-(c;H^-,\hat{J}^-)$ yields a section of $E_g$ in the class $\gamma + S_{\widetilde{g}}$ (take $c'=\widetilde{g}^{-1}c$).
%
%
%
%
%
%

The same gluing argument (since we are only changing the intersection conditions over $z_0,z_{\infty}$) in fact shows more generally that $r_{\widetilde{g}}=\psi^+ \circ \varphi_0 \circ \mathcal{S}_{\widetilde{g}} \circ \psi^-$ agrees on homology with the map in the claim.
\end{proof}
\begin{remark}
 The map $r_{\widetilde{g}}$ is not in general an isomorphism, unlike for closed $M$. This is because the inverse map can no longer be defined: it would involve a non-monotone homotopy $H_s$ from $H_{-1}$ to $H_0$ which has $\partial_s H_s$ unbounded above.
\end{remark}
%
%
%
\subsection{Invariance: the choice of $\hat{J}$}
\label{Subsection Invariance the choice of hat J}
%
In Theorem \ref{Theorem rg element} we did not specify precisely the choice of $\hat{J}$: the proof recovers $\hat{J}$ as a gluing of admissible $\hat{J}$ over the discs $D^{\pm}$ and an admissible $\hat{J}$ arising from Floer's continuation equation.

\begin{lemma}\label{Lemma make H monotone and sections lie near zero}
 $\mathcal{R}_{\widetilde{g}}$, $r_{\widetilde{g}}$ on cohomology do not depend on the choice of $H$ (defining admissibility for $\hat{J}$). We can choose a monotone $H$ with $\partial_s H\leq 0$ and satisfying:
$$
\begin{array}{ll}
H: E_g \to \R, \; H|_{D^{\pm}\times M}=H^{\pm},\\
H^-=0 \textrm{ on } D^-\times M,\\
H^+_t(y) = g^*0 = - K^g_t(g_t(y)) \textrm{ on } \partial D^{\pm} \times M,\\
H^+=0 \textrm{ near the centre of }D^{+}.
\end{array}
$$
For such $H$, the $(j,\hat{J})$-holomorphic sections $s: S^2 \to E_g$ have $s^{\pm}(D^{\pm})\subset M$ landing entirely in the complement of the conical end $(\Sigma \times (-\varepsilon,\infty),d(R\alpha))$ of $(M,\omega)$ (assuming $J$ is conical and $K^g=\kappa R + \textrm{constant}, \kappa>0$, on the conical end).
\end{lemma}
\begin{proof}
This is a standard cobordism argument which is proved by inspecting the $1$-dimensional parts of the parametrized moduli space $\cup_{\lambda} \mathcal{S}(j,\hat{J}_{\lambda},S)$ for a homotopy $(\hat{J}_{\lambda})_{0\leq \lambda \leq 1}$. This proves that the maps $r_{\widetilde{g}}$ obtained for $\hat{J}_0$ and for $\hat{J}_1$ are chain homotopic. We omit the details.

We homotope the glued $\hat{J}$ obtained from \ref{Subsection Description of rg} to a generic $\hat{J}$ which is admissible for a smooth monotone Hamiltonian $H$ satisfying the claim (over $D^+$ we can choose an interpolation $\phi(s)K^g \circ g_t$ where $\phi$ is monotone: $\partial_s \phi\leq 0$, $\phi=0$ for $s\ll 0$ (near the centre of $D^+$) and $\phi=-1$ near $s=0$ (the boundary $\partial D^+$)).

Suppose by contradiction that there is a section which intersects the conical end. Because $H_z$ is monotone, the maximum principle \ref{Lemma Maximum principle} applies in the region where $J$ is conical. So the section must lie in a slice $R=$ constant (which is preserved by $g_t$). In this region $\omega$ is exact and so $\widetilde{\Omega}$ is exact, so the holomorphic sphere $u=s:S^2 \to E_g$ would have $\int_{S^2} u^*\widetilde{\Omega} =0$. By Lemma \ref{Lemma compatibility trick}, $(u^{\pm})^*\widetilde{\Omega}\geq 0$ pointwise, where $u^{\pm}:D^{\pm} \to M$. So $(u^{\pm})^*\widetilde{\Omega}= 0$. Lemma \ref{Lemma compatibility trick} also shows that $u^-$ is constant on $D^-$ (since $H^-=0$ there). Via the transition, this means $t \mapsto u^+(0,t)$ along $\partial D^+$ is a non-constant orbit of $g_t^{-1}$ (it is non-constant since we are assuming $u$ does not lie in the zero section). 
Lemma \ref{Lemma compatibility trick} also shows $\partial_s u^+=0$ and hence $\partial_t u^+ = X_{H^+}$. By the first equation, the non-constant orbit $t \mapsto u^+(s,t)$ of $g_t^{-1}$ is independent of $s \in (-\infty,0]$. But $X_{H^+}=0$ for $s\ll 0$, so the second equation says the orbit is constant. Contradiction. 
%
\end{proof}
%
%
\section{Gromov-Witten invariants}
\label{Section Gromov-Witten invariants}
%
%
\subsection{Gromov-Witten invariants}
\label{Subsection GW invariants in general}
%
We now make some brief remarks about GW invariants, referring to \cite{McDuff-Salamon2,Ruan-Tian} for details.

For a closed symplectic manifold $(X,\omega)$ of dimension $\dim_{\R} X=2n$, satisfying the monotonicity condition, and a generic $\omega$-compatible almost complex structure $J$, the (genus $0$) Gromov-Witten invariant of $J$-holomorphic curves $u:\C P^1 \to X$ with $k \geq 3$ marked points in a class $[u]=\beta \in H_2(X)$ (working over $\Z/2$) is
$$
\mathrm{GW}_{0,k,\beta}^X: H_*(X)^{\otimes k} \to \Z/2, \; (\alpha_1,\ldots, \alpha_k) \mapsto (X_1\times \cdots \times X_k) \cdot \mathrm{ev}_J
$$
where we intersect in $X^k$ the pseudocycle $\mathrm{ev}_J: \mathcal{M}_{0,k}^*(\beta,J) \to X^k$ with a generic representative $X_1 \times \cdots \times X_k$ of $\alpha_1 \times \cdots \times \alpha_k$. Here $\mathcal{M}_{0,k}^*(\beta,J)$ is the moduli space of $PSL(2,\C)$-equivalence classes of stable $k$-pointed curves $(u,z_1,\ldots,z_k)$, where $u: \C P^1 \to X$ is a simple $J$-holomorphic sphere in class $\beta$ and $z_i$ are pairwise distinct points in $\C P^1$ ($\phi \in PSL(2,\C)$ acts by $(u\circ \phi^{-1},\phi(z_1),\ldots, \phi(z_k))$). To get a non-zero invariant, one requires %
$$\sum \mathrm{codim}_{\R}(\mathrm{cycles})\equiv 2nk-\sum |\alpha_i| = 2n+2c_1(TX,\omega)(\beta)+2k-6.$$

To ensure $\mathrm{ev}_J$ is a pseudo-cycle one requires a condition on $\beta$: that $\beta$ is not a multiple of a spherical homology class $B$ with $c_1(TX,\omega)(B)=0$ \cite[Sec 6.6]{McDuff-Salamon2}. The genericity condition on $X_1\times \cdots X_k$ is to ensure that it is transverse to $\mathrm{ev}_J$ and to the evaluations maps involved in the lower strata in the compactification.

If one works over $\Q$, and one chooses differential forms $a_i\in H^{2n-|\alpha_i|}(X)$ Poincar\'e dual to $\alpha_i$ supported near $X_i$, then
$$
\mathrm{GW}_{0,k,\beta}^X(\alpha_1,\cdots,\alpha_k) = \int_{\overline{\mathcal{M}_{0,k}}(\beta,J)} \mathrm{ev}_1^*a_1 \wedge \cdots \wedge \mathrm{ev}_k^*a_k \in \Q
$$
where $\overline{\mathcal{M}_{0,k}}(\beta,J)$ is the compactification by stable maps of the space of $k$-pointed $J$-holomorphic $u:\C P^1 \to X$ in class $\beta$, and $\sum \textrm{deg}(a_i) = 2n + 2c_1(TX,\omega)(\beta) + 2k -6$.
%
\subsection{GW invariants counting sections of $E_g$}
\label{Subsection GW invariants for sections}
%
The story for $(j,\hat{J})$-holomorphic sections $u: S^2 \to E_g$ is slightly different \cite[Def 8.6.6]{McDuff-Salamon2}. The key observations are:

\begin{enumerate}
 \item The quotient by $PSL(2,\C)$ in the definition of the moduli spaces defining GW invariants for $E_g$ is equivalent to imposing the condition that $u: S^2 \to E_g$ is a section, since $u\circ \phi^{-1}$ is a section for a unique $\phi = \pi_g\circ u \in PSL(2,\C)$.

 \item Sections lie in a class $\beta = [S^2] + (j_{z_0})_*\beta_0$, for some $\beta_0 \in H_*(M)$ where $j_{z_0}: M \to E_g$ includes the fibre over $z_0$. So the condition on $\beta$ is automatic since $(\pi_g)_*[u]=[S^2]$, and a section is automatically simple.

 \item Suppose we want to use \emph{fixed} marked points $w_i \in S^2$ (pairwise distinct) and we want the sections to intersect $j_i(X_i)$ where $X_i$ represents $\alpha_i \in H_*(M)$ and $j_i: M \to E_g$ is the inclusion of the fibre over $w_i$. Then, when defining the GW invariants for $E_g$, we can still let the marked points $z_i\in S^2$ vary freely since the intersection condition $u(z_i)\in j_i(X_i)$ automatically forces 
$$ z_i = \pi_g(u(z_i)) = \pi_g(j_i(X_i))=w_i.$$

 \item One can make sense of these GW invariants even when $0\leq k<3$: we can simply add $3-k$ extra marked points and we require the (automatically satisfied) condition that the section intersects $j_i(M)$ for these new marked points. Any section of $E_g$ will automatically intersect $[M]$ once transversely over these new $w_i$. So we are ensuring the divisor axiom \cite[Rmk 7.5.2]{McDuff-Salamon2}.
%
\end{enumerate}

\noindent The upshot, is that the GW invariant
$$
\mathrm{GW}_{0,k,\beta}^{E_g}: QH_*(M)^{\otimes k} \to \Lambda, \; (\alpha_1,\ldots, \alpha_k) \mapsto (j_1(X_1)\times \cdots \times j_k(X_k)) \cdot \mathrm{ev}_J
$$
corresponds precisely to the sections one plans to count modulo $2$, with weight $\gamma_{\beta}$:
$$ 
\mathrm{GW}_{0,k,\beta}^{E_g}(\alpha_1,\ldots,\alpha_k) = \# \{ u \in \mathcal{S}(j,\hat{J},\gamma_{\beta}+S_{\widetilde{g}}): u(w_i) \in j_i(X_i)  \}
$$
using $\hat{J}$ on $E_g$ to define GW, where $\gamma_{\beta}\in \Gamma$ is determined by $\beta$ (here $\beta=[S^2]+(j_{z_0})_*\beta_0$, and $\beta_0\in H_2(M)$ is a spherical class so determines a $\gamma_{\beta}\in \Gamma$), and where we require the dimension is correct: 
$$2m+2c_1(T^vE_g,\Omega)(\gamma_{\beta}+S_{\widetilde{g}})=\sum \textrm{codim}_M(X_i),$$
 which is equivalent to the GW condition 
$$(2m+2)+2c_1(TE_g,\widetilde{\Omega})(\beta) +2k-6 = (2m+2)k-\sum |\alpha_i|$$
where $\dim_{\R} M =2m$ (using $TE_g\cong TS^2 \oplus T^vE_g$, and $6=2+2c_1(TS^2)[S^2]$).

We will only be considering the case: $k=2$, $\alpha_1\in QH_*^{lf}(M)$, $\alpha_2\in QH_*(M)$.
%
%
\section{Negative line bundles}
\label{Section Negative line bundles}
%
%
\subsection{Definition and properties}
\label{Subsection Definition and properties neg line bdle}
Fix $(B,\omega_B)$ any closed symplectic manifold. A complex line bundle $\pi:L \to B$ is called \emph{negative} if for some real $n>0$,
$$
c_1(L) = -n[\omega_B].
$$
\textbf{Examples:} 
\begin{enumerate}
 \item $\mathcal{O}(-n) \to \CP^m$ for integers $n\geq 1$. Recall $\mathcal{O}(-1)=\{(x,v): v\in x\}\subset \CP^m\times \C^{m+1}$, and $\mathcal{O}(-n)=\mathcal{O}(-1)^{\otimes n}$ has $c_1({O}(-n))[\CP^1]=-n$. 

 \item Any $L$ dual to an ample holomorphic line bundle over a compact complex manifold $B$. Indeed for some $k>0$, $L^{-k}$ is very ample, so $L^{-k} = j^*\mathcal{O}(1)$ via the embedding $j: B \to \CP^m$ defined by the global holomorphic sections of $L^{-k}$. Let $\omega_B = j^*\omega_{\CP^m}$. Since the Fubini-Study form $[\omega_{\CP^m}] = c_1(\mathcal{O}(1))$,
$$-k c_1(L)= c_1(L^{-k}) = j^*c_1(\mathcal{O}(1)) = j^*[\omega_{\CP^m}] =[\omega_B].$$
Indeed any compact complex manifold admitting a holomorphic embedding $B\subset \CP^m$ arises in this way, and by Kodaira's embedding theorem these are precisely the compact K\"ahler manifolds with integral K\"ahler form.
\end{enumerate}

\begin{lemma}[see Oancea \cite{Oancea}]\label{Lemma negative curvature is negativity}  $L\to B$ is negative iff $L$ admits a Hermitian metric, and some Hermitian connection whose curvature $\mathcal{F}$ satisfies $\frac{i}{2\pi }\mathcal{F}(v,J_B v)<0$ for all $v\neq 0 \in TB$ and for all almost complex structures $J_B$ compatible with $\omega_B$ (meaning $\omega_B(\cdot,J_B \cdot)$ is a metric).
\end{lemma}

This Lemma essentially follows from the fact that $\frac{i}{2\pi}\mathcal{F}$ represents $c_1(L)$ inside $H^2(B;\R)$. For example, in one direction, if $c_1(L)=-n[\omega_B]$, then there is a Hermitian metric on $L$ whose curvature satisfies $n\omega_B=\frac{1}{2\pi i}(\mathcal{F} + da)$, and by adding the one-form $-a$ to the connection one can get rid of the exact term $da$.
%
\subsection{Construction of the symplectic form}
\label{Subsection construction of the symplectic form neg line bdle}
%
From now on $M$ is the total space of a negative line bundle $\pi: L \to (B,\omega_B)$, and we assume a connection and metric as above are chosen. Thus  
$$c_1(L) = [\tfrac{i}{2\pi} \mathcal{F}] =-n[\omega_B] \in H^2(B,\Z) \cap H^2(B,\R).$$
We choose $\boxed{\Sigma = \{ r=1 \}}$ to be the hypersurface for $M$ (which will be contact). 

\begin{examples*}
 $\mathcal{O}(-1) \to \C P^m$ arises as the blow-up of $\C^{m+1}$ at the origin, so $\Sigma \cong S^{2m+1}$ is the preimage of $S^{2m+1}\subset \C^{m+1}$. The multiplication action on $\C^{m+1}$ by a primitive $n$-th root of unity lifts to the blow-up, fixing the exceptional $\C P^m$ which is the zero section of $\mathcal{O}(-1)$. The quotient by this action defines a bundle map $\mathcal{O}(-1) \to \mathcal{O}(-n)$. So for $\mathcal{O}(-n) \to \C P^m$, $\Sigma=S^{2m+1}/(\Z/n)$ is a Lens space.     
\end{examples*}

We will now construct the symplectic form $\omega$ for $M$ of the form
$$
\boxed{\omega = d\theta + \varepsilon \Omega\qquad (\textrm{fixed }\varepsilon>0)}
$$
consisting of a \emph{non-exact} form $[d\theta]=n\pi^*[\omega_B]$ (only away from the zero section it is exact) and a term $\Omega$ which is fibrewise the area form (not contributing to $[\omega]$). 

For $w \in L$, define the radial function $r$ by $r(w) = |w|$ in the above metric.

The connection defines the fibrewise angular $1$-form $\theta=\frac{1}{4\pi} d^c \log r^2$ on $L\setminus(\textrm{zero section})$, which satisfies 
$$d\theta = -\tfrac{1}{2\pi i} \partial \overline{\partial} \log r^2 \equiv -\pi^*c_1^{\C}(L) = -\tfrac{i}{2\pi}\,\pi^*(\mathcal{F}) = n\pi^*\omega_B.$$
Explicitly \cite[p.132]{Audin-Lafontaine}, $\boxed{\theta_w(\cdot)=\tfrac{1}{2\pi r^2}\langle iw, \cdot \rangle}$ so in the complement of the zero section
$$
\theta_w(w)=0,\; \theta_w(iw) = 1/2\pi
$$
where $w,iw$ is considered as a basis of $T_w^{\mathrm{vert}} L \cong L_{\pi(w)}$, and $\theta=0$ on horizontal vectors.

\begin{lemma}\label{Lemma dtheta properties}
$d\theta(v,\cdot)=0$ for any vertical vector $v\in TL$ $(v \in \ker d\pi)$. On horizontal  $v,v' \in T_w L$, $d\theta(v,v')\! =\! - \theta([v,v']) \!=\! \theta_w(\mathcal{F}_{d\pi\cdot v, d\pi \cdot v'} w)$ \emph{(}see \cite[p.133]{Audin-Lafontaine}\emph{)}. Since $\pi^*\mathcal{F}$ is imaginary valued, we deduce $\boxed{d\theta = \tfrac{1}{2\pi i} \pi^*\mathcal{F}}$, which extends $d\theta$ over the zero section. \emph{Remark:} our curvature is opposite to \cite[p.120]{Audin-Lafontaine}.
\end{lemma}

On $L\setminus (\textrm{zero section})$ define
$\boxed{\Omega = d(r^2 \theta)}.$
Fibrewise this is $(\textrm{area form})/\pi$, so extend $\Omega$ over the zero section by
$$
\Omega|_{\textrm{fibre}} = (\textrm{area form})/\pi \qquad \Omega(T(\textrm{zero section}),\cdot)=0.
$$
%
%
%
\subsection{Liouville and Reeb fields}
\label{Subsection Liouville and Reeb fields neg line bdle}
%
\begin{lemma}\label{Lemma fiberwise liouville and reeb}
Fibrewise the Liouville and Reeb fields for $\Omega$ at $w\in L$ are
$$
Z_{\Omega} = \frac{1}{2} w,\qquad Y_{\Omega}=2\pi i w = 4\pi iZ_{\Omega}.
$$
\end{lemma}
\begin{proof}
By Lemma \ref{Lemma dtheta properties}, $d(r^2\theta)(\frac{w}{2},\cdot) = 2r\,dr(\frac{w}{2}) \theta = (2r^2/2)\theta = r^2\theta$ using $dr(w)=r$ and $\theta(w)=0$; $r^2\theta(2\pi i w)= 1$ on $\Sigma$, $d(r^2\theta)(2\pi i w,\cdot) = 0$ on $T\Sigma$ using $dr(iw)=0$ and $dr(T\Sigma)=0$ (by Lemma \ref{Lemma dtheta properties}, $d\theta(iw,\cdot)=0$ since $iw$ is vertical).
\end{proof}

Now study the conical symplectic manifold $(M,\omega)$ with hypersurface $\Sigma$, where
$$
\omega = d\theta + \varepsilon \Omega = 
d((1+\varepsilon r^2)\, \theta) \qquad (\textrm{fix }\varepsilon >0).
$$
At infinity, indeed in the complement of the zero section, $\omega$ is exact since the primitive $(1+\varepsilon r^2)\, \theta$ is defined there. 

\begin{lemma}\label{Lemma Liouville and Reeb for neg line bdle}
 The Liouville field $Z$ for $(M,\omega)$ is
$$
Z = \frac{1+\varepsilon r^2}{ \varepsilon r^2} \cdot \frac{w}{2}
$$
which is defined away from the zero section and is outward pointing along $\Sigma$. 

The Reeb vector field is
$$Y = \frac{2\pi}{1+\varepsilon}\, i w.$$
The Reeb periods are $k(1+\varepsilon)$ for $k=0,1,2,\ldots$, with a Reeb orbit $w(t)=e^{2\pi i t/(1+\varepsilon)} w_0$ in each fibre with base point $w_0$ and $t\in [0,k(1+\varepsilon)]$. 
\end{lemma}
\begin{proof}
By the previous two Lemmas, 
$
\omega(Z_{\Omega},\cdot) = \varepsilon r^2 \theta.
$
So normalizing: $Z = \frac{1+\varepsilon r^2}{\varepsilon r^2} Z_{\Omega}$.
Since $Y_{\Omega}$ is vertical, by Lemma \ref{Lemma dtheta properties} we have $d\theta(Y_{\Omega},\cdot) = 0$ and $d\Omega(Y_{\Omega},\cdot) = -2 r dr(\cdot)$ (using $\theta(iw)=1/2\pi$). So $\omega(Y_{\Omega},\cdot)=0$ on $T\Sigma$ (parallel transport preserves $r$, so $T\Sigma_w$ is spanned by the horizontal vectors and the vertical $iw$, and $dr(iw)=0$). Finally, on $\Sigma$, $(\theta+\varepsilon r^2\theta)(Y_{\Omega})=1+\varepsilon$. So normalizing: $Y=Y_{\Omega}/(1+\varepsilon)$.\end{proof}
%
\subsection{Conical parametrization}
\label{Subsection conical parametrization neg line bdle}
%
\begin{lemma}\label{Lemma R coordinate for neg line bldes}
The radial coordinate $R$ in the sense of Section \ref{Subsection Conical symplectic manifolds} is
$$
R = \frac{1+\varepsilon r^2}{1+\varepsilon},
$$
defined on all of $M$ with differential $dR = (2\varepsilon r)(1+\varepsilon)^{-1} dr$ vanishing on the zero section. The flow of $Z$ defines the conical parametrization
$$(M_1,\omega|_{M_1}) \cong \left(\Sigma\times \left(\tfrac{1}{1+\varepsilon},\infty\right), d(R\alpha)\right)$$
 where $R$ is the coordinate for the interval, $\alpha=(1+\varepsilon)\theta|_{\Sigma}$, $M_1 =M \setminus (\textrm{zero section})$.
\end{lemma}
\begin{proof}
 Let $w(t)$ solve $\dot{w}(t)=Z(w(t))$ with $w(0)=w_0 \in \Sigma$. The radial coordinate is defined by $R(w(t))=e^t$. The solution $w$ is unique, and we try to solve for $w(t)=r(t)w_0$. Then the equation becomes
$\dot{r} =  (1+\varepsilon r^2)/2\varepsilon r.
$
So $\partial_t(1+\varepsilon r^2) = 2\varepsilon r \dot{r} = 1+\varepsilon r^2$, thus
$
1+\varepsilon r^2 = (1+\varepsilon)e^t =  (1+\varepsilon) R.
$
\end{proof}

%
%
%
\subsection{The Hamiltonians}
\label{Subection Hamiltonians neg line bdle}
%
Consider the Hamiltonian
$$
H=h_k(R) = k (1+\varepsilon) R.
$$
Since in general $X_H = h'(R) Y$, we obtain
$$
X_H = k (1+\varepsilon) Y.
$$
The flow is $w(t) = e^{k  2\pi i t}w(0)$. Observe that for integer values of $k$ the flow is $1$-periodic, but for non-integer values of $k$ the only orbits are the constant orbits lying on the zero section (which is the critical level set for $H$).

%
%

The Hamiltonians $h_k$, $k\notin \Z$, have degenerate $1$-orbits, indeed they are Morse-Bott with critical level set $C$ the zero section.

There are two ways around this. One can introduce an auxiliary Morse function $f$ on $C$, and then one defines $CF^*(h_k,f)$ by standard Morse-Bott techniques (see for example Bourgeois-Oancea \cite{Bourgeois-Oancea}). The generators will be the critical points of $f$ in $C$, and the differential will count rigid trajectories which are suitable combinations of $-\nabla f$-flowlines inside $C$ and Floer flowlines with ends on $C$. This approach is an infinitesimal version of the second approach, which is to explicitly construct a perturbation of the form
$$
h_{k,\epsilon} = h_k + \epsilon f
$$
using a time-dependent function $f$ supported near $C$ and Morse on $C$, and a small enough constant $\epsilon>0$. For small enough $\epsilon$, one then shows that the local Floer cohomology near $C$ is isomorphic to the Morse cohomology of $C$. This is also a standard method (for instance, for $S^1$ critical level sets, see \cite[Prop. 2.2]{CFHW}). We omit these details.
%

\subsection{The $g$-action}
\label{Subection g action neg line bdle}
%
The action by rotation in the fibres, 
$$g_t = e^{2\pi i t},$$ 
is Hamiltonian generated by
$K = h_{1}(R) = (1+\varepsilon) R.$
Since $g_t$ preserves $R$, the pull-back of the Hamiltonians by the $g$-action is:
$$
g^*h_k= h_k \circ g_t - K \circ g_t = (1+\varepsilon) k R - (1+\varepsilon) R = h_{k-1}.
$$
%
%
%
\subsection{Complex structure}
\label{Subsection complex structure neg line bdle}
%
The complex structure $J=i$ does not strictly satisfy ``$JZ=Y$'', but it satisfies a rescaled version:
$$
Y = \left. \frac{4\pi \varepsilon r^2}{(1+\varepsilon)(1+\varepsilon r^2)}\, J Z\right|_{\Sigma} = \frac{4\pi \varepsilon}{(1+\varepsilon)^2}\, J Z,
$$
so the contact condition ``$dR \circ J = -R\alpha$'' is actually rescaled as follows:
$$
dR \circ J = \frac{-4\pi \varepsilon r^2}{(1+\varepsilon)(1+\varepsilon r^2)} R \alpha.
$$
\begin{lemma}[Maximum principle]\label{Lemma Max principle for neg line bdles}
 Lemma \ref{Lemma Maximum principle} holds for $J=i$ everywhere on $M$.
\end{lemma}
\begin{proof} 
We mimick the old proof (Lemma \ref{Lemma Maximum principle}). Let $\rho = (1+\varepsilon r^2)\circ u$. Since $dr|_w (w)=r, dr|_w(iw)=0$, we deduce
$$
dr \circ i = -2\pi r \theta.
$$
Thus, letting $\widetilde{\theta}=(1+\varepsilon r^2) \theta$ denote the primitive for $\omega$,
$$
\begin{array}{lll}
d\rho \circ j &=& 2\varepsilon r dr (i\circ du + \nu\circ j)\\
& =& \frac{4\pi  \varepsilon r^2}{1+\varepsilon r^2} (-u^*\widetilde{\theta}+dt \otimes \widetilde{\theta}(X_H))\\
&=& 4\pi \frac{\rho - 1}{\rho}  (-u^*\widetilde{\theta}+dt \otimes \widetilde{\theta}(X_H))
\end{array}
$$
$$\textstyle
(-d(d\rho\circ j) + 1^{st} \textrm{order in }\rho) \geq 4\pi \frac{\rho-1}{\rho} (-(R\circ u) \partial_s h')\, ds \wedge dt -d(4\pi\frac{\rho-1}{\rho})\wedge \frac{\rho}{4\pi (\rho-1)}(d\rho \circ j).
$$
We need to ensure the right hand side is a positive multiple of $ds\wedge dt$ so that, as in the old proof, $(\Delta \rho + 1^{st} \textrm{ order terms in }\rho) \geq 0$ provided $\partial_s h'\leq 0$.
%
%

So we need $\rho\geq 1$ for the first term. The second term is $-\frac{1}{\rho^2}\cdot \frac{\rho}{\rho-1} d\rho \wedge d\rho \circ j$, and $d\rho \wedge d\rho \circ j = (-(\partial_s \rho)^2-(\partial_t \rho)^2)\,ds \wedge dt$. So $\rho \geq 1$ suffices, equivalently: $r\geq 0$.
\end{proof}
\begin{corollary}\label{Corollary sections lie in zero section neg lbdles}
If $H$ is monotone as in \ref{Subsection Invariance the choice of hat J}, and $\hat{J}$ is admissible with $J=i$,
then $(j,\hat{J})$-holomorphic sections of $E_g \to S^2$ must land in the zero sections of the fibres.
\end{corollary}
\begin{proof}
 Lemmas \ref{Lemma make H monotone and sections lie near zero} and \ref{Lemma Max principle for neg line bdles}, using that $\omega$ is exact except on the zero section.
\end{proof}
\begin{lemma*}
 The (non-admissible) complex structure $\hat{J}=\left[ \begin{smallmatrix} j & 0 \\ 0 & i \end{smallmatrix} \right]$ on $D^{\pm} \times M$ yields a complex structure on $E_g$ ($i$ is $g$-invariant) and it can be used to compute $r_{\widetilde{g}},\mathcal{R}_{\widetilde{g}}$ possibly after a generic small perturbation to make it regular.
\end{lemma*}
\begin{proof}
Let $\hat{J}_{H}=\left[ \begin{smallmatrix} j & 0 \\ ds \otimes X_H - dt \otimes J_zX_H & i \end{smallmatrix} \right]$ constructed for the monotone $H$ as in \ref{Subsection Invariance the choice of hat J}. If $H$ is the same as the Hamiltonian defining $\widetilde{\Omega}$, then we showed in Lemma \ref{Lemma Symplectic form for admissible J} that $\hat{J}_{H}$ is compatible with a symplectic form $\widetilde{\Omega}+\pi_g^*\sigma$.

For $\hat{J}_0$ (the $\hat{J}$ of the claim), compatibility will fail at infinity but it will still hold in a large compact region surrounding the zero section of $E_g$ (which can be made larger by rescaling $\sigma$ by a positive constant).

However, for the purpose of defining $\mathcal{R}_{\widetilde{g}}, r_{\widetilde{g}}$, this lack of compatibility will not matter if we can show that all $(j,\hat{J}_{H_{\lambda}})$-holomorphic sections lie in a compact region where compatibility holds, for each $H_{\lambda}$ in a homotopy $(H_{\lambda})_{0\leq \lambda \leq 1}$ from $H$ to $0$.

Inspecting the proof of Lemma \ref{Lemma Maximum principle} or \ref{Lemma Max principle for neg line bdles}, the new term in $d\rho \circ j$ caused by changing $\hat{J}$ (but keeping $\widetilde{\Omega}$ the same) is the term $dt \otimes \theta(X_{H_{\lambda}-H})$. So, outside that compact region, $H_{\lambda}-H$ is radial, say $(H_{\lambda}-H)(u)=k_{\lambda}(\rho)$, so  the new term in $-d(d\rho\circ j)$ is
$$
-d(\rho k_{\lambda}(\rho)) dt = -(k_{\lambda}(\rho) - \rho k_{\lambda}'(\rho))\, \partial_s \rho\, ds \wedge dt
$$
and these first order terms in $\rho$ don't affect the proof of the maximum principle. 

By \ref{Subsection Invariance the choice of hat J},  $\mathcal{R}_{\widetilde{g}}, r_{\widetilde{g}}$ will not be affected in homology if we homotope $\hat{J}_{H}$ to $\hat{J}_0$.
\end{proof}

\begin{remark*}
  In the notation of the proof, if $\hat{J}_0$ is not regular then one needs to homotope it to $\hat{J}_{L}$, where $L$ is a small perturbation of $0$ typically non-radial near the zero section (the maximum principle will not hold there, so $(j,\hat{J}_L)$-holomorphic sections may not lie entirely in the zero section) but $L=0$ away from the zero section (so the maximum principle applies and sections cannot touch this region).
\end{remark*}

%
\subsection{The choice of $\widetilde{g}$}
\label{Subection gtilde choice neg line bdle}
\label{Subection the maslov index of g neg line bdles}
%
The action of $g$ on $\mathcal{L}_0 M$ lifts to an action of $\widetilde{\mathcal{L}_0 M}$. We choose the lift $\widetilde{g}$ so that the constant orbits $x$ on the zero section lifted to $(c_x,x) \in \widetilde{\mathcal{L}_0 M}$ satisfy $\widetilde{g}\cdot (c_x,x) = (c_x,x)$, where $c_x:D \to M$ is the constant map to $x$. 
So $S_{\widetilde{g}}$ is represented by the constant $s_g^+(z) = c_x(z)=x$, $s_g^-(z)=(\widetilde{g}c_x)(z)=c_x(z)=x$.
\begin{lemma*}
 $\widetilde{\Omega}(s_{\widetilde{g}})=0.$
\end{lemma*}
\begin{proof}
As in \ref{Subsection Invariance the choice of hat J}, choose
$
\widetilde{\Omega}^+= \omega^+ + \phi'(s)K\circ g_t\, ds \wedge dt$, $\widetilde{\Omega}^-=\omega^-$
so $\widetilde{\Omega}(s_{\widetilde{g}}) = \int_{D^+} \phi'(s) K(x) \, ds \wedge dt =0$ since $K=0$ on the zero section.
\end{proof}
\begin{lemma}\label{Lemma calculation of I(g) neg line bdle}
 $I(\widetilde{g}) = 1$ $($defined in Section \ref{Subsection Maslov index I of g}$)$.\end{lemma}
\begin{proof}
Using any $(c_x,x)$ as above, pick a unitary trivialization of $L$ over a neighbourhood of the point $b=\pi(x)\in B$ to obtain
$$
\tau_{c_x}: x^*TM \cong S^1 \times T_{x} M \cong S^1 \times T_b B \times L_b \cong S^1 \times T_b B \times \C.
$$
Now $\widetilde{g}\cdot(c_x,x)=(c_x,x)$, and $g_t$ is a linear holomorphic action given by multiplication by a complex number, so $dg_t$ commutes with $\tau_{c_x}(t)$. Thus
$$
\ell(t) = \tau_{c_x}(t) \circ dg_t \circ \tau_{c_{x}}(t)^{-1} = dg_t \circ \tau_{c_x}(t) \circ \tau_{c_{x}}(t)^{-1} = e^{2\pi i t},
$$
so $\ell(t)$ is the rotation of the $\C$ factor and the identity on the $T_b B$ factor. So $t\mapsto \det e^{2\pi i t}= e^{2\pi it}$ is $1$ in $H_1(S^1;\Z)\cong \Z$. So $I(\widetilde{g}) = \textrm{deg}(\ell) = 1$.
\end{proof}
%
\section{Symplectic cohomology of $M=\mathrm{Tot}(\mathcal{O}(-n)\to \C P^1)$}
\label{Section O(-n) over CP1}
%
%
\subsection{The $r_{\widetilde{g}}$ map for $\mathcal{O}(-n)$}
\label{Subection rg for O(-1)}
%
Consider $M=\textrm{Tot}(\mathcal{O}(-n) \to \C P^1)$ for $n\geq 1$.
The generators of $H_*^{lf}(M)$ are in degree $2$ and $4$:
$$
\begin{array}{rcl}
 F &=& \textrm{ fibre } \C, \textrm{ Poincar\'e dual to the zero section }[\CP^1]\\
 M &=& \textrm{ fundamental chain, Poincar\'e dual to the point class } [\mathrm{pt}]
\end{array}
$$
Using a connection, $TM \cong T\C P^1 \oplus \mathcal{O}(-1)$, so 
$
c_1(TM)[\CP^1] = 2-n.
$
The zero section $[\C P^1]$ generates $\pi_2(M)$. $M$ satisfies weak$^+$ monotonicity: it is either monotone (for $\mathcal{O}(-1)$), or $c_1=0$ (for $\mathcal{O}(-2)$), or the min Chern number $|N|\geq 1$:
$$
N = c_1(TM,\omega)([\C P^1]) = 2-n.
$$
Moreover, $\Lambda$ is generated by $[\C P^1]$ which has
$
c_1(TM,\omega)[\C P^1] = N$,
$\omega[\C P^1] >0.
$
Writing $t=[\C P^1]$ for the generator of $\Lambda$, and $t^m = m[\C P^1]$, we obtain
$$
\begin{array}{lll}
\Lambda &=& \displaystyle \Z[t^{-1},t]] = \{ \sum n_j t^{m_j} : n_j\in \Z/2, \; \lim_{j \to \infty} m_j = \infty \}\\
|t| &=& -2c_1(TM,\omega)[\C P^1] = -2N \quad (\textrm{homological grading})
\end{array}
$$
By Lemma \ref{Lemma calculation of I(g) neg line bdle} and the choice of $\widetilde{g}$ in \ref{Subection gtilde choice neg line bdle},
$
c_1(TE^v_g,\Omega)(S_{\widetilde{g}}) = -I(\widetilde{g}) =-1$, and
$\widetilde{\Omega}(S_{\widetilde{g}}) = 0$.
So the dimension of the space of sections (Lemma \ref{Lemma dim of space of sections}) is
$$
\begin{array}{lll}
\dim \mathcal{S}(j,\hat{J},t^m + S_{\widetilde{g}}) &=& 2\dim_{\C} M + 2 c_1(TE^v_g,\Omega)(S_{\widetilde{g}}) + 2m\,c_1(TM,\omega)(t) \\
&=& 2+2N\cdot m
\end{array}
$$
The condition that the sections intersect $F$ or $M$ at $z_0$ cuts down the dimension respectively by $2$ or $0$, and then evaluation at $z_{\infty}$ sweeps out a locally finite chain in dimension $2Nm$ or $2+2Nm$. So in these two cases, the possibilities are:
$${\tiny
\begin{array}{c|c|c|c|c|c|c}
 -n & N=2-n & |t|=-2N & 2Nm & 2+2Nm & l.f. \; 2Nm\textrm{-chains} & l.f.\; (2+2Nm)\textrm{-chains} \\  
\hline
-1 & 1  & -2  & 2m & 2+2m  & \mathbf{F} (m=1), \textbf{M} (m=2)    & \mathbf{F} (m=0), \textbf{M} (m=1)\\ 
-2 & 0  & 0   &   0 & 2     & \textrm{none}       & \textbf{F} (\textrm{any } m)\\
-3 & -1 & 2   &  -2m & 2-2m  & F (m=-1), M (m=-2)  & \mathbf{F} (m=0), M (m=-1)\\
-4 & -2 & 4   &  -4m & 2-4m  & M (m=-1)            & \mathbf{F} (m=0)\\
\leq -5 & \leq -3 & \geq 6   & \leq -6m & \leq 2-6m & \textrm{none} & \mathbf{F} (m=0)
\end{array}
}
$$
We can rule out $m<0$ since a $(j,\hat{J})$-holomorphic section $S$ has $\widetilde{\Omega}(S)\geq 0$ (by Lemma \ref{Lemma compatibility trick}) and
$
\widetilde{\Omega}(t^m + S_{\widetilde{g}}) = m \omega(\C P^1) + \widetilde{\Omega}(S_{\widetilde{g}}) = m \omega(\C P^1).
$
The sections $s$ for $m=0$ are constant (since $\widetilde{\Omega}(s)=0$).

The sections in class $t^m+S_{\widetilde{g}}$ contribute with Novikov weight $t^{m}$ to $r_{\widetilde{g}}$. Thus viewing $\Lambda^2\equiv QC^{lf}_*(M) \equiv \Lambda\cdot (F\otimes 1) + \Lambda \cdot (M\otimes 1)$, the matrix $r_{\widetilde{g}}: \Lambda^2 \to \Lambda^2$ is
$${
\begin{array}{c|c|c}
 n=1  & n=2 &  n\geq 3\\
\hline
\strut & \strut & \strut \\[-2mm]
\left[\begin{smallmatrix} At & C \\ Bt^{2} & Dt  \end{smallmatrix}\right] &
\left[\begin{smallmatrix} 0 & C\lambda \\ 0 & 0  \end{smallmatrix}\right] &
\left[\begin{smallmatrix} 0 & C \\ 0 & 0  \end{smallmatrix}\right]
\end{array}
}
$$
where $A,B,C,D\in \Z/2$, $\lambda\in \Lambda$. Note this is nilpotent for $n\geq 2$, so:
%
%
\begin{corollary*}
$SH^*(M) = 0$ for $M=\mathrm{Tot}(\mathcal{O}(-n)\to \CP^1)$ and $n\geq 2$.
\end{corollary*}
%
\subsection{Description of $\mathbf{E_g}$ for $\mathbf{M=\mathrm{Tot}(\mathcal{O}(-1)\to \C P^1)}$}
%
%
\begin{lemma*}
The $\C$-line bundle over $\C P^1$ with transition $\partial D^+ \times \C \to \partial D^- \times \C$ given by $([e^{2\pi i t}:1],x) \to ([1:e^{-2\pi i t}],g_t\cdot x)$ is the bundle $\mathcal{O}(-1)$ \emph{(}where $g_t=e^{2\pi it})$.
\end{lemma*}
\begin{proof}
Coordinates: $[w:1]$ on $D^+ = $ Northern hemisphere of $S^2\equiv \C P^1$, and $[1:z]$ on $D^-=$ Southern hemisphere. Claim: $\mathcal{O}(-1)$ is defined by the transition $([w:1],x) \mapsto ([1:\frac{1}{w}],wx)$. Sanity check: $\mathcal{O}(1)$ has transition $g^{-1}=1/w$ and has a holomorphic section $w=1$ on $D^+$, $z=z$ on $D^-$ (simple zero at $z_0$).

We compute $c_1$. The orientation on $D^+$ is induced by $(s,t)\in (-\infty,0]\times \R$ via $w=e^{2\pi(s+it)}$. The equator $C=\{[e^{2\pi it}:1] \}$ is the positively oriented boundary of $D^+$: (outward normal, $\partial_t$) is an oriented basis of $S^2$. The equator is a \emph{negatively} oriented boundary for $D^-$, so $c_1[\C P^1]$ is $-\textrm{deg}($transition from $\partial D^+$ to $\partial D^-$), and $-\textrm{deg}(t \mapsto e^{2\pi i t}\in U(1))=-1$.
\end{proof}

\begin{corollary*}
 For $M=\mathrm{Tot}\,(\pi_M: \mathcal{O}(-1) \to \C P^1)$, the complex line bundle $(\pi_g,\pi_M): E_g \to S^2 \times \C P^1$ is $\mathcal{O}(-1,-1)=\pi_g^*\mathcal{O}(-1)\otimes \pi_M^*\mathcal{O}(-1)$.
\end{corollary*}
\begin{proof}
 The transition along the equator of $S^2$ is as in the previous lemma, and the transition over the equator of $\C P^1$ is the same as the transition for $M=\mathcal{O}(-1)$. 
\end{proof}

\begin{lemma*}
$m=d = \mathrm{degree}(\textrm{sections in class }t^m+S_{\widetilde{g}})$ so the virtual dimension of the space of sections in class $(1,d)\in H^2(S^2 \times \C P^1)$ via $(\pi_g,\pi_M)$ is $2+2d$.
\end{lemma*}
\begin{proof}
Viewing $E_g = \textrm{Tot}(\mathcal{O}(-1,-1)\to S^2\times \C P^1)$, a choice of connection yields $TE_g \cong T(S^2\times \C P^1) \oplus \mathcal{O}(-1,-1)$, so 
$$c_1(TE_g) = (2,2) + (-1,-1) = (1,1) \in H^2(S^2\times \C P^1)$$
Similarly, using $E_g \to S^2$, $TE_g \cong TS^2 \oplus T^vE_g$ so 
$$
c_1(T^v E_g) = c_1(TE_g) - c_1(TS^2) = (1,1)-(2,0) = (-1,1) \in H^2(S^2\times \C P^1).
$$
The space of sections in class $(1,d)$ therefore has $\dim =4 + 2\cdot \langle (-1,1),(1,d) \rangle = 4-2+2d$. Compare this with the formula $4-2+2m$ for sections in class $t^m+S_{\widetilde{g}}$.
\end{proof}
\begin{remark*}
 For $M\!=\!\mathrm{Tot}\,(\mathcal{O}(-n) \!\to\! \C P^1)$, $(E_g \!\to\! S^2 \times \C P^1) \!=\! \mathcal{O}(-1,-n)$ and $m=d$.
\end{remark*}
%
%
\subsection{The sections of $E_g$ for $M=\mathcal{O}(-1)\to \C P^1$}
%
In $$r_{\widetilde{g}}= \left[\begin{array}{cc} At & C \\ Bt^{2} & Dt  \end{array}\right]$$ only $m=0,1,2$ contribute, so we only care about sections in classes $(1,0),(1,1),(1,2)$.

Sections in class $(1,0)$ have area $\widetilde{\Omega}(S_{\widetilde{g}})=0$, so they are constant sections:
$$
u: S^2 \to S^2 \times \C P^1 \subset E_g, z \mapsto (z,y),
$$
some $y\in \C P^1$. This is a $2$-dimensional space of sections, agreeing with $\textrm{virdim}_{\R}=2$.

\begin{lemma}\label{Lemma constants neg l bdle}
 $\hat{J}=\left[ \begin{smallmatrix} j & 0 \\ 0 & i \end{smallmatrix} \right]$ is regular for the  constant sections, and $C=-1$.
\end{lemma}
\begin{proof}
We are in the integrable case, so $D_u$ is just the Dolbeaut operator:
$$
\overline{\partial} = \partial_s + J \partial_t: \Gamma(u^*T^vE_g) \to \Gamma(u^*T^vE_g \otimes_{\C} \Omega^{0,1}S^2),\, D_u \cdot \xi = (\partial_s u + J\partial_t u) \otimes (ds-i\,dt),
$$
(we only differentiate in the vertical directions of $E_g$ since we only consider sections).

Now $(u^*T^vE_g)_z = T(E_g)_{(z,y)} \cong T_y \C P^1 \oplus (\mathcal{O}(-1))_y \cong \C \oplus \C$. The transition over the equator of $S^2$ is multiplication by $dg_t$, which acts by $(\textrm{id},g_t)$ on the fibre $\C \oplus \C$. Thus, as bundles over $S^2$,
$
u^*T^vE_g \cong \underline{\C} \oplus \mathcal{O}(-1).
$
We deduce:
$$
\overline{\partial}: \Gamma(S^2,\underline{\C} \oplus \mathcal{O}(-1)) \to \Omega^{0,1}(S^2,\underline{\C} \oplus \mathcal{O}(-1))
$$
Using Dolbeaut's theorem
$
H^{p,q}_{\overline{\partial}}(S^2,\underline{\C} \oplus \mathcal{O}(-1)) \cong H^q(S^2,\Omega^p(\underline{\C} \oplus \mathcal{O}(-1))),
$
and using Serre duality (for the canonical bundle $T^*S^2=\mathcal{O}(-2)$),
$$
\begin{array}{lll}
\mathrm{coker}\, \overline{\partial} &=& H^{0,1}_{\overline{\partial}}(S^2,\underline{\C} \oplus \mathcal{O}(-1)) 
%
%
\cong H^1(S^2,\mathcal{O}(\underline{\C}) \oplus \mathcal{O}(-1))
\\[1mm] & \cong &
H^0(S^2,(\mathcal{O}(\underline{\C}) \oplus \mathcal{O}(-1))^{\vee} \otimes T^*S^2)^{\vee}
\\[1mm] & \cong &
H^0(S^2,\mathcal{O}(-2))^{\vee} \oplus H^0(S^2,\mathcal{O}(-1))^{\vee} 
%
%
=0.
\end{array}
$$
since $\mathcal{O}(-k)$ has no global holomorphic sections for $k\geq 1$. So $D_u$ is surjective, so $\hat{J}$ is regular for the constants. A small perturbation of $\hat{J}$ to make the other moduli spaces regular will not affect the count of constants, so to find $C$ we can use $\hat{J}$.

$C$ is the multiple of $[F]\in H_*^{lf}(M)$ corresponding to the chain swept out by evaluation at $z_{\infty}$ of the constant sections intersecting $[M]$ at $z_0$. The latter condition is void, so the chain is $[\C P^1]\in H_*^{lf}(M)$. The intersection pairing
$
H_*(M) \otimes  H_{4-*}^{lf}(M) \to \Z$ maps
$[\C P^1]  \otimes  [F] \mapsto 1$ and $
[\C P^1]  \otimes  [\C P^1] \mapsto  [\C P^1]\bullet [\C P^1] = -1
$. So $[\C P^1]=-[F]\in H_2^{lf}(M)$. Thus $C=-1$.
\end{proof}

\begin{remark*}
 For $\mathcal{O}(-n)$, regularity is proved in the same way, so $C\!=\!-n\!=\!c_1(\mathcal{O}(-n))$.
\end{remark*}

\begin{lemma}\label{Lemma Virtual dim for neg l blde}
 Sections in class $(1,d)$ for $d\geq 1$ form a moduli space isomorphic to 
$\mathcal{M}(\mathbb{P}^1\times \mathbb{P}^1;\beta=(1,d))$:
the rational curves in $\mathbb{P}^1\times \mathbb{P}^1$ in class $(1,d)$ (abbreviating $\mathbb{P}^1=\C P^1$) quotiented by $PSL(2,\C)$ reparametrization. Let $Z=\P^1\times \P^1$. We expect an obstruction bundle of rank$_{\R}=2d$ since:
$$
\begin{array}{rcl}
 \dim \mathcal{M}(Z;(1,d)) & = & 2(\dim_{\C}Z+c_1(Z)(\beta) - 3) = 2(2+2+2d-3)
=2+4d\\
\mathrm{virdim}\, \mathcal{M}(E_g;(1,d)) & = & 2+2d.
\end{array}
$$
\end{lemma}
\begin{proof}
Sections in class $(1,1)$ yield a degree $1$ holomorphic map $\pi_M\circ s: S^2 \to \C P^1$, because $\pi_M:M \to \C P^1$ is $(\hat{J},j)$ holomorphic since $\nu \circ j$ lands in the vertical tangent space of $M$. We quotient by the $PSL(2,\C)$ reparametrizations $u \mapsto u\circ \phi^{-1}$ to ensure $\mathbb{P}^1$ maps identically onto the first factor.
\end{proof}
\begin{lemma}\label{Lemma A coefficient neg l bdles}
 $r_{\widetilde{g}}= \left[\begin{smallmatrix} At & -1 \\ 0 & 0  \end{smallmatrix}\right]$, where $A$ is the count of holomorphic sections $S^2 \to E_g$ in the class $(1,1)$ $($after perturbing $J$ to achieve regularity$)$ which intersect $F$ over $z_0$ and a $($perturbed$)$ $\CP^1$ over $z_{\infty}$.
\end{lemma}
\begin{proof}
 The entries $B,D$ involve a count of sections which have some intersection condition at $z_0$ and which sweep out a multiple of $[M]$ under evaluation at $z_{\infty}$. However, even after perturbing $J$ to achieve regularity of the moduli space of sections, the maximum principle implies that the sections all land in a certain compact subset of $E_g$. So evaluation at $z_{\infty}$ involves a bounded lf chain in $M$. The multiple of $[M]$ is determined via Poincar\'e duality by intersecting with the point class. In homology, it does not matter which point we choose, so we can pick a point outside that compact subset of $M$, thus avoiding the bounded lf chain. So $B=D=0$.

 The entry $A$ involves the intersection condition $F$ at $z_0$, and $\C P^1$ at $z_{\infty}$ ($\CP^1$ is the cycle dual to the lf cycle $F$ via intersection product).
\end{proof}

\subsection{Calculation of $A$ using obstruction bundles}
\label{Calculation of A using obstruction bundles}
In our setup, for $M=\mathrm{Tot}(\mathcal{O}(-1)\to \CP^1)$, we want to count sections in class $\beta=(1,1)$:
$$ \boxed{\begin{array}{lll}
A &=& \mathrm{GW}_{0,2,\beta=(1,1)}^{E_g}((j_{z_0})_*[F],(j_{z_{\infty}})_*[\C P^1]) 
\\ &=& \# \{ u \in \mathcal{S}(j,\hat{J},t+S_{\widetilde{g}}): u(z_0) \in j_{z_0}(F), u(z_{\infty})\in j_{z_{\infty}}(\CP^1)  \}
\end{array}
}
$$
The standard $J$ on the fibre $M$ yields a non-regular $\hat{J}=\left[ \begin{smallmatrix} j & 0 \\ 0 & J \end{smallmatrix} \right]$ for the moduli space of sections in class $(1,1)$ by Lemma \ref{Lemma Virtual dim for neg l blde}, with rank$_{\R}=2$ obstruction bundle
$$(\mathrm{Obs} = \mathrm{coker}\, D_u) \to \mathcal{M}_{\hat{J}}$$
$$\mathcal{M}_{\hat{J}}=\{u\in \mathcal{M}(1,1)\cong PSL(2,\C): u(z_0)\in j_{z_0}(F), u(z_{\infty})\in j_{z_{\infty}}(P) \}$$
 where $D_u$ is the linearization of the $\overline{\partial}_{\hat{J}}$ operator defining $(j,\hat{J})$-holomorphic sections, and where $F$ is a generic fibre of $M$ and $P$ is a perturbation of $\CP^1$ (perturbing smoothly in the vertical direction, it will intersect the zero section of $M$ in a point).

\begin{lemma}\label{Lemma obs bundle result} Assuming that we can extend the obstruction bundle smoothly over a smooth compactification of $\overline{\mathcal{M}}_{\hat{J}}$ $($for which the tangent spaces are the kernels $\ker D_u)$, then the coefficient $A$ in Lemma \ref{Lemma A coefficient neg l bdles} is
$$
A = \mathrm{GW}_{0,2,(1,1)}^{E_g}(j_{z_0}F_1,j_{z_{\infty}}\CP^1) = \langle e(\overline{\mathrm{Obs}}),\overline{\mathcal{M}}_{\hat{J}} \rangle.
$$
\end{lemma}
\begin{proof}
We already discussed the first equality. The second equality is a standard cobordism argument analogous to \cite[Sec 7.2]{McDuff-Salamon2}. The idea is that one constructs a smooth family of bundles $\overline{\mathrm{Obs}}_{\hat{J}_t} \to \overline{\mathcal{M}}_{\hat{J}_t}$ such that $\overline{\partial}_{{\hat{J}_t}}$ lands in $\overline{\mathrm{Obs}}_{\hat{J}_t}$, starting at the given bundle at $t=0$ with ${\hat{J}_0}=\hat{J}$, and ending at $t=1$ with a regular admissible ${\hat{J}_1}$. By construction, the zero set of $\overline{\partial}_{\hat{J}_1}$ is the count of $(j,{\hat{J}_1})$-holomorphic sections of $E_g$ in class $(1,1)$ intersecting $F,P$ over $z_0,z_{\infty}$, since ${\hat{J}_1}$ is regular. The Euler number $\langle e(\overline{\mathrm{Obs}}_{\hat{J}_t}),\overline{\mathcal{M}}_{\hat{J}_t}\rangle$ is costant in $t$, and at $t=1$ equals the count of zeros of a section (such as $\overline{\partial}_{\hat{J}_1})$ transverse to the zero section. Hence 
$$A=\langle e(\overline{\mathrm{Obs}}_{\hat{J}_1}),\overline{\mathcal{M}}_{\hat{J}_1} \rangle =  \langle e(\overline{\mathrm{Obs}}),\overline{\mathcal{M}}_{\hat{J}} \rangle.$$

The family is constructed by choosing a homotopy from ${\hat{J}_0}=\hat{J}$ to a regular admissible ${\hat{J}_1}$ in a neighbourhood of $\hat{J}$ inside the space $\mathcal{J}$ of admissible almost complex structures on $E_g$. The family lives over ${\hat{J}_t}$ inside the larger bundle obtained by
extending $\overline{\mathrm{Obs}} \to \overline{\mathcal{M}}$ over a product neighbourhood $W$ of $\overline{\mathcal{M}} \times \{ \hat{J} \}$ inside $C^{\infty}(S^2,E_g) \times \mathcal{J}$ (and imposing the relevant intersection conditions).

This extension is done by an argument involving parallel transporting $\mathrm{Obs}_u \equiv {(\mathrm{im}\, D_u)}^{\perp}$ in directions orthogonal to $\ker D_u$ inside $\Omega^{0}(S^2,u^*TM)$ and then projecting onto $\Omega^{0,1}_J(S^2,u^*TM)$. 

For small $W$ (so we consider admissible ${\hat{J}'}$ close to $\hat{J}$) we can ensure $\mathrm{im}\, D_{u,{\hat{J}'}}$ and $\overline{\mathrm{Obs}}_{u,{\hat{J}'}}$ are transverse inside $\Omega^{0,1}_{\hat{J}'}(S^2,u^*TM)$ and we can ensure the evaluation at $z_0,z_{\infty}$ is transverse to the inclusions of $F,P$. This is because these conditions hold for $\hat{J}$. We therefore obtain a smooth parametrized moduli space
$$
\mathcal{M} = \{ (u,{\hat{J}'}) \in W: \overline{\partial}_{\hat{J}'}(u)\in \mathrm{Obs}_{u,{\hat{J}'}}, u(z_0)\in j_{z_0}(F), u(z_{\infty})\in j_{z_{\infty}}(P)\}
$$
and $\overline{\mathcal{M}}_{\hat{J}'}$ is obtained by compactifying the smooth subset obtained by fixing ${\hat{J}'}$.
\end{proof}
%
\subsection{Compactification of $\mathcal{M}$}
\label{Subsection Compactification of mathcal M}
$\mathcal{M} = \mathcal{M}_{\hat{J}} \subset \mathcal{M}_{0,2,\beta=(1,1)}(E_g)$ are curves intersecting $F,P$ over $z_0,z_{\infty}$, which lie in $S^2 \times \C P^1 \subset E_g$ by the maximum principle.
Simplify notation by writing $S^2 \times \C P^1 = \P^1 \times \P^1$, $z_0=0$, $z_{\infty}=\infty$. We may assume that $j_{z_0}F,j_{z_{\infty}}P$ intersect the zero section in $(0,0)$, $(\infty,\infty)$. Thus,
$$
\begin{array}{lll}
\mathcal{M} &=& \{ u: u(z)=(z,\varphi(z)), \varphi\in PSL(2,\C), \varphi(0)=0, \varphi(\infty)=\infty \}\\
&=& \{ u: u(z)=(z,a z), a\in \C^*\} \\
& \cong & \C^*.
\end{array}
$$
The compactification of $\C^*$ is $\P^1$, and is obtained by considering the limits $a\to 0, a\to \infty$. For example, consider $a\to 0$. Near $(0,0)$ the curve converges in $C^{\infty}$ to $z \mapsto (z,0)$, that is $\P^1 \times 0$. Near $(\infty,\infty)$ the curve can be parametrized as the locus $(\frac{1}{aw},\frac{1}{w})$, using a local fibre coordinate $w\in \C$ (where $w=0$ corresponds to $\infty$). So the reparametrized curve converges in $C^{\infty}$ to $\infty \times \P^1$. Thus, 
$a = 0$ corresponds to the curve $\P^1 \times 0$ with bubble $\infty \times \P^1$. Similarly, $a=\infty$ 
corresponds to the curve $\P^1 \times \infty$ with bubble $0\times \P^1$.
From now on, we write $\overline{\mathcal{M}} \cong \P^1$ for the compactification. 

\subsection{Description of $\mathbf{\mathrm{Obs}}$}
\label{Subsection Description of Obs for O-1 over P1}
Differential geometrically, $\mathrm{Obs}_u = \coker D_u$. We will now explain that, algebraic geometrically, 
$$\mathrm{Obs}_u = R^1 \pi_* f^*E_g$$
where $f:\mathcal{C} \to \P^1\times \P^1$ is the universal curve. 

\begin{definition}\label{Definition Universal curve} In our setup, the universal curve
$$
\xymatrix@R=14pt{
\mathcal{C} \ar[d]^{\pi} \ar[r]^-{f=\mathrm{ev}_3} & \P^1 \times \P^1 \\
\overline{\mathcal{M}}\cong \P^1
}
$$
is the space $\mathcal{C}$ consisting of $u\in \overline{\mathcal{M}}$ with an additional marked point $w$ on the domain, and $f$ is the evaluation $f(u,w)=u(w)$. Universality is because for $u\in \mathcal{M}$, $\P^1 \cong \pi^{-1}(u)$ is parametrized by $w$ and the composite $\P^1 \cong \pi^{-1}(u) \stackrel{f}{\to} \P^1 \times \P^1$ is $u$. 
\end{definition}

\begin{lemma}\label{Lemma first right derived functor O-1 case for P1}
 $\mathrm{Obs}_u = R^1 \pi_* f^*\mathcal{O}(-1,-1)$, where $R^1\pi_*$ is the $1^{st}$ right derived functor of the direct image functor \cite[III.8]{Hartshorne}. This is the compactification for Lemma \ref{Lemma obs bundle result}.
\end{lemma}
\begin{proof}
 Mimick Lemma \ref{Lemma constants neg l bdle}, but work in class $(1,1)$ instead of $(1,0)$. We claim that
$$u^*T^vE_g = u^*(T\P^1 \oplus \mathcal{O}(-1,-1)) = \mathcal{O}(2)\oplus \mathcal{O}(-2).$$
 This is proved by considering the map $\phi = (\pi_g,\pi_M)\circ u: \P^1 \to \mathcal{O}(-1,-1) \to \P^1 \times \P^1$. On cohomology it acts $H^2(\P^1 \times \P^1) \to H^2(\P^1)$ by pairing with $(1,1)$. Finally, use that $c_1$ is functorial and that $T\P^1 =\mathcal{O}(2)$ over (the second) $\P^1$.
$$
D_u =\overline{\partial}: \Gamma(\P^1,\mathcal{O}(2)\oplus \mathcal{O}(-2)) \to \Omega^{0,1}(\P^1,\mathcal{O}(2)\oplus \mathcal{O}(-2))
$$

Thus, omitting $\P^1$ references,
$$
\begin{array}{lll}
\mathrm{Obs} &=& \coker \overline{\partial} = H^{0,1}(\mathcal{O}(2)\oplus \mathcal{O}(-2)) \cong H^1(\mathcal{O}(2)\oplus \mathcal{O}(-2)) \\ &\cong& H^0((\mathcal{O}(2)\oplus \mathcal{O}(-2))^{\vee}\otimes \mathcal{O}(-2)) = H^0(\mathcal{O}(-4)\oplus \mathcal{O}) \\ &=& H^0(\P^1,\mathcal{O}) \quad (\textrm{complex 1 dimensional}.)
\end{array}
$$
So only the $\mathcal{O}(-1,-1)$ contributes to $\mathrm{Obs}$. By universality, the stalk is
$$
\mathrm{Obs}_u = H^1(\P^1,u^*\mathcal{O}(-1,-1)) \cong H^1(\pi^{-1}(u),f^*\mathcal{O}(-1,-1)) = (R^1\pi_*f^*\mathcal{O}(-1,-1))_u
$$
which shows that the map $\mathrm{Obs} \to R^1\pi_*f^*\mathcal{O}(-1,-1)$ (obtained similarly) is an isomorphism of sheaves. Since $R^1\pi_*f^*\mathcal{O}(-1,-1)$ makes sense also over the compactification, we may take that as the definition of $\overline{\mathrm{Obs}}$ in Lemma \ref{Lemma obs bundle result}.
\end{proof}

\begin{lemma}
 $f: \mathcal{C} \to \P^1 \times \P^1$ is the blow-up of $\P^1\times \P^1$ at $(0,0)$ and $(\infty,\infty)$.
\end{lemma}
\begin{proof}
 Consider $Q=f^{-1}(z_3,y_3)$. If $z_3\neq 0,\infty$ and $y_3 \neq 0,\infty$, then $Q$
 is a unique point in $\mathcal{C}$ corresponding to a curve (with additional marked point $(z_3,y_3)$).

For $(z_3=\infty, y_3\neq \infty)$ and $(z_3\neq 0, y_3=0)$, $Q$ is a point corresponding to $a=0$. For $(z_3\neq \infty,y_3=\infty)$ and $(z_3=0,y_3\neq 0)$, $Q$ is a point corresponding to $a=\infty$. 

On the other hand, $f^{-1}(0,0)\cong \P^1$, $f^{-1}(\infty,\infty) \cong \P^1$ corresponding to all $a\in \P^1$ (with additional marked point at $(0,0)$ and $(\infty,\infty)$ respectively).

So $f$ is a biholomorphism except over $(0,0),(\infty,\infty)$. One could argue that since $f$ is a birational morphism of algebraic surfaces it must be a composite of blow-ups. Alternatively, explicitly near $(0,0)$ (the case $(\infty,\infty)$ is similar) we have a parametrization for $\mathcal{C}$ given by $((z_3,y_3),[Z_3:Y_3])\in \C \times \C P^1$ satisfying $z_3 Y_3 =Z_3 y_3$, corresponding to $a=Y_3/Z_3$ with additional marked point $(z_3,y_3)$.
\end{proof}

\begin{theorem}\label{Theorem eul char of obs bundle}
 $\mathrm{Obs}=R^1\pi_*f^*\mathcal{O}(-1,-1) \to \overline{\mathcal{M}}$ is isomorphic to $\mathcal{O}(1)\to \P^1$, so
$$
A = \langle e(\mathrm{Obs}),\overline{\mathcal{M}}\rangle = \mathrm{degree}(\mathcal{O}(1)) = 1.
$$
\end{theorem}
\begin{proof}
 Let $\mathcal{F}=f^*\mathcal{O}(-1,-1)$. 

\textbf{Sub-claim.} 
$$
1-\mathrm{deg}(R^1\pi_* \mathcal{F}) = \int_{\mathcal{C}} \mathrm{ch}(\mathcal{F}) \mathrm{td}(T\mathcal{C}).
$$
\\
\emph{Proof.} Recall the direct image in $K$-theory \cite[Appendix A]{Hartshorne} for a proper morphism $g:X \to Y$ is
$
g_! =\sum (-1)^i R^i g_*: K(X) \to K(Y).
$

For $g: \P^1 \to \mathrm{point}$ and a vector bundle $\mathcal{G}$ on $\P^1$, by Riemann-Roch:
$$
\mathrm{rank}_{\C}\, \mathcal{G}+\mathrm{deg}(\mathcal{G}) = \chi_{holo}(\mathcal{G}) = h^0(\P^1,\mathcal{G})-h^1(\P^1,\mathcal{G}) = h^0(\mathrm{point}, g_!\mathcal{G}).
$$
Consider the composite $\mathcal{C} \stackrel{\pi}{\to} \P^1 \stackrel{g}{\to} \mathrm{point}$. Since $(g\pi)_*=g_*\pi_*$ also $(g\pi)_! = g_!\pi_!$, so:
$$
1+\mathrm{deg}(\pi_! \mathcal{F}) = h^0((g\pi)_!\mathcal{F}).
$$
Grothendieck-Riemann-Roch (see Fulton \cite[Sec.15.2]{Fulton}), written in $K$-theory, states: 
$$(g\pi)_!(\mathcal{F})\cdot \mathrm{td}(\mathrm{point}) = (g\pi)_*(\mathrm{ch}(\mathcal{F}) \cdot \mathrm{td}(\mathcal{C})).$$   
So, using $\mathrm{td}(\mathrm{point})=1$, and taking $h^0$, get:
$
1+\mathrm{deg}(\pi_! \mathcal{F}) = \int_{\mathcal{C}} \mathrm{ch}(\mathcal{F}) \wedge \mathrm{td}(\mathcal{C})
$
where we switched to cohomology notation on the right hand side (intersection product of complementary cycles is integration of the wedge product of the Poincar\'e dual cocycles, and we used that push-forward of a point is a point).

Finally, for dimensional reasons, $\pi_!\mathcal{F} = R^0\pi_*\mathcal{F} - R^1\pi_*\mathcal{F}$. Moreover, the $R^0$ term vanishes since it has stalk $H^0(\P^1,u^*\mathcal{O}(-1,-1))=H^0(\P^1,\mathcal{O}(-2))=0$ (geometrically: you cannot deform sections away from the zero section by the maximum principle). This proves the Sub-claim. $\checkmark$

In our case, the Todd class is
$$
\mathrm{td}(\mathcal{C})\equiv \mathrm{td}(T\mathcal{C}) = 1+\frac{1}{2}c_1 + \frac{1}{12}(c_1^2+c_2)  \in H^*(\mathcal{C},\Z)\otimes \Q,
$$
where we abbreviate $c_i = c_i(T\mathcal{C})$, and the Chern character is just
$$
\mathrm{ch}(\mathcal{F}) = e^z \in H^*(\mathcal{C},\Z)\otimes \Q,
$$
where $z=c_1(\mathcal{F}) = f^*c_1(\mathcal{O}(-1,-1))$.

Now we calculate the integral in the sub-claim, which expands to:
$$
\frac{1}{12} \int_{\mathcal{C}} c_2  + \frac{1}{12} \int_{\mathcal{C}} c_1^2 + \frac{1}{2} \int_{\mathcal{C}}  c_1 z + \frac{1}{2} \int_{\mathcal{C}}  z^2
$$

The first integral is the Euler characteristic of $\mathcal{C}$, which is $6$, since $\mathcal{C}$ has Betti numbers $1,0,4,0,1$ (the homology of $\P^1\times \P^1$ with two additional exceptional $\P^1$).

Recall the following four facts \cite[Prop II.3]{Beauville} about intersection products of divisors in a blow-up $\pi: R \to S$ of algebraic surfaces at a point with exceptional divisor $E$: $\pi^*D\cdot \pi^*D'=D\cdot D'$, $E\cdot \pi^*D=0$, $E\cdot E=-1$, $K_{R}=\pi^*K_{S}+E$ (where $K_S$ is the canonical divisor class corresponding to $T^*S$).

The last fact implies: $T^*\mathcal{C}=f^*T^*(\P^1\times \P^1) + (E_1+E_2)$ (in $K$-theory), where $E_1,E_2$ are the two exceptional fibres of $f$. Thus, by the other three facts, and because $E_1,E_2$ don't intersect:
$$T\mathcal{C}^2=T(\P^1\times \P^1)^2 -2 = \langle (2,2),(2,2)\rangle -2=6,
$$
so the second integral $\int_{\mathcal{C}} c_1^2=6$.

By the second fact, working in $K$-theory, the third integral is:
$$
\begin{array}{lll}
c_1\cdot z &=& (f^*T(\P^1\times \P^1) + (E_1+E_2)) \cdot f^*\mathcal{O}(-1,-1) \\ &=& T(\P^1\times \P^1)\cdot \mathcal{O}(-1,-1) \\ &=& \langle (2,2),(-1,-1)\rangle = -4
\end{array}
$$
The last integral is
$
f^*\mathcal{O}(-1,-1) \cdot f^*\mathcal{O}(-1,-1) = \mathcal{O}(-1,-1)^2 = \langle (-1,-1),(-1,-1) \rangle =2.
$
Therefore:
$$
1-\mathrm{deg}(\mathrm{Obs}) = \int_{\mathcal{C}} \mathrm{ch}(\mathcal{F}) \wedge \mathrm{td}(\mathcal{C}) = \frac{6}{12} + \frac{6}{12}-\frac{4}{2} + \frac{2}{2} = 0.
$$
Thus $\mathrm{deg}(\mathrm{Obs})=1$, and line bundles over $\P^1$ are classified by their degree.
\end{proof}
%
\subsection{Symplectic cohomology of $\mathcal{O}(-1)\to \CP^1$}

\begin{theorem}\label{Theorem SH of O-1 over P1}
 Let $M$ be the total space of $\mathcal{O}(-1)\to \C P^1$. Then $SH^*(M) \cong \Lambda \cdot 1$, and $c^*: QH^*(M) \to SH^*(M)$ maps $c^*(1)=1$, $c^*(\omega_{\C P^1})=-t\cdot 1$.
\end{theorem}
\begin{proof}
Combining Theorem \ref{Theorem eul char of obs bundle} with Lemmas \ref{Lemma obs bundle result} and \ref{Lemma A coefficient neg l bdles} we obtain 
$$
r_{\widetilde{g}} = \left[\begin{smallmatrix} t & -1 \\ 0 & 0 \end{smallmatrix}\right]: \Lambda^2 \to \Lambda^2.
$$
So by Theorem \ref{Theorem SH is lim of S-bRb}, $SH^*(M) \cong \Lambda \cdot 1$, where $1=\psi^{-}(1)\in SH^{\textrm{even}}(M)$ is the unit. Recall $\omega_{\C P^1}=\mathrm{PD}([F])$, $1=\mathrm{PD}([M])$, so $r_{\widetilde{g}}(\omega_{\CP^1})=t\omega_{\CP^1}$ and $r_{\widetilde{g}}(1)=-\omega_{\CP^1}=c_1(\mathcal{O}(-1))$. This represents the continuation $SH^*(H_0) \to SH^*(H_1)$, after identifications with $QH^*(M)$, and this in turn is identified with $c^*$ yielding:
$$
SH^*(M)=QH^*(M)/\ker r_{\widetilde{g}} = \Lambda[\omega_{\CP^1}]/(\omega_{\CP^1}+t\cdot 1).\qedhere
$$
\end{proof}
%
%
\section{Symplectic cohomology of $M=\mathrm{Tot}(\mathcal{O}(-n)\to \P^m)$}
\label{Section O(-n) over CPm}
%
%
\subsection{Description of $\mathbf{M=\mathrm{Tot}(\mathcal{O}(-n) \to \P^m)}$}
\label{Subection Gradings for O(-n) over CPm}
%
Let $M=\mathrm{Tot}(\mathcal{O}(-n) \to \P^m)$.
From now on, we always use complex dimensions to avoid factors of $2$ everywhere. $H^*(\P^m)$ is generated by $\omega^m,\omega^{m-1},\ldots, \omega,1$, where $\omega=\pi_M^*\omega_{\P^m}$, $\omega_{\P^m}[\P^1\!\subset\! \P^m]\!=\!1$. Poincar\'{e} dually, $H_*^{lf}(\P^m)$ is generated by lf cycles $F_1,F_2,\ldots,F_m,F_{m+1}=M$ where $F_j=\pi_M^{-1}(\P^{j-1})$ for some equatorial $\P^1 \subset \P^2 \subset \cdots \P^{m-1}\subset \P^m$, and $j=\dim_{\C} F_j$. 

These lf cycles are dual, with respect to the intersection product, to the cycles $\P^m,\P^{m-1},\ldots,\P^1,pt=\P^0$ since $\P^{1+m-j}\bullet \P^{j-1}=1$ in $\P^m$. The condition of sweeping out $F_j$ at $z_{\infty}$ is thus equivalent to the intersection condition over $z_{\infty}$ with its dual: the (perturbed) $\P^{1+m-j}$. For genericity, one needs to perturb: for $0<j\leq m$, the cycle $\P^j$ can be perturbed vertically (in the smooth category) to a cycle which intersects the zero section in $-n[\P^{j-1}]$, which is the Poincar\'e dual of the Euler class of $\mathcal{O}(-n)$ ($\mathcal{O}(-n)$ pulls back to $\mathcal{O}(-n) \to \P^j$ via $\P^j \hookrightarrow \P^m$).

This time, $c_1(TM)[\P^1] = c_1(T\P^m)[\P^1] + c_1(\mathcal{O}(-n))[\P^1] = 1+m-n$. So define
$$
\boxed{N = 1+m-n}
$$

As before $\Lambda = \Z[t^{-1},t]]$ as $\pi_2(M)$ is generated by $t=[\P^1]$, and $|t|=-2N$ (homological grading). 
So weak$^+$ monotonicity holds except in a small range:
\begin{enumerate}
 \item $\mathbf{1\leq n < 1+m}$: $M$ is monotone \checkmark ($c_1(TM)$ is a positive multiple of $\omega_M$);

 \item $\mathbf{n=1+m}$: critical case: $c_1(TM)=0$ \checkmark (so $SH^*(M)=0$ by Theorem \ref{Theorem c1=0 implies SH=0});

 \item $\mathbf{2+m \leq n \leq 2m}$: this is the range where weak$^+$-monotonicity fails. There may be technical issues in constructing $r_{\widetilde{g}}$ so we will not discuss this; 

 \item $\mathbf{n\geq 1+2m}$: $|N| \geq \dim_{\C} \P^m = m$ \checkmark (and $SH^*(M)=0$ by Corollary \ref{Corollary O-n special cases n=2m+1 2m+2}).

\end{enumerate}

The space of $(j,\hat{J})$-holomorphic sections has complex dimension
$$
\mathrm{virdim}_{\C}\, \mathcal{S}(t^d+S_{\widetilde{g}}) = \dim_{\C} M + c_1(TE^v_g)(S_{\widetilde{g}})+dc_1(TM)(t) = 1+m-1+dN = m+Nd.
$$
The intersection condition at $z_0$ with $F_j$ cuts this down by $1+m-j$. Therefore,
$$
\begin{array}{lll}
\mathrm{virdim}_{\C}\, (\mathcal{S}(t^d+S_{\widetilde{g}})\cap \mathrm{ev}_{z_0}^{-1}(F_j) \cap \mathrm{ev}_{z_{\infty}}^{-1}(\P^{1+m-i})) &=& m+Nd - (1+m-j) - i \\ &=& Nd-i+j-1.
\end{array}
$$
So provided $\boxed{Nd-i+j-1=0}$ this contributes to the entry $(i,j)$ of the matrix $r_{\widetilde{g}}$ viewed as an $(m+1)\times (m+1)$ matrix over $\Lambda$ in the basis $F_1,F_2,\ldots,F_{m+1}$ (or cohomologically: in the basis $\omega^m,\omega^{m-1},\ldots,1$).
\begin{lemma}\label{Lemma rg for O-n over Pm}
 The constants are always regular for $\hat{J}=\left[\begin{smallmatrix} j & 0 \\ 0 & J \end{smallmatrix}\right]$, $J$  integrable, and $r_{\widetilde{g}}$ has the following form in the basis $\omega^m,\omega^{m-1},\ldots,\omega,1$:
$$ r_{\widetilde{g}} = \left[
\begin{smallmatrix}
  0 & -n & 0 & \cdots \\ 
  \vdots & 0 & -n & 0 & \cdots \\
 0   & \vdots  & \vdots \\
A_0 t & 0 & \cdots & 0 & -n & 0 & \cdots \\
  0   & A_1 t & \cdots & 0 & 0 & -n & 0 & \cdots \\
 \vdots   & \vdots  & \vdots \\
B_0 t^2 & 0 & \cdots & 0 & 0 & 0& 0 & -n & 0 & \cdots \\
  0     & B_1 t^2 & \cdots & 0 & 0 & 0& 0 & 0 & -n & 0 & \cdots \\
 \vdots   & \vdots  & \vdots \\
\cdots & \cdots & \cdots & \cdots & \cdots & \cdots & \cdots & \cdots & \cdots & \cdots & -n  \\
0 & 0 & 0 & \cdots & 0 & 0& 0 & 0& 0 & 0 & 0 
\end{smallmatrix}
\right]
$$
The $-n=c_1(\mathcal{O}(-n))[\P^1]$ arise on the second main diagonal, they count constant sections.
The $A_0,B_0,C_0,\ldots$ in positions $(N,1),(2N,1),(3N,1),\ldots$ and the corresponding subdiagonals with entries $A_a,B_a,C_a,\ldots$ count sections in class $\beta=(1,1),(1,2),(1,3),\ldots$ All other entries are zero. Moreover:
$$
\begin{array}{lll}
A_a &=& \mathrm{GW}_{0,2,(1,1)}^{E_g}(j_{z_0}F_{a+1},j_{z_{\infty}}\P^{1+m-a-N}) = \mathrm{GW}_{0,2,(1,1)}^{E_g}(j_{z_0}F_{a+1},j_{z_{\infty}}\P^{n-a})
\\
B_a &=& \mathrm{GW}_{0,2,(1,2)}^{E_g}(j_{z_0}F_{a+1},j_{z_{\infty}}\P^{1+m-a-2N)})\\
C_a &=& \mathrm{GW}_{0,2,(1,3)}^{E_g}(j_{z_0}F_{a+1},j_{z_{\infty}}\P^{1+m-a-3N})\\
\cdots
\end{array}
$$
where $j_{z_0},j_{z_{\infty}}: M \to E_g$ are the inclusions of the fibres over $z_0,z_\infty\in \P^1$.
\end{lemma}
\begin{proof}
 For regularity of constants see Theorem \ref{Theorem consts are regular in general and rg1 is c1}. For $d=0$, the constants sweep out the lf cycle $[\P^m]$ under $\mathrm{ev}_{z_{\infty}}$. So the contribution to $r_{\widetilde{g}}(F_j)$ is $[\P^m]\cap F_j=[\P^{j-1}]$. Intersecting with a perturbed $\P^{1+m-i}$, where $i=j-1$, is $-n[\mathrm{pt}]$ (the perturbation hits the zero section in $-n[\P^{m-i}]$). So constants contribute $-nF_{j-1}$ to $r_{\widetilde{g}}(F_j)$. The last row vanishes because it involves an intersection condition with a point, which we can move to infinity (without affecting $r_{\widetilde{g}}(1)$ in cohomology), so the moduli spaces will never interesect it by the maximum principle. The rest is by dimensions. 
\end{proof}

\begin{corollary}\label{Corollary O-n special cases n=2m+1 2m+2}
 For $n> 2m$, $\mathrm{virdim}_{\C} \mathcal{S}(t^d+S_{\widetilde{g}}) = m+Nd< m-md$ so only $d=0$ occurs, so $r_{\widetilde{g}}$ only has a supdiagonal of $-n$'s, so $r_{\widetilde{g}}$ is nilpotent, so $SH^*(M)=0$.
\end{corollary}

Arguing as in Lemma \ref{Lemma first right derived functor O-1 case for P1}, for $u$ in class $(1,d)$,
$$
u^*T^vE_g = u^*(T \P^m \oplus \mathcal{O}(-1,-n)) = \mathcal{O}(2d)\oplus \mathcal{O}(d) \oplus \cdots \oplus \mathcal{O}(d) \oplus \mathcal{O}(-1-nd)
$$
with $m-1$ copies of $\mathcal{O}(d)$. Here we used that fact that $\P^1 \subset \P^m$ has tangent bundle $\mathcal{O}(2)$ and normal bundle $\nu_{\P^1\subset \P^m} = \nu_{\P^1\subset \P^2}\oplus \nu_{\P^2\subset \P^3} \oplus \cdots \oplus \nu_{\P^{m-1}\subset \P^m}$, and $c_1(\nu_{\P^{j-1}\subset \P^j})=c_1(T\P^j|_{\P^{j-1}})-c_1(T\P^{j-1})=1\cdot \omega_{\P^{j-1}}$. Thus:
$$
\begin{array}{lll}
\ker \overline{\partial} &=& H^0(\P^1,\mathcal{O}(2d)\oplus \mathcal{O}(d)^{\oplus m-1} \oplus \mathcal{O}(-1-nd))\\
&=& H^0(\P^1,\mathcal{O}(2d)\oplus \mathcal{O}(d)^{\oplus m-1})\\
\coker \overline{\partial} &=& H^1(\P^1,\mathcal{O}(2d)\oplus \mathcal{O}(d)^{\oplus m-1} \oplus \mathcal{O}(-1-nd)) \\&\cong & H^0(\P^1, \mathcal{O}(-2d-2) \oplus \mathcal{O}(-d-2)^{\oplus m-1} \oplus \mathcal{O}(nd-1))^{\vee} \\ &\cong & H^0(\P^1,\mathcal{O}(nd-1))^{\vee}
\end{array}
$$
So the obstruction bundle $\mathrm{Obs}$ has $\mathrm{rank}_{\C}=nd$. To determine $r_{\widetilde{g}}$, all $0\leq d \leq \frac{m}{1+m-n}$ will contribute for $n<1+m$.
The $A_a,B_a,\ldots$ are in principle determined by $<e(\mathrm{Obs}),[\mathcal{M}]>$ where $\mathcal{M}$ is the (compactified) moduli space of sections cut down by the relevant intersection conditions described before the Lemma. In practice $\mathrm{Obs}$ becomes rapidly unwieldy for $n\neq 1$, $d>1$. We now study $n=1$ explicitly.

\subsection{Explicit description for $\mathbf{M=\mathrm{Tot}(\mathcal{O}(-1)\to \P^m)}$}
\label{Subection rg for O(-n) over Pm}
%
\begin{lemma}\label{Lemma O-1 over Pm by QH}
 For $M=\mathrm{Tot}(\mathcal{O}(-1)\to \P^m)$, $$\begin{array}{l}                                                                                                  
r_{\widetilde{g}}=\left[
\begin{smallmatrix} 
0 & -1 & 0 & 0 & \cdots \\
0 & 0 & -1 & 0 & \cdots \\[-1mm]
\vdots \\
0 & 0 & 0 & 0 & \cdots & -1 & 0\\
t & 0 & 0 & 0 & \cdots & 0 & -1\\
0 & 0 & 0 & 0 & \cdots & 0 & 0
\end{smallmatrix} \right]\\ \strut\\[-2mm]
SH^*(M) = \Lambda[\omega_Q]/(\omega_Q^m+t \cdot 1)
 \end{array}
$$
 and $c^*: QH^*(M) \to SH^*(M)$ maps $c^*(1)=1$, $c^*(\omega_Q)=\omega_Q$, $c^*(\omega_Q^m) =-t\cdot 1$.
\end{lemma}
\begin{proof}
We only need to find the entry $A_0$. This involves $d=1$, and intersection conditions over $z_0$ with the fibre $F_1$ and over $z_{\infty}$ with $\P^1$. Perturbing $\P^1$ vertically, it will intersect the zero section in $-\mathrm{pt}$. The holomorphic sections of $E_g$ lie in the zero section, and we want those in class $(1,1)$ which intersect $(0,0)$, $(\infty,\infty)$ (where in the second entry, we can assume that $0,\infty \in \P^1 \subset \P^m$ are the intersections of $F$ and $(\P^1 \textrm{ perturbed})$ with the zero section). So we reduce to maps $$\P^1 \to \P^1 \times \P^1 \subset \P^1 \times \P^m,$$ where the first maps are the same as in \ref{Subsection Description of Obs for O-1 over P1}, and the second map is the inclusion. That inclusion pulls back $\mathcal{O}(-1,-1)$ to $\mathcal{O}(-1,-1)$, so the same Grothendieck-Riemann-Roch argument proves $A_0=1$. The rest follows as in Theorem \ref{Theorem SH of O-1 over P1}.
\end{proof}

\begin{theorem}
$QH^*(M) = \Lambda[\omega_Q]/(\omega_Q^{m+1}+t\cdot \omega_Q)$ for $\mathcal{O}(-1) \to \P^m$.
\end{theorem}
\begin{proof}
Denote $\omega$ the canonical generator of $H^2(\P^m)$. We denote $\omega^k$ the ordinary cup product powers, and $\omega_Q^k$ the quantum cup product powers.

For $\mathcal{O}(-n) \to \P^m$ we first calculate for each $j=1,\ldots,m$:
$$
\omega * \omega^j = \sum_{\ell = 1+j-dN} \mathrm{GW}_{0,3,d}^M (F_m,F_{m+1-j},\P^{\ell})\cdot t^d \cdot \omega^{\ell}
$$
where we used that $\omega^{\ell}=\mathrm{PD}(F_{m+1-\ell})$ and $\P^{\ell}=\mathrm{D}(F_{m+1-\ell})$ (where $\mathrm{PD}$ is Poincar\'e duality and $\mathrm{D}$ is intersection duality), and we used the (complex) GW dimension condition $(1)+(j)+(1+m-\ell) = (1+m)+Nd+3-3$.

For $\mathcal{O}(-1) \to \P^m$ we have $N=1+m-n=m$, and since $1\leq j \leq m$ we have $0\leq \ell = 1+j-dN\leq 1+m-dm$, so $d=0$ or $1$. For $d=0$ we count constant $\P^1 \to M$, so the lf cycle we get under evaluation is $M$ and the GW invariant counts
$$M \cap F_m \cap F_{m+1-j} \cap \P^{\ell}=\P^{m-1}\cap \P^{m-j} \cap \P^{\ell} = \P^{\ell-1-j} = \P^0
$$
so we get the ordinary cup product contributions $\omega*\omega^j = \omega^{1+j}+\cdots$. The case $d=1$ forces $j=m-1$ or $j=m$. For $j=m-1$: $\P^0$ can be moved to infinity so GW$=0$. Finally consider $j=m$, $\ell=1$. Regularity of degree $d=1$ holomorphic $u: \P^1 \to \P^m \subset M$ follows from $u^*TM=u^*T\P^m \oplus \mathcal{O}(-1)$, $u^*T\P^m\cong \mathcal{O}(2)\oplus \mathcal{O}(1)^{\oplus m-1}$,
$$\begin{array}{lll}
\mathrm{coker}\, \overline{\partial} &\cong & H^1(\P^1,\mathcal{O}(2)\oplus \mathcal{O}(1)^{\oplus m-1}\oplus \mathcal{O}(-1)) \\
&\cong & H^0(\P^1,\mathcal{O}(-4)\oplus \mathcal{O}(-3)^{\oplus m-1}\oplus \mathcal{O}(-1))^{\vee}=0.
\end{array}
$$
For $d=1,j=m$, we impose intersection conditions $F_m,F_{1},\P^1$. Perturbing that $\P^1$ off the zero section, these three conditions inside the zero section become conditions $\P^{m-1},\P^0,-1\cdot \mathrm{pt}$. There is a unique holomorphic $\P^1$ through two points, and it automatically intersects the $\P^{m-1}$, so $\mathrm{GW}_{0,3,1}^M(F_m,F_{1},\P^1)=-1$.

Conclusion: $\omega*\omega^j=\omega^{1+j}$ for $j=1,\ldots,m-1$, so $\omega_Q^{1+j}=\omega^{1+j}$, and
$$
\omega_Q^{1+m}=\omega*\omega^m_Q = \omega*\omega^m = \omega^{1+m}-t\omega = -t\omega.
$$
So $r_{\widetilde{g}}(1)*\omega^n=(-\omega)*\omega^n=t\omega$ confirming Lemma
\ref{Lemma O-1 over Pm by QH} via Theorem \ref{Theorem Intro r_g on QH and SH}. 
\end{proof}

\subsection{Quantum cohomology of $\mathbf{M=\mathrm{Tot}(\mathcal{O}(-n) \to \P^m)}$}

\begin{corollary*}
Quantum cup product by $c_1(\mathcal{O}(-n))=-n\omega$ defines the matrix $r_g$ of Lemma \ref{Lemma rg for O-n over Pm} in the basis $\omega^m,\ldots,\omega,1$, and so
$$
\begin{array}{lll}
A_a &=& -n\cdot\mathrm{GW}^M_{0,3,1}(F_m,F_{a+1},\P^{1+m-a-N})
= -n\cdot\mathrm{GW}^M_{0,3,1}(F_m,F_{a+1},\P^{n-a})\\
B_a &=& -n\cdot\mathrm{GW}^M_{0,3,2}(F_m,F_{a+1},\P^{1+m-a-2N})\\
C_a &=& -n\cdot\mathrm{GW}^M_{0,3,3}(F_m,F_{a+1},\P^{1+m-a-3N})\\
\cdots
\end{array}
$$
\end{corollary*}

\begin{remark*}
The obstruction bundle involved in calculating the $A_a,B_a,C_a,\ldots$ in this way has fiber $H^0(\P^1,\mathcal{O}(nd-2))^{\vee}$ of $($complex$)$ rank $nd-1$.
\end{remark*}
%
\subsection{Linear algebra}
\label{Subsection linear algebra}
Let $M=\mathrm{Tot}(\mathcal{O}(-n)\to \P^m)$ (although what we say applies also to the cyclic subgroups of $QH^*,SH^*$ generated by $c_1(L)$ for $M=\mathrm{Tot}(L\to B)$).

 Let $c=c_1(\mathcal{O}(-n))=-n\omega_Q$. Taking quantum cup product powers of $c$ yields $c_Q^{m},c_Q^{m-1},\ldots,c_Q,1$, which is a basis for $QH^*(M)$ in characteristic $0$ $($and for odd $n$ in characteristic $2$). The $r_{\widetilde{g}}$ in this basis turns into the canonical form:
$$
\left[\begin{smallmatrix}
 -a_1 & 1 & 0 & 0 & \cdots \\
 -a_2 & 0 & 1 & 0 & \cdots \\
 \vdots\\
 -a_m & 0 & \cdots &  & 0 & 1 \\
 0 & 0 & \cdots & & &  0 
\end{smallmatrix}
\right]
$$
where $\lambda^{m+1} + a_1  \lambda^m + a_2 \lambda^{m-1} + \cdots + a_m \lambda$ is the characteristic polynomial of $r_{\widetilde{g}}$. Here $a_i= 0$ if $i$ is not divisible by $|N|$, and $a_i$ is homogeneous in $t$ of order $t^{N/i}$.

Since $r_{\widetilde{g}}$ is quantum cup product by $c$,
$$
QH^*(M) \equiv \Lambda[c_Q]/(c_Q^{m+1}+a_1c_Q^m+\cdots + a_{m}c_Q).
$$

Suppose there is a largest integer $m\geq p\geq 1$ for which $a_p\neq 0$ $($otherwise $c_Q^{m+1}=0$ and $SH^*(M)=0)$. Then the characteristic polynomial of $r_{\widetilde{g}}$ is 
$$
\lambda^{m+1-p}(\lambda^p + a_1 \lambda^{p-1} + \cdots + a_p).
$$
Since rank $r_{\widetilde{g}}=m$, the above implies the Jordan normal form of $r_{\widetilde{g}}$ has exactly one Jordan block for eigenvalue $0$ of size $m+1-p$. Thus, for $k\geq m+1-p$, $\ker r_{\widetilde{g}}^k$ is the generalized eigenspace of $r_{\widetilde{g}}$ for eigenvalue $0$ which is
$$
\ker r_{\widetilde{g}}^k = \Lambda\cdot (\lambda^p + a_1 \lambda^{p-1} + \cdots + a_p)
$$
\emph{Remark: }$\mathrm{image}(r_{\widetilde{g}}^k)=\mathrm{span}(c_Q^m,\ldots,c_Q^{1+m-p})$ stabilizes for $k\geq m+1-p$.

It follows by Theorem \ref{Theorem Intro r_g on QH and SH} that $SH^*(M)$ has rank $p$ since
$$
SH^*(M)\cong \Lambda[c_Q]/(c_Q^{p}+a_1 c_Q^{p-1} + \cdots + a_p).
$$

\begin{lemma}\label{Lemma aN calculation}
 $ a_{N}=(-1)^N n^{N-1}\sum_{j=0}^{n-1} A_j t$
 \;$($where $N=1+m-n\geq 0)$.
\end{lemma}
\begin{proof}
If the matrix in Lemma \ref{Lemma rg for O-n over Pm} had $-n$'s replaced by $-1$ and $A_j,B_j,\ldots$ replaced by $\widetilde{A}_j,\widetilde{B}_j,\ldots$, then one can easily check that the characteristic polynomial would be $\lambda^{m+1} + \widetilde{a}_N \lambda^{m+1-N} + \cdots$ where $\widetilde{a}_N=(-1)^N\sum \widetilde{A}_jt$. If we replace $r_{\widetilde{g}}$ by $r_{\widetilde{g}}/n$ the matrix has that form with $\widetilde{A}_j= A_j/n$. Under this replacement, the characteristic polynomial changes from $\lambda^{1+m}+a_N \lambda^{1+m-N}+\cdots$ to $\lambda^{1+m} + a_N n^{-N}\lambda^{1+m-N}+\cdots$ So $a_N n^{-N}=\widetilde{a}_N=(-1)^N\sum A_jt/n$.
\end{proof}

\begin{corollary}\label{Corollary linear algebra 2N m case}
For $2N>m$ $($equivalently $n<1+\frac{m}{2})$ only the $A_j$ contribute to $r_{\widetilde{g}}$, and the only non-zero $a_i$ is $a_{N}=-(-n)^{N-1}\sum A_j t$, so putting $\alpha=\sum A_j$:
$$
\begin{array}{lll}
 QH^*(M) = \Lambda[c_Q]/(c_Q^{1+m} -(-n)^{N-1} \alpha t c_Q^{n}) = \Lambda[\omega_Q]/(\omega_Q^{1+m} +n^{-1} \alpha t \omega_Q^{n}) \\[2mm]
SH^*(M) = \Lambda[c_Q]/(c_Q^{N}-(-n)^{N-1}\alpha t)
= \Lambda[\omega_Q]/(\omega_Q^{N}+n^{-1}\alpha t)
\end{array}
$$
where $N=1+m-n$, and in Theorem \ref{Theorem Aj calculation} we calculate $A_j$.
\end{corollary}

\subsection{Calculation of $A_a$ by virtual localization}

We follow closely the notation of Pandharipande's notes \cite{Pandharipande}, which are based on Graber-Pandharipande \cite{Pandharipande2}. Localization was first applied to stable maps by Kontsevich \cite{Kontsevich}. We also mention Cox-Katz \cite[p.277]{Cox-Katz} as a good reference. As a warm-up we redo the $\mathcal{O}(-1) \to \P^1$.

\begin{theorem}
 For $\mathcal{O}(-1) \to \P^1$, $A_0=1$.
\end{theorem}
\begin{proof}
 Consider the deformation long exact sequence \cite[p.549]{Pandharipande},
$$
0 \to \mathrm{Aut}(C) \to \mathrm{Def}(u) \to \mathrm{Def}(C,u) \to \mathrm{Def}(C) \to \mathrm{Ob}(u) \to \mathrm{Ob}(C,u) \to 0
$$
where $C=(\Sigma,x_1,x_2)$ is a $2$-pointed nodal curve of arithmetic genus $0$, and $u: C \to E_g$ are the sections in class $(1,1)$ that we want to count in Lemma \ref{Lemma A coefficient neg l bdles}. The following observations clarify how our setup is different from \cite{Pandharipande}:

\begin{enumerate}
 \item The marked points $x_1,x_2$ are fixed in our setup, indeed as in \ref{Subsection Compactification of mathcal M} we choose $x_1=p_{00}=(0,0)$, $x_2=p_{11}=(\infty,\infty)$. 

\item The holomorphic maps we consider are
$$u:C \to \P^1\times \P^1 \subset E_g$$
in class $(1,1)$, having already imposed the intersection conditions $F,P$ - so we use the moduli space $\overline{\mathcal{M}}$ of \ref{Calculation of A using obstruction bundles}.  We will often refer to the second $\P^1$ as $\CP^1\subset M$ to distinguish it from the first factor.

\item The open part $\mathcal{M}$ are maps of the form $u(z)=(z,az)$. The compactification gives rise to two new stable maps $U_{10},U_{01}: \Sigma_1 \cup \Sigma_2 \to \P^1\times \P^1$, where $C$ is a nodal curve with two $\P^1$'s joined at one node $v$. The first map is specified by: $U_{10}(\Sigma_1)=\P^1\times 0$, $U_{10}(\Sigma_2)=\infty\times \P^1$, $U_{10}(v)=p_{10}=(\infty,0)$. The second: $U_{01}(\Sigma_1)=0\times \P^1$, $U_{01}(\Sigma_2)=\P^1\times \infty$, $U_{01}(v)=p_{01}=(0,\infty)$.

 \item \label{Item torus action} The torus action by $\mathbb{T}=(\C^*)^2 \ni t$ on $\P^1\times \P^1$ is: 
$$([z_0:z_1],[w_0:w_1])\mapsto ([z_0:z_1],[t_0^{-1}w_0:t_1^{-1}w_1]).$$
(the inverses ensure that the action on linear forms in $H^0(\P^1,\mathcal{O}(1))$ involves no inverses).
This induces a natural action on $u(z)=([z_0:z_1],[z_0:az_1])$:
$$(t\cdot u)(z)= ([z_0:z_1],[t_0^{-1}z_0:at_1^{-1}z_1]).$$
 Denote $\alpha_0,\alpha_1$ the weights for $t$.

\item The $\mathbb{T}$-fixed points of $\overline{\mathcal{M}}$ are the two maps $U_{01},U_{10}$. We call $\Gamma_{10},\Gamma_{01}$ the decorated graphs which describe $U_{10},U_{01}$ (explicitly: graphs with two edges, and vertices labeled by $00,10,11$ and $00,01,11$ respectively).

 \item\label{Item which deformations} Because of the intersection conditions, we only consider deformations of $u$ subject to the conditions $u(0)=p_{00}$, $u(\infty)=p_{11}$. There are no reparametrization automorphisms on the main component of $u$ because we only consider sections. There are $PSL(2,\C)$-reparametrization automorphisms for the bubbles arising in the $M$-fibres of $E_g$.

 \item $E_g$ plays the same role as $\P^m$ in \cite[27.6]{Pandharipande}, however we do not consider deformations of $u$ in all $TE_g$-directions, but rather only in $T^vE_g$-directions since we are working with sections. Recall $T^vE_g\cong T\CP^1 \oplus \mathcal{O}(-1,-1)$. So 
$$
\begin{array}{lllll}
(U_{10}^*T^vE_g)|_{\Sigma_1} \cong T_0\CP^1\oplus \mathcal{O}(-1), & (U_{10}^*T^vE_g)|_{\Sigma_2} = TM \cong T\CP^1 \oplus \mathcal{O}(-1),\\
(U_{01}^*T^vE_g)|_{\Sigma_2} \cong T_{\infty}\CP^1\oplus \mathcal{O}(-1), & (U_{01}^*T^vE_g)|_{\Sigma_1} = TM \cong T\CP^1 \oplus \mathcal{O}(-1)
\end{array}
$$

\end{enumerate}

We use the convention of \cite{Pandharipande} that we refer to the fiber of a vector bundle when we mean the vector bundle. In our setup, $\mathrm{Ob}(C,u)=0$ since there are no contracted components in our stable maps. The obstruction bundle is $\mathrm{Ob}(u)=H^1(C,u^*T^vE_g)$, but the deformation bundle $\mathrm{Def}(u)$ is not all of $H^0(C,u^*T^vE_g)$ because of the intersection conditions.

By (\ref{Item which deformations}), $\mathrm{Def}(u)^{\mathrm{mov}}=0$ (the section of $\mathcal{O}(2)$ vanishing at $0,\infty$ has weight zero, so contributes to $\mathrm{Def}(u)^{\mathrm{fix}}$ and it cancels with the bubble reparametrization automorphisms $\mathrm{Aut}(C)^{\mathrm{fix}}$ in the deformation LES). Also by (\ref{Item which deformations}): $\mathrm{Aut}(C)^{\mathrm{mov}}=0$.

By the Atiyah-Bott localization theorem, we want to calculate:
$$
A_0= \int_{\overline{\mathcal{M}}} e(\mathrm{Obs}) = \int_{\overline{\mathcal{M}}^{\mathrm{vir}}} 1 = 
i_{\mathrm{point}}^*\int_{\overline{\mathcal{M}}_{\mathbb{T}}} 1 = \sum_{\Gamma} \frac{1}{e^{\mathbb{T}}(N^{\mathrm{vir}}_{\Gamma})}
$$
where we sum over our two graphs $\Gamma=\Gamma_{10}$ and $\Gamma_{01}$, and where the equivariant Euler class of the virtual normal bundle to the fixed points $U_{10},U_{01}$ is:
$$
e^{\mathbb{T}}(N^{\mathrm{vir}}_{\Gamma}) = 
\frac{
e(\mathrm{Def}(u)^{\mathrm{mov}})\, 
e(\mathrm{Def}(C)^{\mathrm{mov}})
}
{
e(\mathrm{Ob}(u)^{\mathrm{mov}})\, 
e(\mathrm{Aut}(C)^{\mathrm{mov}})
}
=
\frac{
e(\mathrm{Def}(C)^{\mathrm{mov}})
}
{
e(\mathrm{Ob}(u)^{\mathrm{mov}})\, 
}
$$

Now $\mathrm{Def}(C)^{\mathrm{mov}}$ comes from resolving the node $v$ of $\Sigma_1\cup \Sigma_2$. By the \emph{boundary lemma} \cite[25.2.2]{Pandharipande}, the relevant normal bundle associated to this smoothing is $T_v\Sigma_1 \otimes T_v\Sigma_2$. The action on these tangent spaces is induced by the action on the image under the isomorphisms $U_{10}:\Sigma_1 \to \P^1 \times 0$, $U_{01}: \Sigma_2 \to \infty\times \P^1$ for $\Gamma_{10}$, and similarly for $U_{01}$. Recall that if $\mu_0,\mu_1$ are weights for a torus action on $\P^1$ then the weights for $T_0\P^1$,$T_{\infty}\P^1$ are respectively $\mu_0-\mu_1,\mu_1-\mu_0$.
So the weight for the action on the above tensor for $U_{10}$, $U_{01}$ respectively are:
$$0+(\alpha_0-\alpha_1) \qquad (\alpha_1 - \alpha_0) + 0.
$$

Finally, consider $\mathrm{Ob}(u)^{\mathrm{mov}}$. The only contributions come from $\mathcal{O}(-1,-1)$. The normalizing sequence for the node for $u=U_{10}$ is:
$$
0 \to u^*\mathcal{O}(-1,-1) \to \mathcal{O}_{\Sigma_1}(-1)\oplus \mathcal{O}_{\Sigma_2}(-1) \to u^*\mathcal{O}_{p_{10}}(-1,-1) \to 0
$$
Taking the LES in cohomology, using that $H^1(\P^1,\mathcal{O}(-1))=0$, we deduce:
$$
\begin{array}{lll}
\mathrm{Ob}(u)^{\mathrm{mov}} &=& H^0(\Sigma,u^*\mathcal{O}_{p_{10}}(-1,-1)) \equiv H^0(\P^1\times \P^1,\mathcal{O}_{p_{10}}(-1,-1)) \\ &=& \mathcal{O}(-1,-1)|_{p_{10}\in \P^1\times \P^1}.
\end{array}
$$
In general, the action on $\mathcal{O}(-1,-1)$ induced by the $\mathbb{T}$-action on $\P^1\times \P^1$ has weights $-\rho_{ij}$ if $\rho_{ij}$ are the weights for $\P^1\times \P^1$ indexed by its fixed points $p_{ij}$. In our case, we obtain weight $-\alpha_0$. Similarly, for $U_{01}$ we obtain $\mathcal{O}(-1,-1)|_{p_{01}}$ and weight $-\alpha_1$.
$$
A_0 = \frac{-\alpha_0}{\alpha_0-\alpha_1} + \frac{-\alpha_1}{\alpha_1-\alpha_0} = \frac{-\alpha_0 + \alpha_1}{\alpha_0-\alpha_1} = -1.
$$
$A_0$ actually needs to be rescaled by $-n=-1$, because the perturbed $P$ intersects the zero section in $-n[\mathrm{pt}]$. This will become clearer in the next proof. 
\end{proof}

\begin{definition}\label{Definition tau a n}
 Let $\tau_{a,n}$ denote the coefficient of $x^a$ in the degree $n-1$ polynomial 
$$
\prod_{\begin{smallmatrix} A\geq 1,B\geq 1\\ A+B=n \end{smallmatrix}} (Ax+B),
$$
and define $\tau_{0,1}=1$. Observe that $\sum_a \tau_{a,n} = \prod (Ax+B)|_{x=1} = \prod n = n^{n-1}$.

In characteristic $2$ and odd $n$, $\prod (Ax+B) \equiv x^{\frac{n-1}{2}}$, so $\tau_{a,n}\equiv 0$ except for $\tau_{\frac{n-1}{2},n}=1$, and $\sum_a \tau_{a,n}\equiv 1$. For even $n$, $\tau_{a,n}\equiv 0$ except when $n=2$: $\tau_{0,2} = \tau_{1,2} = 1$.
\end{definition}

\begin{theorem}\label{Theorem Aj calculation}
 For $\mathcal{O}(-n) \to \P^m$, $A_a = (-1)^{n-1}n^2\tau_{a,n}$ $($assuming $n<1+m)$.
\end{theorem}
\begin{proof}
$A_a=\mathrm{GW}_{0,2,(1,1)}^{E_g}(j_{z_0}F_{a+1},j_{z_{\infty}}\P^{n-a})$.
We choose $F_{a+1}=\pi_M^{-1}(\P^a)$ where $\P^a\subset \P^m$ involves only the first $a+1$ homogeneous coordinates. We perturb $\P^{n-a}$ vertically so that it intersects the zero section in $-n[ \P^{n-a-1}]$. We can ensure $\P^{n-a-1} \subset \P^m$ involves only the last $n-a$ homogeneous coordinates (notice $\P^a,\P^{n-a-1}$ do not intersect since $n<1+m$). We will calculate the contribution of each $+[\P^{n-a-1}]$ separately, so we need to rescale the final answer by $-n$.

The $\mathbb{T}=(\C^*)^{m+1}$ action on $\P^1\times \P^m$ is analogous to (\ref{Item torus action}) above, acting on $\P^m$ with weights $\alpha_0,\ldots,\alpha_m$. The fixed points in $\P^m$ are labeled $q_{\ell}$ having entry $w_{\ell}=1$ and all other entries $w_r=0$. Abbreviate $p_0=[1:0]=0,p_1=[0:1]=\infty\in \P^1$ and
$$p_{k{\ell}}=p_k\times q_{\ell} \in \P^1 \times \P^m$$
where $k=0,1$ and ${\ell}=0,1,\ldots,m$. 
Among this $\ell$ indexing, we reserve the letter $i=0,1,\ldots,a$ and the letter $j=m-(n-a-1),\ldots,m$. These labels index the fixed points $q_i \in \P^a\subset \P^m$ and $q_j\in \P^{n-a-1}\subset \P^m$.

 The open part of the moduli space $\mathcal{M}$ are holomorphic $u:\P^1 \to \P^1 \times \P^m \subset E_g$ satisfying the intersection conditions 
$$u(0)\in p_0\times \P^a\qquad u(\infty)\in p_1\times\P^{n-a-1}.$$ 
So they are lines which are geometrically determined by the intersection conditions. The union of all points lying on such lines spans a certain $\P^{n}\subset \P^m$.

Explicitly, given $[\vec{x}]\in \P^a$, $[\vec{y}]\in \P^{n-a-1}$, the line $[z_0:z_1]\mapsto [z_0 \vec{x} + z_1 \vec{y}]$ is the unique geometric line through $[\vec{x}],[\vec{y}]$. However, the parametrization is not canonical: there is a $\P^1$-freedom to reparametrize. Thus $\mathcal{M}$ is a $\C^*$-bundle over $\P^a\times \P^{n-a-1}$. The compactification $\overline{\mathcal{M}}$ to a $\P^1$-bundle is just fiberwise the same as the one we did for the $m=n=1$ case: a bubble appears in the $M$-fiber of $E_g$ over $p_0$ or over $p_1$.
The universal curve is again a blow-up:
$$
\xymatrix{
\mathcal{C}=\mathrm{Bl}(\P^1\times \P^n,\P^a \sqcup \P^{n-a-1})
 \ar[d]^{\pi} \ar[r]^-{f=\mathrm{ev}_3} & \P^1 \times \P^n\subset \P^1 \times \P^m \subset E_g \\
\overline{\mathcal{M}} = (\P^1\textrm{-bundle over }\P^a\times \P^{n-a-1})
}
$$
The induced $\mathbb{T}$-action on $\overline{\mathcal{M}}$ is analogous to (\ref{Item torus action}). The fixed stable maps $u:\Sigma_1\cup \Sigma_2 \to \P^1\times \P^m$ are indexed $U_{1ij}$ and $U_{ij0}$, meaning: $u(0)=p_{0i}$, $u(\infty)=p_{1j}$, 
$$U_{1ij}(\mathrm{node})=p_{1i}, \qquad U_{ij0}(\mathrm{node})=p_{0j}.$$ 
The graphs $\Gamma_{1ij},\Gamma_{ij0}$ have two edges and labelling $0i,1i,1j$ and $0i,0j,1j$ respectively.
In this setup, $T^vE_g=T\P^m\oplus \mathcal{O}(-1,-n)$ and
$$
\begin{array}{lllll}
(U_{1ij}^*T^vE_g)|_{\Sigma_1} \cong T_{q_i}\P^m\oplus \mathcal{O}(-1), & (U_{1ij}^*T^vE_g)|_{\Sigma_2} = \mathcal{O}(2)\oplus \mathcal{O}(1)^{m-1} \oplus \mathcal{O}(-n),\\
(U_{ij0}^*T^vE_g)|_{\Sigma_2} \cong T_{q_j}\P^m\oplus \mathcal{O}(-1), & (U_{ij0}^*T^vE_g)|_{\Sigma_1} = \mathcal{O}(2)\oplus \mathcal{O}(1)^{m-1} \oplus \mathcal{O}(-n)
\end{array}
$$
where $\mathcal{O}(2)\oplus \mathcal{O}(1)^{m-1}$ comes from pulling back $T\P^m$.

$\mathrm{Def}(C)^{\mathrm{mov}}$ comes from resolving the node, giving opposite weights
\begin{equation}\label{Equation 1}
 \alpha_i -\alpha_j \qquad \alpha_j-\alpha_i
\end{equation}
respectively for $U_{1ij}, U_{ij0}$. This time, $\mathrm{Def}(u)$ has moving parts because we can deform the image of the fixed marked points $x_1,x_2$ within $\P^a,\P^{n-a-1}$ respectively. This yields two summands: $T_{q_i}\P^a$ and $T_{q_j}\P^{n-a-1}$, which have weights 
\begin{equation}\label{Equation 2}
\alpha_i-\alpha_I \qquad \alpha_j-\alpha_J
\end{equation}
 where $0\leq I \leq a$, $I\neq i$ and $m-(n-a-1)\leq J \leq m$, $J\neq j$.

For $\mathrm{Ob}(u)^{\mathrm{mov}}$ only $\mathcal{O}(-1,-n)$ contributes, the normalizing sequence for $U_{1ij}$ is:
$$
0 \to u^*\mathcal{O}(-1,-n) \to \mathcal{O}_{\Sigma_1}(-1)\oplus \mathcal{O}_{\Sigma_2}(-n) \to u^*\mathcal{O}_{p_{1i}}(-1,-n) \to 0
$$
and taking the LES in cohomology we deduce 
$$\begin{array}{lll}
   \mathrm{Ob}(U_{1ij})^{\mathrm{mov}}=\mathcal{O}(-1,-n)|_{p_{1i}\in \P^1\times \P^m}\oplus H^1(\Sigma_2,\mathcal{O}_{\Sigma_2}(-n))\\
\mathrm{Ob}(U_{ij0})^{\mathrm{mov}}=\mathcal{O}(-1,-n)|_{p_{0j}\in \P^1\times \P^m} \oplus H^1(\Sigma_1,\mathcal{O}_{\Sigma_1}(-n))
  \end{array}
$$
The first summands yield the following weights for $U_{1ij}, U_{ij0}$ respectively:
\begin{equation}\label{Equation 3}
-n\alpha_i \qquad  - n\alpha_j.
\end{equation}
We now seek the weights for the $H^1$ summands. We consider the case $u=U_{1ij}$. By Serre duality, $H^1(\Sigma_2,\mathcal{O}_{\Sigma_2}(-n))\cong H^0(\Sigma_2,K_{\Sigma_2}\otimes \mathcal{O}_{\Sigma_2}(n))^{\vee}$. The weights for the canonical bundle $K_{\Sigma_2}=T^*\Sigma_2$ at $p_{1i},p_{1j}$ are respectively $\alpha_j-\alpha_i$ and $\alpha_i-\alpha_j$. The $\mathcal{O}_{\Sigma_2}(n)$ comes from pulling back $\mathcal{O}(1,n)$ via an embedding $\Sigma_2 \hookrightarrow \P^1\times \P^m$, and the weights for $\mathcal{O}(1,n)$ at $p_{1i},p_{1j}$ are $n\alpha_i$ and $n\alpha_j$. The total weights on $K_{\Sigma_2}\otimes \mathcal{O}_{\Sigma_2}(n)$ are therefore $\alpha_j+(n-1)\alpha_i$ and $\alpha_i+(n-1)\alpha_j$.

Since $\mathrm{deg}(K_{\Sigma_2}\otimes \mathcal{O}_{\Sigma_2}(n))=n-2$, it follows \cite[27.2.3]{Pandharipande} that
the weights on $H^0(\Sigma_2,K_{\Sigma_2}\otimes \mathcal{O}_{\Sigma_2}(n))^{\vee}$ are
\begin{equation}\label{Equation 4}
-\left\{ \frac{a}{n-2} [\alpha_j+(n-1)\alpha_i] + \frac{b}{n-2}[\alpha_i+(n-1)\alpha_j] \right\}= -( A\alpha_i + B\alpha_j)
\end{equation}
where $a+b=n-2$ and $a,b\geq 0$, and where we simplified the expression using $A=a+1,B=b+1$, so $A,B\geq 1$ and $A+B=n$. We remark that the global minus sign in \eqref{Equation 4} appears because the $H^0(\Sigma_2,K_{\Sigma_2}\otimes \mathcal{O}_{\Sigma_2}(n))$ group is dualized.

Similarly, for $U_{ij0}$ we get weights $-(A\alpha_i + B\alpha_j)$.

We now apply virtual localization, so we calculate $\sum \frac{1}{e^{\mathbb{T}}(N_{\Gamma}^{\mathrm{vir}})}$: 
$$ 
\sum_{i,j}\frac{\displaystyle [(-n\alpha_i)-(-n\alpha_j)]\prod_{A,B} -(A\alpha_i+B\alpha_j)}{\displaystyle (\alpha_i-\alpha_j)\prod_{I}(\alpha_i-\alpha_I)\prod_{J}(\alpha_j-\alpha_J)}
=
-n(-1)^{n-1}\sum_{i,j}\frac{\prod (A\alpha_i+B\alpha_j)}{\prod(\alpha_i-\alpha_I)\prod(\alpha_j-\alpha_J)}
$$
This is supposed to be an integer: this can be verified taking common denominators:
$$
-n(-1)^{n-1}\frac{\sum_{i,j}\prod(A\alpha_i+B\alpha_j)\prod(\alpha_{\widehat{i}}-\alpha_I)\prod(\alpha_{\widehat{j}}-\alpha_J)}{\prod(\alpha_c-\alpha_d)\prod(\alpha_p-\alpha_q)}
$$
where $c\neq d$ vary in $\{0,1,\ldots,a\}$; $p\neq q$ in $\{m-(n-a-1),\ldots,m\}$; $\hat{i}$ in $\{0,1,\ldots,a\}\setminus i$; and $\hat{j}$ in $\{m-(n-a-1),\ldots,m\}\setminus j$; and $A,B,I\neq \hat{i}$, $J\neq \hat{j}$ are as usual. One then needs to show that each factor on the denominator, such as $(\alpha_c-\alpha_d)(\alpha_d-\alpha_c)$, divides the numerator. This amounts to checking that the numerator vanishes to order $2$ when putting $\alpha_c=\alpha_d$.

To find that integer value we consider the fraction as a Laurent polynomial in one variable, say $\alpha_0$, with coefficients in the ring $\Z(\alpha_1,\ldots,\alpha_m)$. Since only the $\alpha_0^0$ term survives, we can drop all terms of different order. The denominator $\prod(\alpha_i-\alpha_I)$ of the original sum has order $\alpha_0^a$, and the numerator has no $\alpha_0$ terms unless $i=0$. So we can put $i=0$. Now, we can let $\alpha_0\in \R$ and $\alpha_0\to \infty$, so only this survives:
$$
-n(-1)^{n-1} \sum_j \frac{\alpha_j^{n-a-1}\tau_{a,n}}{\prod (\alpha_j-\alpha_J)}
$$
Let $\alpha_m\in \R$ and $\alpha_m\to \infty$, so only the $j=m$ term survives:
$
-n(-1)^{n-1} \tau_{a,n}.
$
Finally, recall from the beginning of the proof that we need to rescale the final answer by $-n$.
\end{proof}

\textbf{Proof of Theorem \ref{Theorem Intro O-n over Pm}.}
Over characteristic $0$ $($Remark \ref{Remark orientation signs}$)$, in Lemma \ref{Lemma aN calculation}: $a_{N}=(-1)^N n^{N-1} (\sum (-1)^{n-1}n^2 \tau_{a,n}) t =(-1)^{N+n-1} n^{1+m} t= -(-n)^{1+m}$ by the previous Theorem. For $n<1+\frac{m}{2}$,
$$
\begin{array}{rcccl}
QH^*(M)\hspace{-2mm} &=& \hspace{-2mm} \Lambda[c_Q]/(c_Q^{1+m}- (-n)^{1+m} tc_Q^n) \hspace{-2mm} &=& \hspace{-2mm}
\Lambda[\omega_Q]/(\omega_Q^{1+m}- (-n)^{n} t\omega_Q^n)\\[2mm]
SH^*(M)\hspace{-2mm} &=& \hspace{-2mm} \Lambda[c_Q]/(c_Q^{N}-(-n)^{1+m} t)  \hspace{-2mm} &=& \hspace{-2mm} \Lambda[\omega_Q]/(\omega_Q^N - (-n)^{n} t).
\end{array}
$$
For $n<1+m$, $QH^*(M)=\Lambda[\omega_Q]/(\omega_Q^{1+m}-(-n)^{n} t\omega_Q^n+\cdots)$ may have lower order correction terms from $d\geq 2$ contributions, but we still deduce $SH^*(M)\neq 0$. For $n$ even, vanishing in characteristic $2$  is because $c_1(L)[\P^1]=-n\equiv 0$ mod $2$. $\qed$
%
\section{General theory for negative line bundles $M=\mathrm{Tot}(\pi_M:L\to B)$}
\label{Section General theory for negative line bundles}


\begin{theorem}\label{Theorem consts are regular in general and rg1 is c1}
 For $M=\mathrm{Tot}(\pi_M:L\to B)$ (satisfying weak$^+$ monotonicity), the constant sections are regular for $\hat{J}=\left[ \begin{smallmatrix} j & 0 \\ 0 & i \end{smallmatrix} \right]$ and they determine 
$$r_{\widetilde{g}}(1)=(1+\lambda_+)\, \pi_M^*c_1(L)$$
where $\lambda_+$ lies in the subring $ \Lambda_+^0\subset \Lambda$ generated by the $\pi_2(M)$-classes with $\omega>0$ and $c_1(TM)=0$ $($for monotone $M$, $\lambda^+=0)$. In particular, $(1+\lambda^+)$ is a unit of $\Lambda$, so for the purposes of calculating $SH^*(M)$, we may rescale $r_{\widetilde{g}}(1)=\pi_M^*c_1(L)$.
\end{theorem}
\begin{proof}
 Consider the dimension of the moduli space of sections
$$
\dim_{\C} \mathcal{S}(j,\hat{J},\gamma+S_{\widetilde{g}})=b + c
$$
where $b=\dim_{\C}B$ and $c=c_1(TM,\omega)(\gamma)$. Since $M$ is weak, $c\geq 0$ or $c\leq -b$. Since $r_{\widetilde{g}}(1)$ sweeps an lf cycle, we may assume $1\leq b+c \leq b$ ($[\mathrm{pt}]$ is a boundary lf cycle, and we cannot sweep $[M]$ by the maximum principle). Combining: $c=0$. In the monotone case, this implies the sections are constant, so $\gamma=0$. In general, constant sections of $E_g$ are regular for the integrable $\hat{J}$ by mimicking Lemma \ref{Lemma constants neg l bdle}:
$$
\coker \overline{\partial} = H^1(\P^1,\mathcal{O}(\underline{\C}^{\oplus \dim_{\C}B})\oplus \mathcal{O}(-1)) \cong H^0(\P^1,\mathcal{O}(-2)^{\oplus \dim_{\C}B}\oplus \mathcal{O}(-1))^{\vee} = 0.
$$

The constant sections of $E_g$ sweep the lf cycle $\mathrm{ev}_{\infty}(\mathcal{S}(j,\hat{J},S_{\widetilde{g}}))=[B]$ (not $[M]$: the transition map for $E_g$ over the equator of $S^2$ rotates the fibres of $L$, only the fixed point set $B$ of $g$ will give rise to constant sections of $E_g$). We will now show that the lf cycle $[B]$ is Poincar\'e dual to $\pi_M^*c_1(L)$. Consider a $1$-cycle $\alpha\subset B$: 
$$
\alpha\bullet_M [B] = \alpha\bullet_B (\textrm{zero set of a generic }C^{\infty}\textrm{-section}).
$$
The zero set is obtained by perturbing $[B]$, it represents PD$_B(c_{\mathrm{top}}(L))$ in $B$. Pull-back $\pi_M^*\!:\!H^*(B)\!\to\! H^*(M)$ in cohomology is Poincar\'{e} dual to taking pre-images $\pi_M^{-1}\!:\!H_*(B)\! \to\! H_{*+2}^{lf}(M)$ (Bott-Tu \cite[Sec.6]{Bott-Tu}). So $\pi_M^{-1}\mathrm{PD}_B(c_1(L))\!=\!\mathrm{PD}_M(\pi_M^*c_1(L))$.

In the non-monotone case, it may happen that $c=0$ but $\omega(\gamma)>0$. The only lf $2b$-cycles supported near the zero section are multiples of $[B]$ (generator of $H_{2b}(B)$), so $\mathcal{S}(j,\hat{J},\gamma+S_{\widetilde{g}})$ is a multiple of $[B]$. These determine $\lambda_+$. 
\end{proof}

\begin{example*}
 For $\mathcal{O}(-n)\to \P^m$: the lf cycle $[\P^m]$ when perturbed vertically will intersect the zero section in $-n[\P^{m-1}]$, so the intersection number $\P^1 \bullet [\P^m] = -n$ in $M$, so $[\P^m]=-n[F_m] \in H_{2m}^{lf}(M)$, so $\mathrm{PD}[\P^m]=-n[\omega_{\P^m}]=c_1(\mathcal{O}(-n))\in H^2(M)$.
\end{example*}

\begin{remark*}
 $\lambda_+ \neq 0$ can occur only for non-monotone $M$, and a base $B$ admitting a holomorphic map $v: \P^1 \to B$ through any given point with $c_1(TB)(v)=n\omega_B(\pi_M v)$.  
\end{remark*}

\begin{remark}[\textbf{Orientation Signs and char$\mathbf{(\Lambda)\neq 2}$}]\label{Remark orientation signs}
 For regular integrable complex structures the moduli spaces of holomorphic curves are canonically oriented $($see \cite[Rmk 3.2.6, p.51]{McDuff-Salamon2}$)$ and the $0$-dimensional ones always contribute with sign $+1$. So 
the dominant term $\pi_M^*c_1(L)$ in $r_{\widetilde{g}}(1)$ is correct also if we work over $\Lambda$ of characteristic zero $($e.g. in \ref{Subsection geometrical novikov ring} replace $\Z/2$ by $\Z$ or $\Q)$. Lemma \ref{Lemma HF H0 and QH are iso as rings}, $QH^*(M)\cong HF^*(H_0)$, still holds $($orientation signs for $SH^*(M)$ and its product structure were constructed by the author in \cite{Ritter3}$)$. Therefore Theorem \ref{Theorem Intro r_g on QH and SH} holds also over characteristic zero.
\end{remark}
%
\section{Negative vector bundles}
\label{Section Negative vector bundles}
%
%
\subsection{Definition of negative vector bundles and choice of symplectic form}
\label{Subsection Definition and properties neg vector bdle}

\begin{definition}\label{Definition negative vector bundle}
A complex vector bundle $E\to B$ over a closed symplectic manifold $(B,\omega_B)$ is \emph{negative} if $E$ admits a Hermitian metric, and some Hermitian connection whose curvature $\mathcal{F}\in \Omega^2(B,\mathfrak{u}(E))$ satisfies $$\frac{i}{2\pi}\mathcal{F}(v,J_B v)<0$$ for all $v\neq 0 \in TB$ (meaning that is a negative definite Hermitian endomorphism of $E$), for all almost complex structures $J_B$ compatible with $\omega_B$.
\end{definition}

\begin{lemma}\label{Lemma neg v bdles are symplectic}
The total space $M$ of a negative vector bundle $E\to B$ is symplectic (but non-conical) for the form $\omega=\pi_{E}^*\omega_B + \Omega$, defined using the connection above:
$$
\begin{array}{l}
\Omega = \frac{1}{\pi} (\mathrm{area} \; \mathrm{form}) \textrm{ on vertical vectors \emph{(}in a local unitary frame for }E) \\
\Omega_{(b,w)}(\cdot,\cdot) = \frac{1}{2\pi i} w^{\dagger}   \mathcal{F}_{(d\pi_E\cdot,d\pi_E\cdot)} w \textrm{ on horizontal vectors, for }w\neq 0 \\
\Omega_{(b,0)}(TB,\cdot) = 0\\
\Omega_{(b,w)}(h,v)=0 \textrm{ if }h\textrm{ is horizontal, and }v \textrm{ is vertical.}
\end{array}
$$
and $\omega$ is compatible with $J=J_B\oplus i$ acting on $T^{horiz}E\oplus T^{vert}E$.
\end{lemma}
\begin{proof}
We start by a standard trick from algebraic geometry. Consider the (complex) projectivization $\P(E)$ of $E$. Let $L=\mathcal{O}(-1)\to \P(E)$ be the tautological line bundle, so $L$ is just $\mathcal{O}(-1) \to \P^{\mathrm{rank}(E)-1}$ over each fibre of $\P(E)\to B$:
$$
\xymatrix@C=16pt@R=14pt{
L \ar[d]_{\pi_L} & E \ar[d]^{\pi_E} & & \mathcal{O}(-1) \ar[d] & E_b  \ar[d] \\
\P(E) \ar[r]_-{\pi_{\P}} & B & & \P^{\mathrm{rank}\, E-1}  \ar[r] & b
}
$$
By Leray-Hirsch (see \cite[Appendix A]{Hartshorne} or \cite[p.134]{Audin-Lafontaine}): $H^*(\P(E))$ is a free $H^*(B)$-module via $\pi_{\P}^*$, generated by $1,c_1(L),\ldots,c_1(L)^{\mathrm{rank}(E)-1}$.

%

A horizontal distribution for $E$ yields a horizontal distribution for $L$ spanned by the horizontal vectors of $E$ and the horizontal vectors of each $\mathcal{O}(-1) \to \P^{\mathrm{rank}(E)-1}$.

Suppose $E$ is negative, and pick a Hermitian metric and connection as in the definition. It is carefully proved in Oancea \cite[Sec.3.4]{Oancea} that this determines a canonical Hermitian metric and Hermitian connection on $\P(E)$ and $L$, and that this determines a canonical symplectic form on $\mathrm{Tot}(\P(E))$ given by the curvature
$$
\omega_{\P} = -\tfrac{i}{2\pi} \mathcal{F}^L \quad (\textrm{hence } c_1(L)=-[\omega_{\P}])
$$
which restricts to the normalized Fubini-Study form on each $\P^{\mathrm{rank}(E)-1}$. So $L$ is a negative line bundle. Explicitly, the curvature at $(b,[w])$, where $[w]=\C w$ is a line in the fibre $E_b$, is
$$
\begin{array}{ll}
\mathcal{F}^L_{(\cdot,\cdot)} = \frac{1}{r^2} w^{\dagger} \mathcal{F}^E_{(d\pi_{\P}\cdot,d\pi_{\P}\cdot)}  w & \textrm{on horizontal vectors of }\P(E)\to B \\ & \textrm{(identifiable with horizontal vectors of }E)\\
\mathcal{F}^L_{(\cdot,\cdot)} = \mathcal{F}^{\mathcal{O}(-1)\to \P^{\mathrm{rank(E)-1}}}_{(\cdot,\cdot)} & \textrm{on }\ker d\pi_{\P} \textrm{, the vertical vectors of }\P(E)\to B \\ \hspace{5.5ex} &  \textrm{(identifiable with } (\C w)^{\perp} \subset T^{\textrm{vert}}_{(b,w)} E \equiv E_b)
\end{array}
$$
and mixed terms vanish (the horizontal vectors of $\P(E) \to B$ are $\omega_{\P}$-orthogonal to the vertical ones). Let $\theta^L$, $\Omega^L=d(r^2\theta^L)$ be as in \ref{Subsection Definition and properties neg line bdle}, so $d\theta^L=\frac{1}{2\pi i}\pi_L^*\mathcal{F}^L$. Let 
$$\tau: \mathrm{Tot}(E\setminus 0_E) \to \mathrm{Tot}(L\setminus 0_L)$$
be the tautological isomorphism (a point on the right is a choice of complex line in a fibre of $E$ together with a choice of vector in that line, so it is point in $\mathrm{Tot}(E)$). 

Outside the zero section of $E$, define
$$
\theta =\tau^*(\theta^L), \quad \Omega=\tau^*(\Omega^L) = d(r^2\theta)
$$
The angular form $\theta_w = \frac{1}{2\pi r^2} \langle iw,\cdot \rangle$ (taking the vertical component of $\cdot$ and using the Hermitian metric) is $U(\mathrm{rank}\,E)$-invariant, and fiberwise $\Omega=(\textrm{area form})/\pi$. As in \ref{Subsection Definition and properties neg line bdle}, $\Omega$ extends over the zero section. To finish proving $\Omega$ is as claimed, we use Lemma \ref{Lemma dtheta properties}: for $w\neq 0$, and horizontal vectors $h,h'$ of $E$,
$$
d\theta_w \cdot (h,h') = \theta_{[w],w}^L(\tfrac{1}{r^2} w^{\dagger} \mathcal{F}^E_{(d\pi_E h,d\pi_E h')} w) = \tfrac{1}{2\pi i r^2} w^{\dagger}\mathcal{F}^E_{(d\pi_E h,d\pi_E h')}w
$$
using that $\mathcal{F}^E$ is skew-Hermitian and $d\pi_{\P}d\pi_L d\tau=d\pi_E$. 

We now prove $\omega=\pi_E^*\omega_B + \Omega$ is symplectic. Let $J_B$ be $\omega_B$-compatible, then we obtain an almost complex structure $J=J_B\oplus i$ on $T^{horiz}E\oplus T^{vert}E=TE$ ($J_B$ canonically lifts to an action on horizontal vectors). On $h\neq 0\in T_{(b,w)}^{horiz}E$, 
$$
\omega(h,J h) = \omega_B(d\pi_E h,J_B d\pi_E h) + \frac{1}{2\pi i} w^{\dagger} \mathcal{F}_{(d\pi_E h,J_B d\pi_E h)} w >0
$$
using negativity of $E$ (omitting the second term if $w=0$). On $\ker d\pi_E$, $\Omega$ is the area form so it is symplectic and $i$-compatible. So $\omega$ is symplectic and $J$-compatible.
\end{proof}

\begin{remark*}
 For $\mathrm{rank}_{\C}E\geq 2$ any conical $\omega$ would be exact by the LES for the pair: $0\!=\!H^2(E,E\setminus 0) \!\to \!H^2(E) \!\to\! H^2(E\setminus 0)$. So the zero section would not be symplectic.
\end{remark*}

\subsection{The maximum principle for negative vector bundles}
\label{Subsection Remarks about the maximum principle for negative vector bundles}
The total space $M=\mathrm{Tot}(E\to B)$ is not conical (see the previous Remark), so we need to reprove the maximum principle, Lemma \ref{Lemma Maximum principle}. The strategy is to consider the negative line bundle $\pi_L:L\to \mathbb{P}(E)$ constructed in the proof of Lemma \ref{Lemma neg v bdles are symplectic}, and to reduce the problem to the known maximum principle for $L$.

Recall from Section \ref{Subsection construction of the symplectic form neg line bdle} that the symplectic form chosen on $L$ is $\omega_L=\pi_L^*\omega_{\mathbb{P}} + \Omega^L$ and that outside of the zero section this is exact and equal to $d(\theta^L + r^2 \theta^L)$ (the forms $\omega_{\mathbb{P}}$, $\Omega^L$, $\theta^L$ were constructed in the proof of Lemma \ref{Lemma neg v bdles are symplectic}). Recall that we can identify the complements of the zero sections,
$$
E\setminus 0_E \equiv L\setminus 0_L,
$$
in particular the radial coordinates agree, but the symplectic forms do not (indeed the problem is that the form $\omega=\pi_E^*\omega_B + \Omega$ for $E$ is \emph{not} exact at infinity).

We know the maximum principle applies to $(L\setminus 0_L, \omega_L)$, by Lemma \ref{Lemma Maximum principle}, so we just need to ensure that our choices of $(J,H)$ on $E$ will satisfy the assumptions of that maximum principle on $L\setminus 0_L \equiv E\setminus 0_E$. So we reduce to ensuring that:
\begin{enumerate}
 \item the almost complex structure $J$ on $E$ (which is compatible with $\omega$) is also compatible with $\omega_L$ and is of contact type on $L\setminus 0_L$;
 
 \item the Hamiltonians $H$ used for $E$ have the form $h(R^L)$ on $L\setminus 0_L$, with $h$ linear in the $R$-coordinate $R^L=(1+r^2)/2$ for $L$ (Lemma \ref{Lemma R coordinate for neg line bldes} taking $\varepsilon=1$);

 \item the Floer equations on $E\setminus 0_E$ and on $L\setminus 0_L$ agree for this choice of data.
\end{enumerate}

Away from the zero section we chose $J=J_B \oplus i$ (in the horizontal/vertical splitting of $TM$ determined by the Hermitian connection on $E$). Then (1) is immediately satisfied, in particular $J$ restricts to multiplication by $i$ on the tangent spaces of the fibres of $L$.

The linearity condition $h(R^L)=\textrm{constant}\cdot R^L$ in (2) is equivalent to requiring that $H$ has the following form at infinity,
$$
 H = \textrm{constant} \cdot r^2.
$$

Finally, to prove (3) we just need to check that the Hamiltonian vector fields $X_H^E = X_h^L$ agree on $E\setminus 0_E \equiv L\setminus 0_L$. Observe that if $X$ is a vertical vector for $L\to \mathbb{P}(E)$ (away from the zero section) then it is also a vertical vector for $E$, and therefore $\omega_L(\cdot,X)$, $\omega(\cdot,X)$ both equal $\Omega(\cdot,X)$ (recall we identified $\Omega^L\equiv \Omega$ away from the zero sections). We know $X_h^L(w)=\textrm{constant}\cdot i w$ explicitly, since the fibre of $L$ is a standard copy of $\C$ (using a Hermitian frame $w$), in particular $X_h^L$ is vertical. We conclude that $X_H^E=X_h^L$ agree, as required.

The above also shows that the flow of $X=X_H^E$ is the natural rotation in the fibre determined by the Hermitian metric. Indeed, the analysis of the $1$-periodic orbits of $X$ reduces to the analysis of $1$-periodic orbits of $X_h^L$ carried out in Section \ref{Subection Hamiltonians neg line bdle} applied to the negative line bundle $L \to \mathbb{P}(E)$.

\subsection{Calculation of $r_g(1)=\pi_M^*c_{\mathrm{top}}(E)$ for negative vector bundles}

\begin{lemma}
 $I(\widetilde{g}) = \mathrm{rank}_{\C} E$ for $g_t=e^{2\pi i t}$ acting by rotation in the fibres of $E$, lifted canonically to the $\widetilde{g}$ which fixes constants on the zero section.
\end{lemma}
\begin{proof}
This is proved as in Lemma \ref{Lemma calculation of I(g) neg line bdle}: using a local unitary frame for $E_b$, $\ell(t)$ is the identity on the $T_b B$ factor and rotation by $e^{2\pi it}$ of the fibre factor $E_b\cong \C^{\mathrm{rank}\,E}$. So $t\mapsto \det (I \oplus e^{2\pi i t}I)= e^{2\pi it \cdot \mathrm{rank}_{\C}\,E}$ is $\mathrm{rank_{\C}\,E}$ in $H_1(S^1;\Z)\cong \Z$.
\end{proof}
\begin{theorem}\label{Theorem v bdles consts are regular in general and rg1 is ctop}
 For $M=\mathrm{Tot}(\pi_M:E\to B)$ (satisfying weak$^+$ monotonicity), the analogue of Theorem \ref{Theorem consts are regular in general and rg1 is c1} holds using $\pi_M^*c_{\mathrm{rank}_{\C}\,E}(E)$.
\end{theorem}
\begin{proof}
The dimension of the moduli space (using Definition \ref{Definition c1 of Sg in terms of I})
$$
\dim_{\C} \mathcal{S}(j,\hat{J},\gamma+S_{\widetilde{g}}) = \dim_{\C} M - \mathrm{rank}_{\C}\,E + c_1(TM,\omega)(\gamma) 
= b + c.
$$
$M$ is weak, so $c\geq 0$ or $c\leq 1-\dim_{\C}M=1-b-r$ (let $r=\mathrm{rank}_{\C}E$). Since we need to sweep an lf cycle, we may assume $r\leq b+c \leq b+r-1$ (since $H_*^{lf}(M)\cong H^{2b+2r-*}(B)$, and it cannot sweep $[M]$ by the maximum principle). So $c\geq 0$. But the only lf cycles of degree $\geq 2b$ supported near the zero section are multiples of $[B]$ ($H_k(B)=0$ for $k>2b$). 
So $c=0$. The rest is as in the proof of Theorem \ref{Theorem consts are regular in general and rg1 is c1}.

Proving constants are regular: for the constant section $u(z)=(z,y)$, $(u^*T^v E_g)_z = T_y B \oplus \C^{r}$ with transition $(\mathrm{id},g_t^{\oplus r})$ over the equator of $S^2$. Therefore
$$
\coker \overline{\partial} = H^1(\P^1,\mathcal{O}(\underline{\C}^{\oplus b})\oplus \mathcal{O}(-1)^{\oplus r}) \cong  H^0(\P^1,\mathcal{O}(-2)^{\oplus b}\oplus \mathcal{O}(-1)^{\oplus r})^{\vee} = 0.\qedhere
$$
\end{proof}

%

\end{document}